\documentclass{amsart}

\usepackage{lscape}
\usepackage{hhline}
\usepackage{amsfonts}
\usepackage{amsthm}
\usepackage[psamsfonts]{amssymb}
\usepackage[dvips]{graphicx}
\usepackage{graphicx}
\usepackage{ascmac}
\usepackage{amsmath}
\usepackage{amsthm}
\usepackage{fancybox}
\usepackage{fancyhdr}
\usepackage{wrapfig}
\usepackage{verbatim}
\usepackage{enumerate}
\usepackage{calrsfs}
\usepackage{mathrsfs}
\usepackage{color}
\usepackage{multicol}
\usepackage{framed}
\usepackage{stmaryrd}
\usepackage{multirow,bigdelim}

\usepackage{amscd}
\usepackage{booktabs}
\usepackage{amsthm}
\usepackage{pifont}

\newtheorem{thm}{Theorem}

\newtheorem{prop}[thm]{Proposition}

\newtheorem{lem}[thm]{Lemma}

\newcommand{\Proof}{\hspace{-4mm}{\it Proof}. \ }
\newcommand{\QED}{{\unskip\nobreak\hfil
\penalty50\quad\null\nobreak\hfil
{$\Box$}\parfillskip0pt\finalhyphendemerits0\par\medskip}}

\renewcommand{\a}{\alpha}
\renewcommand{\b}{\beta}

\renewcommand{\d}{\delta}
\newcommand{\e}{\epsilon}

\renewcommand{\l}{\lambda}

\newcommand{\s}{\sigma}

\renewcommand{\t}{\tau}

\newcommand{\ga}{{\mathfrak{a}}}

\newcommand{\gf}{{\mathfrak{f}}}

\newcommand{\gn}{{\mathfrak{n}}}
\newcommand{\go}{{\mathfrak{o}}}
\newcommand{\gp}{{\mathfrak{p}}}
\newcommand{\gq}{{\mathfrak{q}}}

\newcommand{\gA}{{\mathfrak{A}}}
\newcommand{\gC}{{\mathfrak{C}}}
\newcommand{\gD}{{\mathfrak{D}}}

\newcommand{\gS}{{\mathfrak{S}}}

\newcommand{\gZ}{{\mathfrak{Z}}}

\newcommand{\Acal}{{\mathcal A}}
\newcommand{\Bcal}{{\mathcal B}}

\newcommand{\Fcal}{{\mathcal F}}
\newcommand{\Gcal}{{\mathcal G}}
\newcommand{\Hcal}{{\mathcal H}}

\newcommand{\Ocal}{{\mathcal O}}

\newcommand{\Scal}{{\mathcal S}}

\renewcommand{\AA}{\mathbb{A}}

\newcommand{\CC}{\mathbb{C}}
\newcommand{\DD}{\mathbb{D}}

\newcommand{\II}{\mathbb{I}}
\newcommand{\JJ}{\mathbb{J}}
\newcommand{\LL}{\mathbb{L}}
\newcommand{\NN}{\mathbb{N}}
\newcommand{\PP}{\mathbb{P}}
\newcommand{\QQ}{\mathbb{Q}}
\newcommand{\RR}{\mathbb{R}}
\newcommand{\TT}{\mathbb{T}}
\newcommand{\ZZ}{\mathbb{Z}}

\newcommand{\bfc}{{\mathbf c}}

\newcommand{\bfs}{{\mathbf s}}

\newcommand{\bfK}{{\mathbf K}}

\newcommand{\ord}{\operatorname{ord}}

\newcommand{\fin}{{\rm fin}}
\renewcommand{\Re}{\operatorname{Re}}
\renewcommand{\Im}{\operatorname{Im}}

\newcommand{\Res}{\operatorname{Res}}
\newcommand{\vol}{\operatorname{vol}}
\newcommand{\N}{\operatorname{N}}

\title{Asymptotic behaviors of means of central values of automorphic $L$-functions for GL(2)}
\author{Shingo Sugiyama}
\address{Department of Mathematics, Graduate School of Science, 
Osaka University, Toyonaka, Osaka 560-0043, Japan}
\email{s-sugiyama@cr.math.sci.osaka-u.ac.jp}
\date{}

\begin{document}

\begin{abstract}
Let $\AA$ be the adele ring of a totally real algebraic number field $F$.
We push forward an explicit computation of a relative trace formula
for periods of automorphic forms along a split torus in $GL(2)$
from a square free level case done by Masao Tsuzuki \cite{TsuzukiSpec}, to an arbitrary level case.
By using a relative trace formula, we study
central values of automorphic $L$-functions for cuspidal automorphic representations of $GL(2, \AA)$ corresponding to Maass forms with arbitrary level.
\end{abstract}
\maketitle
\pagestyle{myheadings}
\markboth{\hfill Sugiyama\hfill}{\hfill Means of central $L$-values\hfill }
\section*{Introduction}

Let $k \ge 4$ be an even integer.
For a prime $N$,
let $S_{k}^{\rm new}(N)$ be the space of all cuspidal new forms on the Poincar$\acute{\rm e}$ upper half plane of weight $k$ for $\Gamma_{0}(N)$.
The space $S_{k}^{\rm new}(N)$ has an orthogonal basis ${\Fcal^{\rm new}_{k}(N)}$ consisting of normalized Hecke eigen forms.
For $\varphi \in S_{k}^{\rm new}(N)$, we denote by $L(s, \varphi)$ the completed automorphic $L$-function for $\varphi$ satisfying a functional equation that relates
the values at $s$ and $1-s$.
Let $\eta$ be a quadratic Dirichlet character of conductor $D$ with $\eta(-1) = -1$.
The Dirichlet $L$-series associated to $\eta$ is denoted by $L_{\fin}(s, \eta)$.
For a fixed prime $p\nmid D$, $J_{p, S}$ denotes the set of all primes $N$ satisfying
both $\gcd(p, N) = \gcd(D, N) = 1$ and $\eta(N)=-1$.
Let $a_{p}(\varphi)$ denote the $p$-th Fourier coefficient of $\varphi$ multiplied by $p^{-(k-1)/2}$. 
Ramakrishnan and Rogawski \cite{Ramakrishnan-Rogawski} studied a
sum of central values of $L(s, \varphi)$ and proved the following theorem.
\begin{thm}
\cite[Theorem A]{Ramakrishnan-Rogawski}

For any interval $J \subset [-2, 2]$, we have
$$\lim_{\begin{subarray}{c}
N \rightarrow \infty \\
N \in J_{p, \eta}\end{subarray}}
\sum_{\begin{subarray}{c} \varphi \in \Fcal^{\rm new}_{k}(N),\\ a_{p}(\varphi) \in J\end{subarray}}
\frac{L(1/2, \varphi)L(1/2, \varphi \otimes \eta)}{||\varphi||^{2}} = 2^{k-1}\frac{\{(k/2-1)!\}^{2}}{\pi (k-2)!}L_{\fin}(1, \eta) \mu_{p}^{\eta}(J),$$
where
$||\varphi||$ denotes the Petersson norm of $\varphi$
and $\mu_{p}^{\eta}$ denotes the probability measure on $[-2, 2]$ defined by
$$\mu_{p}^{\eta}(x)=
\left\{ \begin{array}{l}
\displaystyle \frac{p-1}{(p^{1/2}+p^{-1/2}-x)^{2}}\ \mu_{ST}(x) \hspace{6mm} (\eta(p)=1),\\
\ \vspace{-2mm}\\
\displaystyle \frac{p+1}{(p^{1/2}+p^{-1/2})^{2}-x^{2}}\ \mu_{ST}(x) \hspace{5mm} (\eta(p)=-1).
\end{array}\right.
$$
Here, $\mu_{ST}(x)$ is the Sato-Tate measure $(2 \pi)^{-1} \sqrt{4-x^{2}}dx$.
\end{thm}

Feigon and Whitehouse \cite{Feigon-White} generalized this result to the case of Hilbert modular forms.
Tsuzuki \cite{TsuzukiSpec} gave a similar kind of asymptotic formula for Maass cusp forms
with square free level in terms of automorphic representations of $GL(2)$ over a totally real algebraic number field.
The purpose of this paper is to generalize Tsuzuki's results of \cite{TsuzukiSpec} to the case of arbitrary level.

To state our results in this paper, we prepare some notations.
Let $F$ be a totally real algebraic number field, ${\go}_{F}$ its integer ring and
$\AA$ the adele ring of $F$.
We denote by $\Sigma_{F}$ (resp. $\Sigma_{\infty}$ and $\Sigma_{\fin}$) the set of all places (resp. all infinite places and all finite places)
of $F$.
For each $v \in \Sigma_{\fin}$, we fix a uniformizer
$\varpi_{v}$ of $F_{v}$ and
denote by $q_{v}$ the cardinality of the residue field of $F_{v}$.
For an ideal ${\ga}$ of ${\go}_{F}$, let $S({\ga})$ denote the
set of all $v \in \Sigma_{\fin}$ such that $\ord_{v}(\ga) \ge 1$.
The absolute norm of $\ga$ is denoted by $\N(\ga)$.

Fix a quadratic character $\eta = \prod_{v \in \Sigma_{F}} \eta_{v}$ of $F^{\times} \backslash \AA^{\times}$ of conductor $\gf_{\eta}$
so that $\eta_{v}$ is trivial for any $v \in \Sigma_{\infty}$.
Fix a finite subset $S$ of $\Sigma_{F}$ such that $\Sigma_{\infty} \subset S$ and $S \cap S(\gf_{\eta}) = \emptyset$.
Let $J_{S, \eta}$ be the set of all ideals $\gn$ of $\go_{F}$ satisfying the following three conditions:
\begin{enumerate}
\item $S(\gn) \cap S(\gf_{\eta}) = \emptyset$ and $S(\gn) \cap S = \emptyset$,
\item $\eta_{v}(\varpi_{v}) = -1$ for any $v \in S(\gn)$,
\item $\tilde{\eta}(\gn) = \prod_{v \in \Sigma_{\fin}} \eta_{v}(\varpi_{v}^{\ord_{v}(\gn)}) = 1$.
\end{enumerate}

Let $\bfK_{\infty}$ be the standard maximal compact subgroup of $GL(2, F \otimes_{\QQ} \RR)$.
For an ideal $\gn$ of $\go_{F}$,
$\bfK_{0}(\gn)$ denotes the Hecke congruence subgroup of $GL(2, \AA_{\fin})$ of level $\gn$.
Let $\Pi_{\rm cus}(\gn)$ denote the set of all irreducible cuspidal automorphic
representations of $PGL(2, \AA)$ having nonzero $\bfK_{\infty}\bfK_{0}(\gn)$-invariant vectors.
Let $\Pi_{\rm cus}^{*}(\gn)$ be the set of all $\pi \in \Pi_{\rm cus}(\gn)$ with conductor $\gf_{\pi} = \gn$.
The standard $L$-function of $\pi$ is denoted by $L(s, \pi)$.
Let $S_{\pi}$ denote the set of all $v \in \Sigma_{\fin}$ satisfying $\ord_{v}(\gf_{\pi})\ge 2$.

For $\gn \in J_{S, \eta}$ and $\pi = \otimes_{v} \pi_{v} \in \Pi_{\rm cus}(\gn)$,  $\pi_{v}$ is isomorphic to a unitarizable
spherical principal series representation $\pi(|\cdot|_{v}^{\nu_{v}/2}, |\cdot|_{v}^{-\nu_{v}/2})$ of $GL(2, F_{v})$ for any $v\in S$.
The spectral parameter $\nu_{\pi, S}$ at $S$ of $\pi$ is defined as $\nu_{\pi, S} = (\nu_{v})_{v \in S}$. 
It is known that $\nu_{\pi, S} \in \mathfrak{X}_{S}^{0+} = \prod_{v \in S}\mathfrak{X}_{v}^{0+}$, where
$\mathfrak{X}_{v}^{0+} = i \RR_{\ge0} \cup (0, 1)$ for
$v \in \Sigma_{\infty}$ and
$\mathfrak{X}_{v}^{0+} = i[0, 2\pi (\log q_{v})^{-1}] \cup \{ x + i y \ | \ x \in (0, 1), \ y \in \{0, 2 \pi i(\log q_{v})^{-1}\} \ \}$ for $v \in S_{\fin}=S \cap \Sigma_{\fin}$, respectively.


We define a measure $\l_{S}^{\eta}$ on $\mathfrak{X}_{S}^{0+}$ by
$4 D_{F}^{3/2} L(1, \eta) \bigotimes_{v \in S} \l_{v}^{\eta_{v}}$,
where $D_{F}$ denotes the absolute value of the discriminant of $F$,
and for any $v \in S$, the measure $\l_{v}^{\eta_{v}}$ on $\mathfrak{X}_{v}^{0+}$ with support in $i \RR$ is given by
{\allowdisplaybreaks\begin{align*}
d\l_{v}^{\eta_{v}} (i y)= & \frac{L(1/2, \pi(|\cdot|_{v}^{i y/2}, |\cdot|_{v}^{-iy/2}))
L(1/2, \pi(|\cdot|_{v}^{iy/2}, |\cdot|_{v}^{-iy/2})\otimes \eta_{v})}{L(1, \eta_{v})}
 \left\{ \begin{array}{l}
\displaystyle \frac{1}{4 \pi}|\Gamma(iy/2)|^{-2}dy \hspace{5mm} (v \in \Sigma_{\infty}),\\
\ \vspace{-2mm}\\
\displaystyle \frac{\log q_{v}} {4\pi} |1-q_{v}^{-i y}|^{2}dy \hspace{2mm} (v \in S_{\fin}).
\end{array}\right.
\end{align*}
}We remark that when $F = \QQ$ and $v =p < \infty$, $\l_{v}^{\eta_{v}}(i y)$ is exactly equal to $\mu_{p}^{\eta}(x)$
by changing a variable $x = p^{i y/2} + p^{- iy/2}$.
The main theorem of this paper is stated as follows.
\begin{thm}\label{thm:asymptotic of central values}
Let $\Lambda$ be an infinite subset of $J_{S, \eta}$.
For any $f \in C_{c}(\mathfrak{X}_{S}^{0+})$, we have
\begin{align*}
\frac{1}{{\rm N}(\gn)}\sum_{ \pi \in \Pi_{\rm cus}^{*}(\gn)}\frac{L(1/2, \pi)L(1/2, \pi \otimes \eta)}{ L^{S_{\pi}}(1, \pi; {\rm Ad})} f(\nu_{\pi, S})
\sim
C(\gn)\langle \l_{S}^{\eta}, f \rangle
\end{align*}
with $$C(\gn) = \prod_{v \in S_{2}(\gn)}\{1-(q_{v}^{2}-q_{v})^{-1}\}
\prod_{v \in S(\gn)- (S_{1}(\gn)\cup S_{2}(\gn))}(1-q_{v}^{-2})$$
as ${\rm N}(\gn) \rightarrow \infty$ in $\gn \in \Lambda$.
Here $S_{1}(\gn)$ $($resp. $S_{2}(\gn)$$)$ denotes the set of all $v\in S(\gn)$ such that $\ord_{v}(\gn)=1$ $($resp. $\ord_{v}(\gn)=2$$)$.
\end{thm}

This asymptotic formula gives the following counterpart of \cite[Corollary B]{Ramakrishnan-Rogawski}.
\begin{thm}\label{thm:nonvanishing}
Let $\Lambda$ be an infinite subset of $J_{S, \eta}$
and let $\{ J_{v} \}_{v \in S}$ be a family of intervals such that $J_{v}$ is contained in $[1/4, \infty)$ for any $v \in \Sigma_{\infty}$ and in $[-2, 2]$ for any $v \in S_{\fin}$. Then, for any $M > 0$, there exists an irreducible cuspidal automorphic representation $\pi$ of $GL(2, \AA)$ with trivial central character satisfying the following conditions:
\begin{enumerate}
\item The conductor $\gf_{\pi}$ of $\pi$ belongs to $\Lambda$ and ${\rm N}(\gf_{\pi}) > M$ holds.
\item Both $L(1/2, \pi) \neq 0$ and $L(1/2, \pi \otimes \eta) \neq 0$ hold.
\item The spectral parameter $\nu_{\pi, S} = (\nu_{v})_{v \in S}$ at S of $\pi$ satisfies $(1 - \nu_{v}^{2})/4 \in J_{v}$ for any $v \in \Sigma_{\infty}$ and $q_{v}^{- \nu_{v}/2} + q_{v}^{\nu_{v}/2} \in J_{v}$ for any $v \in S_{\fin}$.
\end{enumerate}
\end{thm}
We remark that $L(1/2, \pi)L(1/2, \pi \otimes \eta) > 0$ if $L(1/2, \pi)L(1/2, \pi \otimes \eta) \neq 0$ by Guo's result \cite{Guo}.
Let $\{v_{j}\}_{j \in \NN}$ be the set of all places $v \in \Sigma_{\fin}- (S \cup S(\gf_{\eta}))$ such that $\eta_{v}(\varpi_{v}) = -1$ and let
$\{ \gp_{j} \}_{j \in \NN}$ be the set of all prime ideals of $\go_{F}$ corresponding to $\{ v_{j} \}_{j \in \NN}$.
Here are some examples of $\Lambda$ in Theorems \ref{thm:asymptotic of central values} and \ref{thm:nonvanishing}:
\begin{enumerate}
\item $\Lambda = \{ \gn = \gp_{1} \cdots \gp_{2n} \ |\ n \in \NN \}$,

\item $\Lambda = \{ \gn = \gp_{1}^{2n} \ |\ n \in \NN \}$,

\item $\Lambda = \{ \gn = \gp_{n}^{2a} \ |\ n \in \NN \}$ for a fixed $a \in \NN$,

\item $\Lambda = \{ \gn = \gp_{1}^{an}\gp_{2}^{bn} \ |\ n \in \NN \}$ for fixed odd integers $a, b >0$.
\end{enumerate}
The case (1) was treated by Tsuzuki \cite[Theorem 1.1 and Corollary 1.2]{TsuzukiSpec}.

Motohashi \cite{Motohashi} studied the growth of the square mean of central values of automorphic $L$-functions attached to Maass forms with full level via Kuznetsov's trace formula.
Tsuzuki \cite[Theorem 1.3]{TsuzukiSpec} considered a similar growth in the case where the level is square free and the base field is totally real.
We give a generalization of \cite[Theorem 1.3]{TsuzukiSpec} to the case of arbitrary level.
\begin{thm}
\label{thm:asymptotic of trivial}
We set $d_{F}=[F: \QQ]$.
Let $\gn$ be an arbitrary ideal of $\go_{F}$
and $\eta : F^{\times} \backslash \AA^{\times} \rightarrow \{ \pm 1\}$ a character
of conductor relatively prime to $\gn$.
Assume that $\tilde{\eta}(\gn) =1$ and $\eta_{v}(-1)=1$ for any $v \in \Sigma_{\infty}$.
Let $J$ be a compact subset of $\mathfrak{X}_{\Sigma_{\infty}}^{0+} \cap \prod_{v \in \Sigma_{\infty}} i \RR_{>0}$ with smooth boundary.
Then, for any $\e >0$, we have
{\allowdisplaybreaks\begin{align*}
\sum_{\begin{subarray}{c}\pi \in \Pi_{\rm cus}(\gn), \\
\nu_{\pi, \Sigma_{\infty}}\in tJ \end{subarray}}w_{\gn}^{\eta}(\pi) \frac{L(1/2, \pi)L(1/2, \pi \otimes \eta)}
{{\rm N}(\gf_{\pi})L^{S_{\pi}}(1, \pi; {\rm Ad})}
=& \frac{4D_{F}^{3/2}}{(4 \pi )^{d_{F}}}(1 + \d_{\gn,\go_{F}}) \vol(J) t^{d_{F}}(d_{F}\Res_{s=1}L(s,\eta) \log t + {\bf C}^{\eta}(\gn, F)) \\
& + \Ocal(t^{d_{F}-1}(\log t)^{3}) + \Ocal(t^{d_{F}(1+4\theta)+\e}), \ t \rightarrow \infty,
\end{align*}
}where
$w_{\pi}^{\eta}(\gn)$ is a constant explicitly defined in Lemma \ref{lem:positivity of P},
$${\bf C}^{\eta}(\gn, F) = {\rm CT}_{s = 1}L(s, \eta) + \Res_{s = 1}L(s, \eta) \bigg\{\frac{d_{F}}{2}(C_{\rm Euler}+ 2\log 2 - \log \pi) + \log(D_{F}{\rm N}(\gn)^{1/2})\bigg\}$$
and
$\theta \in \RR$ is a constant such that
$$|L(1/2 + i t, \chi)| \ll \mathfrak{q}(\chi |\cdot|_{\AA}^{i t})^{1/4 + \theta}, \hspace{3mm}t \in \RR$$
holds uniformly for any character $\chi$ of $F^{\times} \backslash \AA^{\times}$.
Here $\mathfrak{q}(\chi |\cdot|_{\AA}^{i t})$ is the analytic conductor of $\chi |\cdot|_{\AA}^{i t}$.
\end{thm}
Moreover,
we obtain the following results on subconvexity bounds depending on $\theta<0$.

\begin{thm}\label{thm:subconvexity}
Let $\gn$ be an arbitrary ideal of $\go_{F}$.
Let $J \subset \mathfrak{X}_{\Sigma_{\infty}}^{0+} \cap \prod_{v \in \Sigma_{\infty}}i \RR_{>0}$ be a closed cone.
Then, for any $\e>0$, we have
$$|L_{\fin}(1/2, \pi)| \ll (1+ ||\nu_{\pi, \Sigma_{\infty}}||)^{d_{F} + \sup(2d_{F} \theta, -1/2) + \e}
$$
for
$\pi \in \Pi_{\rm cus}(\gn)_{J} = \{ \pi \in \Pi_{\rm cus}(\gn)\ | \ \nu_{\pi, \Sigma_{\infty}} \in J \}$.
\end{thm}
We remark that Theorem \ref{thm:subconvexity} was proved by Tsuzuki \cite[Corollary 1.4]{TsuzukiSpec} when $\gn$ is square free.
Michel and Venkatesh \cite{Michel-Venkatesh} gave sharp subconvexity bounds for automorphic $L$-functions for $GL(1)$ and $GL(2)$ in more general case.

Our method to prove from Theorem \ref{thm:asymptotic of central values} to Theorem \ref{thm:subconvexity} is based on that of \cite{TsuzukiSpec}.
We introduce adelic Green functions on $GL(2, \AA)$ for ideals of $\go_{F}$
and then we give a relative trace formula by computing the regularized period of a Poincar$\acute{\rm e}$ series of an adelic Green function in two different ways.
Regularized periods used in this paper are toral period integrals regularized
by Tsuzuki in order to define periods for nonrapidly decreasing functions on $GL(2, \AA)$.
Since regularized periods of automorphic forms
on $GL(2, \AA)$ with arbitrary level were studied by a previous paper \cite{Sugiyama},
we can compute the spectral side of our relative trace formula.

We explain the structure of this paper. In \S \ref{Preliminaries}, we introduce notations used throughout this paper.
In \S \ref{Regularized periods of automorphic forms}, we review results of \cite{Sugiyama} and prepare several lemmas.
In \S \ref{Adelic Green functions}, we introduce adelic Green functions on $GL(2, \AA)$.
In \S \ref{Spectral expansions of renormalized Green functions}, we regularize a Poincar$\acute{\rm e}$ series of an adelic Green function.
The regularized Poincar$\acute{\rm e}$ series is called the regularized automorphic smoothed kernel.
In \S \ref{Periods of regularized automorphic smoothed kernels: spectral side}, we compute the regularized period of the regularized automorphic smoothed kernel by using the spectral expansion.
In \S \ref{Periods of regularized automorphic smoothed kernels: geometric side}, we decompose the regularized period of the regularized automorphic smoothed kernel
into the sum of terms derived from some orbits.
In \S \ref{Proofs of main theorems}, we prove from Theorem \ref{thm:asymptotic of central values} to Theorem \ref{thm:subconvexity}.

\subsection*{Notation.}

We write $\NN$ for the set of natural numbers and put $\NN_{0}=\NN \cup\{0\}$.
For sets $A$ and $B$, the set ${\rm Map}(A, B)$ denotes the set of mappings from $A$ to $B$.
For $f, g \in {\rm Map}(A, \RR_{\ge 0})$, let us denote by $f(x)\ll g(x), x\in A$
an inequality $f(x)\le Cg(x)$ with some constant $C>0$.
For a given condition $P$, $\d(P) \in \{ 0, 1\}$ is defined by
$\d(P)=1$ (resp. $\d(P)=0$) if $P$ is true (resp. false).

Let $F$ be an totally real field with its degree $d_{F}$ and $\go_{F}$ its integer ring.
Let $\AA$ and $\AA_{\fin}$ be the adele ring and the finite adele ring of $F$, respectively.
The symbols $\Sigma_{\infty}$ and $\Sigma_{\fin}$ denote
the set of all infinite places and the set of all finite places of $F$, respectively.
For a place $v \in \Sigma_{F} = \Sigma_{\infty} \cup \Sigma_{\fin}$, let $|\cdot |_{v}$
denote the normalized valuation of the completion $F_{v}$ of $F$ at $v$.
For each $v \in \Sigma_{\fin}$, let $\varpi_{v}$ be a uniformizer of $F_{v}$.
Then, $\gp_{v} = \varpi_{v}\go_{v}$ for $v \in \Sigma_{\fin}$ is a maximal ideal of the integer ring $\go_{v}$ of $F_{v}$
and we have $|\varpi_{v}|_{v}=q_{v}^{-1}$, where $q_{v}$ is the cardinality of the residue field $\go_{v}/\gp_{v}$.
For an ideal $\ga$ of $\go_{F}$, let $S(\ga)$ denote the set of all $v \in \Sigma_{\fin}$
such that $v$ divides $\ga$. For any $k \in \NN$, we write $S_{k}(\ga)$ for
the set of all $v \in S(\ga)$ with $\ord_{v}(\ga) =k$, where $\ord_{v}(\ga)$ is the order of $\ga$
at $v$. Let $\N(\ga)$ denote the absolute norm of $\ga$.

Let $G$ be the algebraic group $GL(2)$ with unit element $e$.
For any $F$-algebraic subgroup $M$ of $G$, 
we set $M_{F}= M(F)$, $M_{v}=M(F_{v})$ (for $v\in \Sigma_{F}$), $M_{\AA} = M(\AA)$ and $M_{\fin}=M(\AA_{\fin})$, respectively.
The diagonal maximal split torus of $G$ is denoted by $H$.
Then, the Borel subgroup $B=HN$ of $G$ consists of all upper triangular matrices of $G$,
where $N$ is the subgroup of $G$ consisting of all unipotent matrices of $G$.
The center of $G$ is denoted by $Z$.
We put $\bfK_{v}=O(2, \RR)$ (resp. $\bfK_{v}=GL(2, \go_{v})$) for $v \in \Sigma_{\infty}$
(resp. $v \in \Sigma_{\fin}$).
Then, $\bfK = \prod_{v\in\Sigma_{F}} \bfK_{v}$ is a maximal compact subgroup of $G_{\AA}$.
Set ${\bf K}_{\infty} = \prod_{v \in \Sigma_{\infty}}{\bf K}_{v}$ and
$\bfK_{0}({\gp}_{v}^{n})= 
\left\{ \left(\begin{matrix}
a & b \\
c & d
\end{matrix}\right) \in {\bf K}_{v} \bigg| c \equiv 0 ({\rm mod} \ {\gp}_{v}^{n}) \right\}$ for any 
$n\in \NN_{0}$.
For an ideal $\ga$ of ${\go}_{F}$, we put
${\bf K}_{0}(\ga) = \prod_{v \in \Sigma_{\fin}}{\bf K}_{0}(\ga\go_{v})$.

\section{Preliminaries}
\label{Preliminaries}
Let $\AA_{\QQ}$ be the adele ring of $\QQ$ and
$\psi_{\QQ} = \prod_{p}\psi_{p}$ the additive character of $\QQ \backslash \AA_{\QQ}$ with archimedean component $\psi_{\infty}(x)={\exp}(2\pi ix)$ for $x \in \RR$. 
Then, $\psi_{F} = \psi \circ {\rm tr}_{F/\mathbb{Q}} = \prod_{v \in \Sigma_{F}}\psi_{F_{v}}$ is a nontrivial additive
character of $F \backslash \AA$.
Let $\gD_{F/\QQ}$ be the global different of $F/\QQ$ and set $\ord_{v}\gD_{F/\QQ} = d_{v}$ for any $v \in \Sigma_{\fin}$.

For $v \in \Sigma_{F}$, let $dx_{v}$ be the self-dual Haar measure of $F_{v}$ with
respect to $\psi_{F_{v}}$.
We set
$d^{\times}x_{v} = (1-q_{v}^{-1})^{-1}dx_{v}/|x_{v}|_{v}$ for $v \in \Sigma_{\fin}$
and $d^{\times}x_{v} = d^{\times}x_{v}/|x_{v}|_{v}$ for $v \in \Sigma_{\infty}$, respectively.
Then, $d^{\times} x_{v}$ is a Haar measure of $F_{v}^{\times}$ and
the product measure $d^{\times}x=\prod_{v \in \Sigma_{F}}d^{\times}x_{v}$ gives a Haar measure on $\AA^{\times}$.
For each $v \in \Sigma_{F}$, we take a Haar measure $dk_{v}$ on $\bfK_{v}$ such that total volume is one,
and take a Haar measure $dg_{v}$ on $G_{v}$ in the following way.
Let $dh_{v}$ (resp. $dn_{v}$) denotes the Haar measure on $H_{v}$ (resp. $N_{v}$)
induced via the isomorphism $H_{v} \cong F_{v}^{\times} \times F_{v}^{\times}$
(resp. $N_{v} \cong F_{v}$).
Then, $dg_{v} = dh_{v}dn_{v}dk_{v}$ gives a Haar measure on $G_{v}$ via the Iwasawa decomposition $g_{v}=h_{v}n_{v}k_{v} \in H_{v}N_{v}\bfK_{v}$.
We remark $\vol(\bfK_{v}, dg_{v}) = q_{v}^{-3d_{v}/2}$ for any $v \in \Sigma_{\fin}$.
We denote the Haar measure $\prod_{v \in \Sigma_{F}}dk_{v}$ of $\bfK$ by $dk$.

Let $|\cdot|_{\AA}=\prod_{v\in\Sigma_{F}}|\cdot|_{v}$ be the idele norm of $\AA^{\times}$
and set $\AA^{1} = \{x\in\AA^{\times}| \ |x|_{\AA} = 1\}$.
For $y \in \RR_{>0}$, $\underline{y}$ denotes the idele
such that the $v$-component of $\underline{y}$ satisfies
$\underline{y}_{v}=y^{1/d_{F}}$ (resp. $\underline{y}_{v}=1$) for $v \in \Sigma_{\infty}$
(resp. $v \in \Sigma_{\fin}$).
Set $G_{\AA}^{1} = \{g\in G_{\AA}| \ |\det g|_{\AA} = 1\}$ and $\gA = 
\left\{\left(\begin{matrix}\underline{y} & 0 \\ 0 &  \underline{y}\end{matrix}\right)\bigg| \ y>0\right\}$.
Then,  $G_{\AA} = \gA G_{\AA}^{1}$ holds.

For $v \in \Sigma_{\fin}$ and a quasi-character $\chi_{v}$ of
$F_{v}^{\times}$, $\gp_{v}^{f(\chi_{v})}$ denotes the conductor of $\chi_{v}$.
We define the Gauss sum associated with $\chi_{v}$ by
$${\mathcal G}(\chi_{v}) = \int_{{\go}^{\times}_{v}}
\chi_{v}(u\varpi_{v}^{-d_{v}-f(\chi_{v})})
\psi_{F_{v}}(u \varpi_{v}^{-d_{v}-f(\chi_{v})}) d^{\times}u.$$
For any quasi-character $\chi=\prod_{v\in\Sigma_{F}}\chi_{v}$ of $F^{\times} \backslash \AA^{\times}$, we define the conductor of $\chi$ by the ideal ${\gf}_{\chi}$ of ${\go}_{F}$ such that
${\gf}_{\chi}{\go}_{v} = {\gp}_{v}^{f(\chi_{v})}$ for all $v\in\Sigma_{\fin}$. We write $\chi_{\fin}$ for $\prod_{v\in\Sigma_{\fin}}\chi_{v}$.
The Gauss sum associated with $\chi$ is defined by
the product of $\Gcal(\chi_{v})$ over all $v \in \Sigma_{\fin}$.
We set $\tilde{\chi}(\ga) = \prod_{v \in \Sigma_{\fin}}\chi_{v}(\varpi_{v}^{\ord_{v}(\ga)})$ for any ideal $\ga$ of $\go_{F}$.
For $v \in \Sigma_{F}$, we denote the trivial character of $F_{v}^{\times}$ by ${\bf 1}_{v}$,
and the trivial character of $\AA^{\times}$ by ${\bf 1}$.
Throughout this paper, any quasi-character $\chi$ of $F^{\times} \backslash \AA^{\times}$ is assumed to satisfy $\chi(\underline{y}) = 1$ for all $y \in \RR_{>0}$.
Such a quasi-character is a character.
For any $v \in \Sigma_{F}$ and any character $\chi_{v}$ of $F_{v}^{\times}$, let $b(\chi_{v})$ denote $b_{v} \in \RR$ (resp. $b_{v} \in [0, 2\pi(\log q_{v})^{-1})$) such that
the restriction of $\chi_{v}$ to $\RR_{>0}$ (resp. $\varpi_{v}^{\ZZ}$) is of the form $|\cdot|_{v}^{i b_{v}}$.
For any character $\chi$ of $F^{\times} \backslash \AA^{\times}$, the analytic conductor $\gq(\chi)$ of $\chi$ is defined to be
$$\gq(\chi) = \big\{\prod_{v \in \Sigma_{\infty}} (3 + |b(\chi_{v})|)\big\}{\rm N}(\gf_{\chi}).$$

Let $\gn$ be an ideal of $\go_{F}$.
For an ideal $\mathfrak{c}$ of $\go_{F}$, let $\Xi_{0}(\mathfrak{c})$ be the set of all characters $\chi$ of $F^{\times} \backslash \AA^{\times}$ such that $\gf_{\chi}=\mathfrak{c}$ and
$\chi_{v}(-1)=1$ for all $v \in \Sigma_{\infty}$.
We write $\Xi(\gn)$ for $\bigcup_{ \mathfrak{c}^{2}| \gn} \Xi_{0}(\mathfrak{c})$.
Let $U_{F}^{+}$ be the set of all totally positive units of $\go_{F}$
and set
$$\log U_{F}^{+} = \{ (\log u_{v})_{v \in \Sigma_{\infty}} \ | \ (u_{v})_{v \in \Sigma_{\infty}} \in U_{F}^{+}\}.$$
Then, $\log U_{F}^{+}$ is a lattice of $\ZZ$-rank $d_{F}-1$ in $V$, where
$V = \{ (x_{v}) \in \RR^{d_{F}} \ | \ \sum_{v \in \Sigma_{\infty}}x_{v} =0 \}.$
Set
$$L_{0} = \{ (b_{v})_{v \in \Sigma_{\infty}} \in V \ | \ \sum_{v \in \Sigma_{\infty}} b_{v} l_{v} \in \ZZ \ for \ all \ (l_{v})_{v \in \Sigma_{\infty}} \in \log U_{F}^{+} \}.$$
Then, $L_{0}$ is also a $\ZZ$-lattice in $V$.
Let $\chi$ be a character of $F^{\times} \backslash \AA^{\times}$. Since $\chi(\underline{y}) =1$ for any $y\in \RR_{>0}$, we have
$\sum_{v \in \Sigma_{\infty}} b(\chi_{v}) = 0$.
Thus, if we denote by $b(\chi)$ the element $(b(\chi_{v}))_{v \in {\Sigma_{\infty}}}$ of $\RR^{d_{F}}$,
then $b(\chi) \in L_{0}$ holds.
Therefore the mapping $\chi \mapsto b(\chi)$ is a surjection from
$\Xi(\gn)$ onto $L_{0}$ and the kernel $\Xi_{\ker}(\gn)$ of this mapping is a finite abelian group.

\begin{lem}\label{esti of X(n)}
Let $X(\gn)$ be the order of  $\Xi_{\ker}(\gn)$.
Then, for any $\e>0$, the estimate
$$X(\gn) \ll {\rm N}(\gn)^{1/2+\e}$$
holds with the implied constant independent of $\gn$.
\end{lem}
\Proof
For any ideal $\ga$ of $\go_{F}$, we set
$ I_{F}({\ga})= \prod_{v \in \Sigma_{\infty}} \RR^{\times} \times \prod_{v \in \Sigma_{\fin}}(1+\ga\go_{v})$. Then,
the ray class group $C_{F}({\ga})$ modulo $\ga$ is defined by
$C_{F}({\ga}) = F^{\times} \backslash F^{\times} I_{F}({\ga})$.
For any fixed $\mathfrak{c}$ satisfying $\mathfrak{c}^{2}|\gn$, the group $\Xi_{0}(\mathfrak{c}) \cap \Xi_{\ker}(\gn)$ is equal to the set of all characters of $F^{\times} \backslash \AA^{\times}$ of finite order contained in $\Xi_{0}(\mathfrak{c})$.
Hence
$$\# (\Xi_{0}(\mathfrak{c}) \cap \Xi_{\ker}(\gn) )\le
\# (F^{\times} \backslash \AA^{\times} / I_{F}(\mathfrak{c}))
= h_{F} \#( C_{F}({\go_{F}})/C_{F}({\mathfrak{c}}))
\le h_{F} {\rm N}(\mathfrak{c})\le h_{F} {\rm N}(\gn)^{1/2}$$
holds, where $h_{F}$ is the class number of $F$.
Noting $\sum_{\mathfrak{c}^{2} | \gn}1 \ll \log(1+{\rm N}(\gn)) \ll {\rm N}(\gn)^{\e}$ for any $\e > 0$, we obtain the assertion.
\QED


\section{Regularized periods of automorphic forms}
\label{Regularized periods of automorphic forms}
In this section, we recall an explicit formula in \cite{Sugiyama} of the regularized periods of automorphic forms on $G_{\AA}$.
\subsection{Zeta integrals of cusp forms}
\label{Zeta integrals of cusp forms}

Let $\pi$ be a $\bfK_{\infty}$-spherical irreducible cuspidal automorphic representation of $G_{\AA}$
with trivial central character, where the representation space $V_{\pi}$ is realized in the space of cusp forms.
For any quasi-character $\eta$ of $F^{\times} \backslash \AA^{\times}$ and $\varphi \in V_{\pi}$, we define
the global zeta integral by
$$Z(s,\eta,\varphi) = \int_{F^{\times} \backslash \mathbb{A}^{\times}}\varphi
\left(
\begin{matrix}
t & 0 \\
0 & 1
\end{matrix}
\right) \eta(t)
|t|_{\mathbb{A}}^{s-1/2}d^{\times}t, \hspace{5mm} s \in \mathbb{C}.$$
The defining integral converges absolutely
for any $s \in \mathbb{C}$, and hence
$Z(s,\eta,\varphi)$ is an entire function in $s$.

We fix a family $\{ \pi_{v} \}_{v \in \Sigma_{F}}$ consisting of irreducible
admissible representations such that $\pi \cong \bigotimes_{v\in \Sigma_{F}}\pi_{v}$.
The conductor of ${\pi}$ is denoted by ${\gf}_{\pi}$,
which is the ideal of $\go_{F}$ defined by ${\gf}_{\pi}{\go}_{v} = {\gp}_{v}^{c(\pi_{v})}$ for
all $v \in \Sigma_{\fin}$, where $\gp_{v}^{c(\pi_{v})}$ is the conductor of $\pi_{v}$.
Let ${\gn}$ be an ideal of ${\go}_{F}$ which is divided by ${\gf}_{\pi}$.

Let $n$ be the maximal nonnegative integer $m$ such that $S_{m}({\gn}{\gf}_{\pi}^{-1})\neq \emptyset$.
For $\rho = (\rho_{k})_{1\le k \le n} \in \Lambda_{\pi}^{0}(\gn) = \prod_{k=1}^{n}{\rm Map}\left(S_{k}({\gn}{\gf}_{\pi}^{-1}),
\{0, \ldots, k\}\right)$,
let $\varphi_{\pi, \rho}$ denote the cusp
form in $V_{\pi}^{{\bf K}_{\infty}{\bf K}_{0}(\gn)}$ corresponding to 
$$
\displaystyle \bigotimes_{v\in\Sigma_{\infty}} \phi_{0,v}\otimes \bigotimes_{v\in S_{1}({\gn}{\gf}_{\pi}^{-1})}\phi_{\rho_{1}(v),v}\otimes \cdots \otimes \bigotimes_{v\in
S_{n}({\gn}{\gf}_{\pi}^{-1})}\phi_{\rho_{n}(v),v}\otimes \bigotimes_{v\in\Sigma_{\fin}-S({\gn}{\gf}_{\pi}^{-1})}\phi_{0,v}
$$
by the isomorphism $V_{\pi}\cong \bigotimes_{v\in \Sigma_{F}} V_{\pi_{v}}$.
Here, $V_{\pi_{v}}$ denotes the Whittaker model of $\pi_{v}$ with respect to $\psi_{F_{v}}$,
$\phi_{0, v}$ is the spherical vector in $V_{\pi_{v}}$ for $v \in \Sigma_{\infty}$ given in
\cite[1.4]{Sugiyama},
and the function $\phi_{k, v}$ is the $\bfK_{0}(\gn \go_{v})$-invariant vector for $v \in \Sigma_{\fin}$, which is
constructed in \cite[\S 2 and \S 3]{Sugiyama}.
Then, the finite set $\{ \varphi_{\pi, \rho} \}_{\rho \in \Lambda_{\pi}^{0}(\gn)}$
is an orthogonal basis of $V_{\pi}^{{\bf K}_{\infty}{\bf K}_{0}(\gn)}$. Here
$V_{\pi} \subset L^{2}(Z_{\AA}G_{F}\backslash G_{\AA})$ is equipped with the $L^2$-inner product (cf. \cite[Proposition 17]{Sugiyama}).

We consider a character $\eta$ of $F^{\times} \backslash \AA^{\times}$ satisfying 
\begin{center}
($\star$)
$\left\{ \begin{array}{l}
\eta^{2}= {\bf 1}, \\
v \in \Sigma_{\infty} \Rightarrow \eta_{v} = {\bf 1}_{v}, \\
{\gf}_{\eta} {\rm \ is \ relatively \ prime \ to \ {\gn} \ and \ \tilde{\eta}(\gn)=1}.
\end{array}\right.$
\end{center}
\ \\
For such a character $\eta$ and $\varphi \in V_{\pi}^{{\bf K}_{\infty}{\bf K}_{0}(\gn)}$, we
define the modified global zeta integral by
$$Z^{*}(s,\eta,\varphi) = \eta_{\fin}(x_{\eta, {\rm
fin}})Z\left(s,\eta,\pi \left(\begin{matrix}1 & x_{\eta}\\ 0 & 1\end{matrix}\right)
\varphi\right), \hspace{5mm} s \in \CC.$$
Here $x_{\eta} = (x_{\eta,v})_{v \in \Sigma_{F}} \in \AA$ is the adele
whose $v$-component satisfies
$x_{\eta,v} = 0$ and $x_{\eta, v} = \varpi_{v}^{-f(\eta_{v})}$ for 
$v \in\Sigma_{\infty}$ and $v \in \Sigma_{\fin}$, respectively,
and $x_{\eta, \fin}$ denotes the projection of $x_{\eta}$ to
$\AA_{\fin}$.

\subsection{Regularized periods of cusp forms}
\label{Regularized periods of cusp forms}
We recall a definition of regularized periods of automorphic forms on $G_{\AA}$ defined in \cite[\S 7]{TsuzukiSpec}.
For $C>0$, let $\Bcal(C)$ be the space of all holomorphic even functions $\b$ on $\{z \in \CC
\ | \ |\Re(z)|<C\}$ satisfying
$|\b(\s + it)|\ll
(1+|t|)^{-l}, \ \s \in [a,b], t \in \RR$ hold for any $[a,b]\subset (-C,C)$ and any $l>0$.
Let $\Bcal$ be the space of all entire functions $\b$ on $\CC$ such that the restriction of $\b$ to $\{z \in \CC \ | \ |\Re(z)|<C\}$ is contained in $\Bcal(C)$ for any $C>0$.

For $\b \in \Bcal$ and $\lambda \in \CC$,
we define a function $\hat\b_{\l}$ on $\RR_{>0}$ by
$${\hat{\b}}_{\lambda}(t) = \frac{1}{2\pi i}\int_{L_{\sigma}}\frac{\b(z)}{z +
\lambda}t^{z}dz,\hspace{3mm}(\sigma > -{\Re}(\lambda)),
$$
where $L_{\s} = \{z\in\CC |\Re(z)=\s \}$.

For $\b \in \Bcal$, $\lambda \in \CC$, a character $\eta$ of $F^{\times} \backslash \AA^{\times}$ satisfying ($\star$) and a function $\varphi:\gA G_{F}\backslash
G_{\AA}\rightarrow \CC$, we consider
$$P_{\b, \l}^{\eta}(\varphi) = \int_{F^{\times} \backslash \AA^{\times}}
\{{\hat{\b}}_{\l}(|t|_{\AA}) + {\hat{\b}}_{\l}(|t|_{\AA}^{-1})\}
\varphi\left(\left(\begin{matrix}t & 0 \\0 & 1\end{matrix}\right)
\left(\begin{matrix}1 & x_{\eta} \\0 & 1\end{matrix}\right)\right)
\eta(t)\eta_{\rm fin}(x_{\eta,\rm fin})d^{\times}t.
$$
Now we assume that
for any $\b \in \Bcal$, there exists a constant $C \in \RR$ such that if ${\rm Re}(\lambda)>C$ the integral $P_{\b,
\lambda}^{\eta}(\varphi)$ converges
and the function $\{z \in \CC \ | \ {\rm Re}(z)>C\}\ni \lambda \mapsto
P_{\beta,\lambda}^{\eta}(\varphi)$ is continued meromorphically to a neighborhood
of $\lambda$ = 0.
Then a constant $P_{\rm reg}^{\eta}(\varphi)$ is called {\it the
regularized $\eta$-period of $\varphi$} if
${\rm CT}_{\lambda = 0}P_{\beta,\lambda}^{\eta}(\varphi) = P_{\rm
reg}^{\eta}(\varphi)\b(0)$
for all $\b \in \Bcal$.
Then the following was proved in \cite{Sugiyama}.

\begin{prop}\label{prop:period = L(1/2)}
\cite[Main Theorem A]{Sugiyama}
For any $\rho = (\rho_{k})_{1 \le k \le n} \in \Lambda_{\pi}^{0}(\gn)$ and $\eta$ satisfying ($\star$), the period $P_{\rm reg}^{\eta}(\varphi_{\pi, \rho})$ can be defined and we have
$$P_{\rm reg}^{\eta}(\varphi_{\pi, \rho}) = Z^{*}(1/2, \eta, \varphi)=
{\Gcal}(\eta)
\{\prod_{k=1}^{n}\prod_{v \in S_{k}({\gn}{\gf}_{\pi}^{-1})}
Q_{\rho_{k}(v),v}^{\pi_{v}}(\eta_{v}, 1)\}L(1/2,\pi \otimes \eta),$$
where the constants $Q_{\rho_{k}(v),v}^{\pi_{v}}(\eta_{v}, 1)$ are given as follows: 
\begin{itemize}
\item If $c(\pi_{v}) = 0$ and $(\a_{v}, \a_{v}^{-1})$ is the Satake parameter of  $\pi_{v}$, then
{\allowdisplaybreaks\begin{align*}
Q_{k,v}^{\pi_{v}}(\eta_{v}, 1)
=&
\begin{cases}
1 & \text{$($if $k=0$$)$}, \\
\displaystyle \eta_{v}(\varpi_{v}) - \frac{\a_{v} + \a_{v}^{-1}}{q_{v}^{1/2}+q_{v}^{-1/2}} \vspace{2mm}& 
\text{$($if $k = 1$$)$}, \\
q_{v}^{-1}\eta_{v}(\varpi_{v})^{k-2}(\a_{v} q_{v}^{1/2}\eta_{v}(\varpi_{v}) - 1)(\a_{v}^{-1}q_{v}^{1/2}\eta_{v}(\varpi_{v}) - 1) & 
\text{$($if $k\ge 2$$)$}.
\end{cases}
\end{align*}
}

\item If $c(\pi_{v}) = 1$, then $\pi_{v}$ is isomorphic to a special representation $\s(\chi_{v}|\cdot|_{v}^{1/2}, \chi_{v}|\cdot|_{v}^{-1/2})$
for some unramified character $\chi_{v}$ of $F_{v}^{\times}$ and
$$Q_{k,v}^{\pi_{v}}(\eta_{v}, 1) =
\begin{cases}
1 & \text{$($if $k=0$$)$}, \\
\eta_{v}(\varpi_{v})^{k-1}(\eta_{v}(\varpi_{v})-q_{v}^{-1}\chi_{v}(\varpi_{v})^{-1}) & 
\text{$($if $k\ge 1$$)$}.
\end{cases}$$

\item If $c(\pi_{v}) \ge 2$, then $Q_{k, v}^{\pi_{v}}(\eta_{v}, 1) = \eta_{v}(\varpi_{v})^{k}$
for any $k\in \NN_{0}$.
\end{itemize}
\end{prop}

\subsection{Preliminaries for regularized periods of Eisenstein series}
\label{Preliminaries for regularized periods of Eisenstein series}
We fix a character $\chi = \prod_{v\in \Sigma_{F}}\chi_{v}$ of $F^{\times} \backslash \AA^{\times}$.
For $\nu \in \CC$, we denote by $I(\chi|\cdot|_{\AA}^{\nu/2})$ the space of all smooth $\CC$-valued right $\bfK$-finite functions
$f$ on $G_{\AA}$ with the $B_{\AA}$-equivariance
$$f\left(\left(
\begin{matrix}
a & b \\
0 & d
\end{matrix}
\right) g\right) = \chi(a/d)|a/d|_{\AA}^{(\nu + 1)/2}f(g)$$ for all $\left(\begin{matrix}a & b \\0 & d\end{matrix}\right) \in B_{\AA}$ and $g \in G_{\AA}$.
If $\nu \in i\RR$, then the space $I(\chi|\cdot|_{\AA}^{\nu/2})$ is unitarizable and a $G_{\AA}$-invariant hermitian inner product is given by
$$(f_{1}|f_{2}) = \int_{\bfK}f_{1}(k)\overline{f_{2}(k)}dk$$
for any $f_{1}, f_{2}\in I(\chi|\cdot|_{\AA}^{\nu/2})$.

For $\nu \in \CC$ and $f^{(\nu)} \in I(\chi|\cdot|_{\AA}^{\nu/2})$,
The family $\{f^{(\nu)}\}_{\nu \in \CC}$ is called a flat section if the restriction of $f^{(\nu)}$ to ${\bf K}$ is independent of $\nu \in \CC$.
We define the Eisenstein series for $f^{(\nu)} \in I(\chi|\cdot|_{\AA}^{\nu/2})$ by
$$E(f^{(\nu)}, g) = \sum_{\gamma \in B_{F}\backslash G_{F}}f^{(\nu)}(\gamma g),  g\in G_{\AA}.$$
The defining series converges absolutely if $\Re(\nu)>1$.
If $\{f^{(\nu)}\}_{\nu \in \CC}$ is a flat section,
then $E(f^{(\nu)}, g)$ is continued meromorphically to $\CC$ as a function in $\nu$.
We remark that the function $E(f^{(\nu)}, g)$ is holomorphic on $i\RR$.
On the half plane $\Re(\nu)>0$,
$E(f^{(\nu)}, g)$ is holomorphic except for $\nu=1$, and $\nu=1$ is a pole of $E(f^{(\nu)}, g)$
if and only if $\chi^{2}={\bf 1}$.

Let $\gn$ be an ideal of $\go_{F}$.
Throughout \S 2, we assume that a character $\chi$ of $F^{\times} \backslash \AA^{\times}$ is contained in $\Xi(\gn)$.

\subsection{Zeta integrals of Eisenstein series}
\label{Zeta integrals of Eisenstein series}
We consider Eisenstein series for $f \in I(\chi|\cdot|_{\AA}^{\nu/2}))^{\bfK_{\infty}\bfK_{0}(\gn)}$.
Let $n$ be the maximal nonnegative integer $m$ such that $S_{m}({\gn}{\gf}_{\chi}^{-2}) \neq \emptyset$.
For each $v \in \Sigma_{F}$, the space $I(\chi_{v}|\cdot|_{v}^{\nu/2})$ is defined in the same way as the global case.
For
$\rho = (\rho_{k})_{1 \le k \le n} \in \Lambda_{\chi}(\gn) = \prod_{k=1}^{n}{\rm Map}(S_{k}({\gn}{\gf}_{\chi}^{-2}),\{0, \ldots,  k\})$, let $f_{\chi, \rho}^{(\nu)}$ denote the vector in $I(\chi |\cdot|_{\AA}^{\nu/2})$ corresponding to
$$
\bigotimes_{v\in\Sigma_{\infty}} f_{0,\chi_{v}}^{(\nu)}\otimes \bigotimes_{v\in S_{1}(\gn\gf_{\chi}^{-2})}\tilde{f}_{\rho_{1}(v),\chi_{v}}^{(\nu)}
\otimes \cdots \otimes \bigotimes_{v \in S_{n}(\gn\gf_{\chi}^{-2})} \tilde{f}_{\rho_{n}(v), \chi_{v}}^{(\nu)}\otimes \bigotimes_{v\in\Sigma_{\fin}-S(\gn\gf_{\chi}^{-2})}\tilde{f}_{0,\chi_{v}}^{(\nu)}
$$
by the isomorphism
$I(\chi |\cdot|_{\AA}^{\nu/2})\cong \bigotimes_{v\in \Sigma_{F}} I(\chi_{v} |\cdot|_{v}^{\nu/2})$,
where
for $v \in \Sigma_{\infty}$, $f_{0, \chi_{v}}^{(\nu)}$ is the spherical vector in $I(\chi_{v}|\cdot|_{v}^{\nu/2})$ normalized so that $f_{0, \chi_{v}}^{(\nu)}(e)$ equals one
and for $v \in \Sigma_{\fin}$, $\tilde{f}_{k, v}$ is the $\bfK_{0}(\gn \go_{v})$-invariant vector
for $v \in \Sigma_{\fin}$, which is constructed in \cite[\S 7 and \S 8]{Sugiyama}.
Then, for any $\rho = (\rho_{k})_{1 \le k \le n} \in \Lambda_{\chi}(\gn)$,
the family $\{f_{\chi, \rho}^{(\nu)}\}_{\nu \in \CC}$ is a flat section.
Moreover, if $\nu \in i\RR$, the finite set
$\{ f_{\chi, \rho}^{(\nu)} \}_{\rho \in \Lambda_{\chi}(\gn)}$
is an orthonormal basis of 
$I(\chi|\cdot|_{\AA}^{\nu/2})^{{\bf K}_{\infty}{\bf K}_{0}(\gn)}$ (cf. \cite[Proposition 33]{Sugiyama}).

Let $\rho \in \Lambda_{\chi}(\gn)$ and set $E_{\chi, \rho}(\nu,g) = E(f_{\chi, \rho}^{(\nu)}, g)$.
The constant term of $E(f_{\chi, \rho}^{(\nu)}, g)$ is defined by
$$E_{\chi, \rho}^{\circ}(\nu,g) = \int_{F\backslash \AA}E_{\chi, \rho}\left(\nu,\left(\begin{matrix}1 & x
\\ 0 & 1 \end{matrix}\right)g\right)dx.$$
For $k \in \{1, \ldots, n\}$, the sets $U_{k}(\rho)$, $R_{k}(\rho)$ and $R_{0}(\rho)$ are defined as follows:
$$U_{k}(\rho) = \bigcup_{m = k}^{n}\rho_{m}^{-1}(k) - S(\gf_{\chi}),
\hspace{5mm}R_{k}(\rho) = \bigcup_{m = k}^{n}\rho_{m}^{-1}(k) \cap S(\gf_{\chi}),$$
$$R_{0}(\rho) = \left(\bigcup_{m = 0}^{n}\rho_{m}^{-1}(0)\cap S(\gf_{\chi})\right) \bigcup (S(\gf_{\chi})-S(\gn\gf_{\chi}^{-2})).$$
Furthermore, for any $k \in \NN_{0}$, set
$$S_{k}(\rho) = \begin{cases}
R_{0}(\rho) & \text{(if $k = 0$)}, \\
U_{k}(\rho)\cup R_{k}(\rho) & \text{(if $k \ge 1$)},
\end{cases}$$
$R(\rho) = \bigcup_{k=0}^{n}R_{k}(\rho)$
and $S(\rho) = \bigcup_{k=0}^{n} S_{k}(\rho)$.
Then, by \cite[Proposition 34]{Sugiyama} we have
$$E_{\chi, \rho}^{\circ}(\nu,g) = f_{\chi, \rho}^{(\nu)}(g) + D_{F}^{-1/2}A_{\chi, \rho}(\nu) \frac{L(\nu, \chi^{2})}{L(1+\nu, \chi^{2})}f_{\chi^{-1}, \rho}^{(-\nu)}(g),$$
where
{\allowdisplaybreaks\begin{align*}
A_{\chi, \rho}(\nu)
= &
{\rm N}(\gf_{\chi})^{-\nu} \prod_{k=0}^{n}\prod_{v\in S_{k}(\rho)}
\bigg\{q_{v}^{d_{v}/2} q_{v}^{-k\nu}
\frac{\e(1 - \nu, \chi_{v}^{-2}, \psi_{F_{v}})\e(1 + \nu/2, \chi_{v}, \psi_{F_{v}})}
{\e(1 - \nu/2, \chi_{v}^{-1}, \psi_{F_{v}})}
\frac{L(1+\nu, \chi_{v}^{2})}{L(1-\nu, \chi_{v}^{-2})}\bigg\}.
\end{align*}
}

We fix a character $\eta$ of $F^{\times} \backslash \AA^{\times}$ satisfying ($\star$) in \S 2.1. 
For any $v \in \Sigma_{\fin}-S(\gf_{\eta})$ and $k\in \NN_{0}$,
let $Q_{k, \chi}^{(\nu)}(\eta_{v}, X)$ be the polynomial defined in \cite[\S 9]{Sugiyama}.
Then, we have the following.
\begin{prop}\label{prop:Z* of Eisen}
\cite[Proposition 35]{Sugiyama}
We set $E_{\chi, \rho}^{\natural}(\nu, g) = E_{\chi, \rho}(\nu, g) -E_{\chi, \rho}^{\circ}(\nu, g)$. Then $E_{\chi, \rho}^{\natural}(\nu, -)$ is left $B_{F}$-invariant and we have
{\allowdisplaybreaks\begin{align*}
Z^{*}(s, \eta, E_{\chi, \rho}^{\natural}(\nu,-)) = &
\Gcal(\eta)D_{F}^{-\nu/2}
{\rm N}(\gf_{\chi})^{1/2-\nu}
B_{\chi, \rho}^{\eta}(s, \nu)\frac{L(s + \nu/2, \chi \eta)L(s - \nu/2, \chi^{-1} \eta)}{L(1+\nu, \chi^{2})},
\end{align*}
}where
{\allowdisplaybreaks\begin{align*}
B_{\chi,\rho}^{\eta}(s, \nu)
= & D_{F}^{s-1/2}\left\{\prod_{k=0}^{n}\prod_{v\in S_{k}(\rho)} Q^{(\nu)}_{k,\chi_{v}}(\eta_{v}, q_{v}^{1/2-s})L(1+\nu, \chi_{v}^{2})\right\} \\
& \times \prod_{v\in U_{1}(\rho)}(1+q_{v}^{-1})q_{v}^{-\nu/2}\prod_{k=2}^{n}
\prod_{v\in U_{k}(\rho)}\left(\frac{q_{v}+1}{q_{v}-1}\right)^{1/2}q_{v}^{-k\nu/2}\\
& \times \left\{\prod_{k=0}^{n}\prod_{v\in R_{k}(\rho)}q_{v}^{d_{v}/2-k\nu/2}(1-q_{v}^{-1})^{1/2}\overline{\Gcal(\chi_{v})}\right\}
\prod_{v\in \Sigma_{\fin}-R(\rho)}\chi_{v}(\varpi_{v})^{d_{v}}.
\end{align*}
}
\end{prop}

\subsection{Regularized periods of Eisenstein series}
\label{Regularized periods of Eisenstein series}
For any characters $\chi_{1}$ and $\chi_{2}$ of $F^{\times} \backslash \AA^{\times}$, we put
$\d_{\chi_{1}, \chi_{2}} = \d(\chi_{1} = \chi_{2})$.
The regularized period $P_{\rm reg}^{\eta}(E_{\chi, \rho}(\nu, -))$ was computed as follows in \cite{Sugiyama}.

\begin{prop}\label{prop:Eisen P=L}
\cite[Main Theorem B]{Sugiyama}
Assume $\nu \in i\RR$.
Then the integral
$P_{\b,\l}^{\eta}(E_{\chi, \rho}(\nu, -))$ converges absolutely for any $(\b, \l)\in \Bcal \times \CC$ such that $\Re(\l) > 1$.
Moreover $P_{\rm reg}^{\eta}(E_{\chi, \rho}(\nu, -))$ can be defined, and we have
{\allowdisplaybreaks\begin{align*}
P_{\rm reg}^{\eta}(E_{\chi, \rho}(\nu, -)) = & \Gcal(\eta)D_{F}^{-\nu/2}{\rm N}(\gf_{\chi})^{1/2-\nu}
B_{\chi, \rho}^{\eta}(1/2, \nu)\frac{L((1 + {\nu})/{2}, \chi \eta)L((1 - {\nu})/{2}, \chi^{-1} \eta)}{L(1+\nu, \chi^{2})}.
\end{align*}
}
\end{prop}

We define two functions ${\mathfrak e}_{\chi, \rho, -1}$ and ${\mathfrak e}_{\chi, \rho, 0}$
on $G_{\AA}$ by the Laurent expansion
$$E_{\chi, \rho}(\nu, g) = \frac{{\mathfrak e}_{\chi, \rho, -1}(g)}{\nu-1} + {\mathfrak e}_{\chi, \rho, 0}(g) + {\mathcal O}(\nu-1), \hspace{5mm}(\nu \rightarrow 1).$$
We explain the regularized $\eta$-periods of 
${\mathfrak e}_{\chi, \rho, -1}$ and that of ${\mathfrak e}_{\chi, \rho, 0}$.
Set $R_{F} = \Res_{s=1}\zeta_{F}(s) = \vol(F^{\times}\backslash \AA^{1})$,
where $\zeta_{F}(s)$ is the completed Dedekind zeta function of $F$.
The regularized period $P_{\rm reg}^{\eta}({\mathfrak e}_{\chi, \rho, -1})$ was computed as follows in \cite{Sugiyama}.

\begin{prop}
\label{prop:e-1 = chi det}
\cite[Lemma 38 and Theorem 39]{Sugiyama}
We have
$${\mathfrak e}_{\chi, \rho, -1}(g) = \d\left(\chi^{2} = {\bf 1}, \gf_{\chi} = \go_{F}, S(\rho) = \emptyset \right) \frac{D_{F}^{-1/2}R_{F}}{\zeta_{F}(2)}\chi(\det g)$$
for any $g \in G_{\AA}$.
Moreover, for $\l \in \CC$ such that $\Re(\l) > 0$, we have
$$P_{\b, \l}^{\eta}({\mathfrak e}_{\chi, \rho, -1}) = \d\left(\chi = \eta, \gf_{\chi} = \go_{F}, S(\rho) = \emptyset \right)
\frac{2D_{F}^{-1/2}R_{F}^{2}}{\zeta_{F}(2)}\frac{\b(0)}{\l}
$$
and $P_{\rm reg}^{\eta}({\mathfrak e}_{\chi, \rho, -1}) = 0$.
\end{prop}

For any character $\xi$ of $F^{\times} \backslash \AA^{\times}$, we define $R(\xi)$, $C_{0}(\xi)$ and $C_{1}(\xi)$ by the Laurent expansion 
$$L(s, \xi) = \frac{R(\xi)}{s-1} + C_{0}(\xi) + C_{1}(\xi)(s-1) + {\mathcal O}((s-1)^{2}), \hspace{5mm} (s \rightarrow 1).$$
We note $R({\xi}) = \d_{\xi, {\bf 1}}R_{F}$ for any character $\xi$ of $F^{\times} \backslash \AA^{\times}$.
The regularized period $P_{\rm reg}^{\eta}({\mathfrak e}_{\chi, \rho ,0})$ is defined under some conditions, and
$P_{\b,\l}^{\eta}({\mathfrak e}_{\chi, \rho ,0})$ was computed as follows in \cite{Sugiyama}.

\begin{prop}
\label{prop:P(e_{chi, rho, 0})}
\cite[Theorem 40 and Corollary 41]{Sugiyama}
Let $\eta$ be a character of $F^{\times} \backslash \AA^{\times}$ satisfying ($\star$) in \S 2.1.
The integral $P_{\b,\l}^{\eta}({\mathfrak e}_{\chi, \rho ,0})$ converges absolutely for any $(\b, \l)\in \Bcal \times \CC$ such that $\Re(\l) >1$.
There exists an entire function $f(\l)$ on $\CC$ such that
\begin{eqnarray*}
P_{\b,\l}^{\eta}({\mathfrak e}_{\chi, \rho ,0}) & = &
\d_{\chi, \eta}R_{F}f_{\chi, \rho}^{(1)}(e)
\left\{\frac{1}{\l-1} + \frac{1}{\l+1} \right\} \b(1) \\
&&+
2\d_{\chi, \eta}R_{F}\frac{D_{F}^{-1/2}f_{\chi, \rho}^{(1)}(e)}{\zeta_{F}(2)} \bigg\{R_{F}\left(-\frac{\zeta_{F}'(2)}{\zeta_{F}(2)}A_{\chi, \rho}(1) + A'_{\chi, \rho}(1)\right) + C_{0}({\bf 1})A_{\chi, \rho}(1) \bigg\}\frac{\b(0)}{\l} \\
&& + f(\l)
- \Gcal(\eta)D_{F}^{-1/2}R_{F}\d_{\chi, \eta}
\bigg\{
- \frac{\tilde{B}_{\chi, \rho}^{\eta}(1)}{\l + 1} +\frac{\tilde{B}_{\chi, \rho}^{\eta}(-1)}{\l - 1}
\bigg\}\b(1) \\
& & - \frac{\Gcal(\eta)D_{F}^{-1/2}}{\zeta_{F}(2)}\d_{\chi, \eta}
\bigg\{
- (\tilde{B}_{\chi, \rho}^{\eta})^{\prime}(0)R_{F}^{2}\frac{\b(0)}{\l}
+ \tilde{B}_{\chi, \rho}^{\eta}(0)R_{F}^{2}\frac{\b(0)}{\l^{2}}
\bigg\},
\end{eqnarray*}
where $\tilde{B}^{\eta}_{\chi, \rho}(z) = \e(-z, \chi^{-1}\eta)B_{\chi, \rho}^{\eta}(-z+1/2, 1)
$.
Moreover we have
\begin{eqnarray*}
 {\rm CT}_{\l=0}P_{\b, \l}^{\eta}({\mathfrak e}_{\chi, \rho ,0})
=  \frac{\Gcal(\eta)D_{F}^{-1/2}{\rm N}(\gf_{\chi})^{-1/2}}{L(2, \chi^{2})}
\bigg\{
- \frac{1}{2}\d_{\chi, \eta} \tilde{B}_{\chi, \rho}^{\eta}(0)R_{F}^{2}\b^{\prime\prime}(0) 
+ a_{\chi, \rho}^{\eta}(0)\b(0)\bigg\},
\end{eqnarray*}
where
{\allowdisplaybreaks\begin{align*}
a_{\chi, \rho}^{\eta}(0) = 
 - \frac{1}{2}\d_{\chi, \eta}(\tilde{B}_{\chi, \rho}^{\eta})^{\prime\prime}(0)R_{F}^{2}
 - 2\d_{\chi, \eta}\tilde{B}_{\chi, \rho}^{\eta}(0)R_{F}C_{1}({\bf 1})
+ \tilde{B}_{\chi, \rho}^{\eta}(0)C_{0}(\chi\eta)^{2}.
\end{align*}
}
\end{prop}

\subsection{An orthonormal basis of $V_{\pi}^{\bfK_{\infty}\bfK_{0}(\gn)}$}
Let $\pi$ be a cuspidal automorphic representation of $G_{\AA}$ such that $\pi \in \Pi_{\rm cus}(\gn)$.
We put
$$\PP^{\eta}(\pi ; \bfK_{0}(\gn)) = \sum_{\varphi \in B} \overline{Z^{*}(1/2, {\bf 1}, \varphi)} Z^{*}(1/2, \eta, \varphi),$$
where $B$ is an orthonormal basis of $V_{\pi}^{\bfK_{\infty}\bfK_{0}(\gn)}$.
In this subsection, we examine $\PP^{\eta}(\pi ; \bfK_{0}(\gn))$.
Set $\varphi_{\pi}^{\rm new} =\varphi_{\pi, \rho_{\pi}}$, where $\rho_{\pi}$ is a unique element of $\Lambda_{\pi}^{0}(\gf_{\pi})$.

\begin{lem}\label{lem:positivity of P}
The value $\PP^{\eta}(\pi ; \bfK_{0}(\gn))$ is independent of the choice of an orthonormal basis of $V_{\pi}^{\bfK_{\infty}\bfK_{0}(\gn)}$ and we have
$$\PP^{\eta}(\pi ; \bfK_{0}(\gn)) = D_{F}^{-1/2}\Gcal(\eta)w^{\eta}_{\gn}(\pi) \frac{L(1/2, \pi)L(1/2, \pi \otimes \eta)}{||\varphi_{\pi}^{\rm new}||^{2}}.$$
Here $w_{\gn}^{\eta}(\pi)$ is an explicit nonnegative constant defined as
$$w_{\gn}^{\eta}(\pi) = \prod_{k=1}^{n} \prod_{v \in S_{k}(\gn\gf_{\pi}^{-1})} r(\pi_{v}, \eta_{v}, k)
= \prod_{v \in S(\gn\gf_{\pi}^{-1})} r(\pi_{v}, \eta_{v}, \ord_{v}(\gn\gf_{\pi}^{-1}))$$
where $r(\pi_{v}, \eta_{v}, k)$ is defined as follows:
\begin{itemize}
\item If $c(\pi_{v}) \ge 2$, then $r(\pi_{v}, \eta_{v}, k) = \begin{cases}
k+1 & \text{$($if $\eta_{v}(\varpi_{v})=1$$)$,} \\
2^{-1}(1 + (-1)^{k}) & \text{$($if $\eta_{v}(\varpi_{v})=-1$$)$.}
\end{cases}$

\item If $c(\pi_{v}) =1$, then $\pi_{v}$ is isomorphic to $\s(\chi_{v}|\cdot|_{v}^{1/2}, \chi_{v}|\cdot|_{v}^{-1/2})$ for some unramified character $\chi_{v}$ of $F_{v}^{\times}$. Then
$$r(\pi_{v}, \eta_{v}, k) = 
\begin{cases}
\displaystyle 1+ \frac{1-\chi_{v}(\varpi_{v})q_{v}^{-1}}{1+\chi_{v}(\varpi_{v})q_{v}^{-1}}k & 
\text{$($if $\eta_{v}(\varpi_{v})=1$$)$,} \\
\ \\
2^{-1}(1+(-1)^{k}) & \text{$($if $\eta_{v}(\varpi_{v})=-1$$)$.}
\end{cases}$$

\item If $c(\pi_{v})=0$ and $(\a_{v}, \a_{v}^{-1})$ is the Satake parameter of $\pi_{v}$, then
{\allowdisplaybreaks\begin{align*}
r(\pi_{v}, \eta_{v},k)=
\begin{cases}
\displaystyle \frac{2}{1+Q(\pi_{v}) }+ \frac{1-Q(\pi_{v}) }{1+Q(\pi_{v}) }
\frac{q_{v}+1}{q_{v}-1}(k-1)
& \text{$($if $\eta_{v}(\varpi_{v})=1$$)$,} \\
\ \\
\displaystyle 
\frac{q_{v}+1}{q_{v}-1}\frac{1+(-1)^{k}}{2} & \text{$($if $\eta_{v}(\varpi_{v})=-1$$)$,}
\end{cases}
\end{align*}
}where $Q(\pi_{v}) = (\a_{v}+\a_{v}^{-1})(q_{v}^{1/2}+q_{v}^{-1/2})^{-1}$.
\end{itemize}

Moreover, $\Gcal(\eta)^{-1}\PP^{\eta}(\pi ; \bfK_{0}(\gn))$ is nonnegative.
\end{lem}
\Proof
The first assertion is obvious. Thus we may take $\{ ||\varphi_{\pi, \rho}||^{-1} \varphi_{\pi, \rho}\}_{\rho \in \Lambda_{\pi}^{0}(\gn)}$ as $B$. 
By virtue of Proposition \ref{prop:period = L(1/2)}, we have
{\allowdisplaybreaks\begin{align*}
& \PP^{\eta}(\pi ; \bfK_{0}(\gn)) = \sum_{\rho \in \Lambda_{\pi}^{0}(\gn)}\frac{1}{||\varphi_{\pi, \rho}||^{2}}
Z^{*}(1/2, {\bf 1}, \varphi) Z^{*}(1/2, \eta, \varphi) \\
= & \sum_{\rho \in \Lambda_{\pi}^{0}(\gn)} \prod_{k = 1}^{n}\prod_{v \in S_{k}(\gn \gf_{\pi}^{-1})}
\left\{ \frac{\overline{Q_{\rho_{k}(v), v}^{\pi_{v}}({\bf 1}_{v}, 1)}Q_{\rho_{k}(v), v}^{\pi_{v}}(\eta_{v}, 1)} {\t_{\pi_{v}}(\rho_{k}(v), \rho_{k}(v))} \right\}
\frac{\Gcal({\bf 1})\Gcal(\eta) L(1/2, \pi)L(1/2, \pi \otimes \eta)}{||\varphi_{\pi}^{\rm new}||^{2}}.
\end{align*}
Then, we obtain the second assertion by setting 
$$w_{\gn}^{\eta}(\pi) = \sum_{\rho \in \Lambda_{\pi}^{0}(\gn)} \prod_{k = 1}^{n}\prod_{v \in S_{k}(\gn \gf_{\pi}^{-1})}
\left\{ \frac{\overline{Q_{\rho_{k}(v), v}^{\pi_{v}}({\bf 1}_{v}, 1)}Q_{\rho_{k}(v), v}^{\pi_{v}}(\eta_{v}, 1)} {\t_{\pi_{v}}(\rho_{k}(v), \rho_{k}(v))} \right\}.$$
Here $\tau_{\pi_{v}}(j, j) = ||\phi_{j, v}||_{v}^{2}$ for $k \in \NN$, where $||\cdot||_{v}$ is the norm on $V_{\pi_{v}}$ defined by the $G_{v}$-invariant inner product normalized
so that $||\phi_{0,v}||_{v}=1$. We remark that an explicit formula of $\tau_{\pi_{v}}(j, j)$ was given in \cite[Corollaries 12 and 16]{Sugiyama}.
By definition and a direct computation, we have
$$w_{\gn}^{\eta}(\pi)=
\prod_{k=1}^{n}\bigg\{
\sum_{ (j_{v})_{v}\in \{0, \ldots, k\}^{^{S_{k}(\gn\gf_{\pi}^{-1})}}}
\prod_{v \in S_{k}(\gn\gf_{\pi}^{-1})}r_{v, j_{v}}\bigg\}
=\prod_{k=1}^{n} \prod_{v \in S_{k}(\gn\gf_{\pi}^{-1})} \sum_{j=0}^{k}r_{v, j}
$$
and $\sum_{j=0}^{k}r_{v, j} = r(\pi_{v}, \eta_{v}, k)$,
where
$r_{v, j} = \overline{Q_{j, v}^{\pi_{v}}({\bf 1}_{v}, 1)}Q_{j, v}^{\pi_{v}}(\eta_{v}, 1) \t_{\pi_{v}}(j, j)^{-1}$.

Then, one can check $w_{\gn}^{\eta}(\pi) \in \RR_{\ge 0}$ easily by noting $|Q(\pi_{v})|<1$
when $c(\pi_{v})=0$.
Since the estimate $L(1/2, \pi)L(1/2, \pi \otimes \eta) \ge 0$ holds by \cite{Guo}, we can prove the lemma.
\QED

Since $\eta^{2}= {\bf 1}$, we have $\tilde{\eta}(\gn)=\pm1$.
We consider only the case of $\tilde{\eta}(\gn)=1$ because of the following reason.
\begin{lem}
Let $\pi$ be a $\bfK_{\infty}$-spherical irreducible cuspidal automorphic representation of $G_{\AA}$ with trivial central character.
Let $\eta$ be a character of $F^{\times} \backslash \AA^{\times}$ such that $\eta^{2}={\bf 1}$ and $\gf_{\eta}$ is relatively prime to $\gf_{\pi}$.
Suppose that $\eta_{v}(-1)=1$ for any $v \in \Sigma_{\infty}$.
Then, $L(1/2, \pi)L(1/2, \pi \otimes \eta) = 0$ unless $\tilde{\eta}(\gf_{\pi})=1$.
\end{lem}
\begin{proof}
By the argument in the proof of \cite[Lemma 2.3]{TsuzukiSpec}, it is enough to show
$\e(1/2, \pi_{v}, \psi_{F_{v}})\e(1/2, \pi_{v} \otimes \eta_{v}, \psi_{F_{v}})= \eta_{v}(\varpi_{v}^{c(\pi_{v})})$
for any $v \in \bigcup_{k\ge 2} S_{k}(\gf_{\pi})$.
It follows immediately from fundamental properties of $\e$-factors (cf. \cite[1.1]{Schmidt}).
We note that $\eta_{v}$ is unramified if $v \in S(\gf_{\pi})$.
\end{proof}
\subsection{Adjoint $L$-functions}
Let $\pi$ be a cuspidal automorphic representation of $G_{\AA}$ contained in $\Pi_{\rm cus}(\gn)$.
We examine an explicit description of $||\varphi_{\pi}^{\rm new}||^{2}$ in terms of the adjoint $L$-function of $\pi$
by computing the Rankin-Selberg integral.
For any $v \in \Sigma_{F}$, we denote by $Z_{v}(s)$ the local Rankin-Selberg integral
$$\int_{\bfK_{v}} \int_{F_{v}^{\times}} \phi_{0, v}\left(\left(\begin{matrix} t&0 \\ 0&1 \end{matrix}\right)k_{v}\right)
\overline{\phi_{0, v}\left(\left(\begin{matrix} t&0 \\ 0&1 \end{matrix}\right)k_{v}\right)} |t_{v}|_{v}^{s-1} d^{\times}t_{v}dk_{v}.
$$
Let $\rho_{0}$ denote a unique element of $\Lambda_{\bf1}(\gn)$.
\begin{lem}\label{Adjoint L}
Set $S_{\pi}=\{ v \in \Sigma_{\fin} \ | \ \ord_{v}(\gf_{\pi})\ge2 \}$. We have
{\allowdisplaybreaks\begin{align*}
& \int_{Z_{\AA}G_{F}\backslash G_{\AA}}\varphi_{\pi}^{\rm new}(g)\overline{\varphi_{\pi}^{\rm new}(g)}E_{{\bf 1}, \rho_{0}}(2 s-1, g) dg \\
= & [\bfK_{\fin}: \bfK_{0}(\gf_{\pi})]^{-1}{\rm N}(\gf_{\pi})^{s}D_{F}^{s-3/2}\zeta_{F}(2s)^{-1}\zeta_{F}(s)L(s, \pi, {\rm Ad}) \prod_{v \in S_{\pi}} \frac{q_{v}^{d_{v}(3/2-s)}q_{v}^{c(\pi_{v})(1-s)}Z_{v}(s)}{L(s, \pi_{v}, {\rm Ad})}\frac{1+q_{v}^{-1}}{1+q_{v}^{-s}}
\end{align*}
}for $\Re(s) \gg 0$.
Moreover, we have $||\varphi_{\pi}^{\rm new}||^{2} = 2 {\rm N}(\gf_{\pi})[\bfK_{\fin}: \bfK_{0}(\gf_{\pi})]^{-1}L^{S_{\pi}}(1, \pi; {\rm Ad}).$
\end{lem}
\Proof
If $v \in \Sigma_{F} - S_{\pi}$, then $Z_{v}(s)$ is computed in \cite[Lemma 2.14]{TsuzukiSpec}. Hence, it suffices to examine $Z_{v}(s)$ when $v \in S_{\pi}$. By
$[\bfK_{v}: \bfK_{0}(\gp_{v}^{c(\pi_{v})})] = q_{v}^{c(\pi_{v})}(1+q_{v}^{-1})$,
we obtain the first assertion.

We note
$Z_{v}(1) = q_{v}^{-d_{v}/2}$ for $v\in S_{\pi}$.
Then we obtain the second assertion by taking the residue at $s=1$ since
$\Res_{s=1}E_{\bf 1, \rho_{0}}(s, g) = D_{F}^{-1/2}R_{F}\zeta_{F}(2)^{-1}$ holds by Proposition \ref{prop:e-1 = chi det}.
\QED

\section{Adelic Green functions}
\label{Adelic Green functions}
We define the adelic Green function on $G_{\AA}$ associated to an ideal $\gn$ of $\go_{F}$.
This was introduced in \cite{TsuzukiSpec} in the case where $\gn$ is square free. We define the function in the case where $\gn$ is an arbitrary ideal of $\go_{F}$.

For $v \in \Sigma_{\infty}$, the Green function on $GL(2, F_{v})$, denoted by $\Psi_{v}^{(z)}(s, -)$, is a function with the following property.

\begin{prop}
\cite[\S 4]{TsuzukiSpec}
Suppose $v \in \Sigma_{\infty}$.
We fix $s \in \CC$ and $z \in \CC$.
Let $f: G_{v} \rightarrow \CC$ be a smooth function such that
$f\left(\left(\begin{matrix} t_{1}&0\\ 0&t_{2} \end{matrix}\right)gk\right) = |t_{1}/t_{2}|_{v}^{-z}f(g)$ for any
$\left(\begin{matrix} t_{1}&0\\0&t_{2} \end{matrix}\right) \in H_{v}$, $k \in \bfK_{v}$ and $g \in G_{v}$.
Suppose $$\sum_{m=0}^{2}\left|\frac{d^{m}}{dr^{m}}f
\left( \begin{matrix} 
\cosh r & \sinh r \\ \sinh r & \cosh r
\end{matrix} \right)
\right| \ll (\cosh 2r)^{|\Re(z)|}, \hspace{5mm} r \in \RR.$$
If $\Re(s) > 2|\Re(z)|+1$,  then the equality
$$\int_{H_{v}\backslash G_{v}}\Psi_{v}^{(z)}(s, g)[R(\Omega_{v} - (s^{2}-1)/2)f](g)dg=f(e)$$
holds with the integral being convergent absolutely.
Here $\Omega_{v}$ denotes the Casimir operator of $G_{v} \cong GL(2, \RR)$.
\end{prop}
For $v \in \Sigma_{\fin}$, let $\TT_{v}$ denote the characteristic function of $\bfK_{v}
\left(\begin{matrix}\varpi_{v}&0 \\ 0&1 \end{matrix}\right)
\bfK_{v}$ on $G_{v}$ divided by $\vol(\bfK_{v} ,dg_{v})$.
The function $\TT_{v}$ is an element of the spherical Hecke algebra $\Hcal(G_{v}, \bfK_{v})$
and is called the $v$-th Hecke operator. 
The Green function on $GL(2, F_{v})$, also denoted by $\Psi_{v}^{(z)}(s, -)$, is a function with the following property.
\begin{prop}
\cite[\S 5]{TsuzukiSpec}
Suppose $v \in \Sigma_{\fin}$.
We fix $s \in \CC$ and $z\in \CC$ such that $\Re(s) > 2|\Re(z)|$.
Let $f: G_{v} \rightarrow \CC$ be a smooth function such that
$f\left(\left(\begin{matrix} t_{1}&0\\ 0&t_{2} \end{matrix}\right)gk\right) = |t_{1}/t_{2}|_{v}^{-z}f(g)$ for any
$\left(\begin{matrix} t_{1}&0\\0&t_{2} \end{matrix}\right) \in H_{v}$, $k \in \bfK_{v}$ and $g \in G_{v}$.
Then, the equality
$$\int_{H_{v}\backslash G_{v}}\Psi_{v}^{(z)}(s, g)[R(\TT_{v} - (q_{v}^{(1-s)/2} +q_{v}^{(1+s)/2})1_{\bfK_{v}}) f](g)dg = \vol(H_{v} \backslash H_{v}\bfK_{v})f(e)$$
holds as long as the integral on the left hand side converges absolutely.
Here $1_{\bfK_{v}}$ is the characteristic function of $\bfK_{v}$ on $G_{v}$ divided by $\vol(\bfK_{v}, dg_{v})$.
\end{prop}

For any $v \in S(\gn)$, we set
$$\Phi_{\gn, v}^{(z)} \left(\left(\begin{matrix}t_{1} & 0 \\ 0 & t_{2} \end{matrix}\right)
\left(\begin{matrix} 1 & x \\ 0 & 1 \end{matrix}\right) k
\right) = \left|\frac{t_{1}}{t_{2}}\right|_{v}^{z}\d(x \in \go_{v}) \d(k \in \bfK_{0}(\gn \go_{v}))$$
for any $t_{1}, t_{2} \in F_{v}^{\times}$, $x \in F_{v}$ and $k \in \bfK_{v}$
and
put $\Phi_{0, v}^{(z)} = \Phi_{\gn, v}^{(z)}$ for $v \in \Sigma_{\fin} - S(\gn)$.

Set $\mathfrak{X}_{S} = \prod_{v \in \Sigma_{\infty}}\CC \times \prod_{v \in S_{\fin}}(\CC/4 \pi i (\log q_{v})^{-1}\ZZ)$ and
$q(\bfs)= \inf_{v\in S}(\Re(s_{v})+1)/4$ for any $\bfs \in \mathfrak{X}_{S}$.
Fix a finite subset $S$ of $\Sigma_{F}$ such that $\Sigma_{\infty} \subset S$.
For any $\bfs \in \mathfrak{X}_{S}$ and $z \in \CC$ such that $q(\bfs) > |\Re(z)| + 1$,
{\it the adelic Green function} is defined by
$$\Psi^{(z)}(\gn|{\bf s}, g) = \prod_{v \in \Sigma_{\infty}} \Psi_{v}^{(z)}(s_{v}, g_{v}) \prod_{v \in S_{\fin}} \Psi_{v}^{(z)}(s_{v}, g_{v})
\prod_{v \in S(\gn)} \Phi_{\gn, v}^{(z)}(g_{v}) \prod_{v \notin S\cup S(\gn)} \Phi_{0, v}^{(z)}(g_{v})$$
for any $g = (g_{v})_{v\in \Sigma_{F}} \in G_{\AA}$.
Note that the function $\Psi^{(z)}(\gn|\bfs; -)$ on $G_{\AA}$ is right $\bfK_{\infty}\bfK_{0}(\gn)$-invariant and continuous on $G_{\AA}$.
Moreover, we have $\Psi^{(z)}\left(\gn \bigg|\bfs; \left(\begin{matrix} t_{1}&0\\ 0&t_{2} \end{matrix}\right)g\right) = |t_{1}/t_{2}|_{\AA}^{z} \Psi^{(z)}(\gn|\bfs; g)$
for any $\left(\begin{matrix} t_{1}&0\\ 0&t_{2} \end{matrix}\right) \in H_{\AA}$ and $g \in G_{\AA}.$

To state the property of adelic Green functions, for any $\varphi \in C_{c}^{\infty}(\gA G_{F} \backslash G_{\AA})$ we consider the integral
$$\varphi^{H, (z)}(g) = \int_{\gA H_{F} \backslash H_{\AA}} \varphi(hg)\chi_{z}(h)dh,
$$
where $\chi_{z}:H_{F}\backslash H_{\AA} \rightarrow \CC^{\times}$ is defined by
$$
\chi_{z}\left(\begin{matrix}t_{1} & 0 \\ 0 & t_{2}\end{matrix}\right) = |t_{1}/t_{2}|_{\AA}^{z}
$$
for any $t_{1}, t_{2} \in \AA^{\times}$.
The integral $\varphi^{H, (z)}(g)$ converges absolutely and $\varphi^{H, (z)}(hg) = \chi_{z}(h)^{-1}\varphi^{H, (z)}(g)$ holds for any $h \in H_{\AA}$
(cf. \cite[\S 6.2]{TsuzukiSpec}).

Let $\gZ(\mathfrak{g}_{\infty})$ be the center of the universal enveloping algebra
of the complexification of the Lie algebra of $GL(2, F \otimes_{\QQ} \RR)$.
For $\bfs \in \mathfrak{X}_{S}$, the element ${\bf \Omega}_{S}(\bfs)$ of
the algebra $\gZ(\mathfrak{g}_{\infty})\otimes \{\bigotimes_{v \in S_{\fin}} \Hcal(G_{v}, \bfK_{v})\}$
is defined as
$${\bf \Omega}_{S}(\bfs) = \bigotimes_{v \in \Sigma_{\infty}} \left(\Omega_{v} - \frac{s_{v}^{2}-1}{2}\right) \bigotimes_{v \in S_{\fin}}\left(\TT_{v} - (q_{v}^{(1-s_{v})/2} + q_{v}^{(1+s_{v})/2}) 1_{\bfK_{v}}\right).$$
The following proposition is proved in a similar way to \cite[Lemma 6.3]{TsuzukiSpec}.
\begin{prop}
\label{prop:property of adelic Green}
Suppose
$q({\bf s}) > 2|\Re(z)|+1$. Then,
for any $\varphi \in C_{c}^{\infty}(\gA G_{F} \backslash G_{\AA})^{\bfK_{\infty}\bfK_{0}(\gn)}$,
the function $g \mapsto \Psi^{(z)}(\gn | \bfs; g)\varphi^{H, (z)}(g)$ is integrable on $H_{\AA} \backslash G_{\AA}$ and we have
$$
\int_{H_{\AA} \backslash G_{\AA}} \Psi^{(z)}(\gn|{\bf s}, g) [R(\Omega_{S}(\bfs))\varphi^{H, (z)}](g)dg =
\vol(H_{\fin} \backslash H_{\fin}\bfK_{0}(\gn)) \varphi^{H, (z)}(e).
$$
\end{prop}

\section{Spectral expansions of renormalized Green functions}
\label{Spectral expansions of renormalized Green functions}
The set $\mathfrak{X}_{S} = \prod_{v \in \Sigma_{\infty}}\CC \times \prod_{v \in S_{\fin}}(\CC/4 \pi i (\log q_{v})^{-1}\ZZ)$
is considered as a complex manifold with respect to a usual complex structure.
Let $\Acal_{S}$ be the space of holomorphic functions $\a(\bfs)$ on $\mathfrak{X}_{S}$ such that for any $v \in S$ and $\bfs' \in \mathfrak{X}_{S-\{v\}}$, the function $s_{v} \mapsto \a(\bfs', s_{v})$ is contained in $\Bcal$.

For $\bfc\in \RR^{S}$, we put $\LL_{S}(\bfc) = \{ \bfs \in \mathfrak{X}_{S} \ | \ \Re(\bfs)=\bfc \}$.
A multidimensional contour integral of a holomorphic function $f(\bfs)$ on $\mathfrak{X}_{S}$
along $\LL_{S}(\bfc)$ is defined as in \cite[\S6.1]{TsuzukiSpec}.
Its contour integral is defined inductively as
$$\int_{\LL_{S}(\bfc)}f(\bfs)d\mu_{S}(\bfs) = \int_{L_{v}(c_{v})}\left\{ \int_{\LL_{S-\{v\}}(\bfc')}
f(\bfs', s_{v})d\mu_{S-\{v\}}(\bfs') \right\}d\mu_{v}(s_{v})$$
for $\bfc =(\bfc', c_{v}) \in \RR^{S}$,
where
$$
d\mu_{v}(s) = \begin{cases}
sds & \text{(if $v \in \Sigma_{\infty}$)} \\
\displaystyle \frac{1}{2} (\log q_{v}) (q_{v}^{(1+s)/2} - q_{v}^{(1-s)/2}) ds & 
\text{(if $v \in \Sigma_{\fin}$)}
\end{cases}
$$
and $L(c_{v})$ stands for $c_{v} + i \RR$ and $c_{v} +\CC/ 4 \pi(\log q_{v})^{-1} \ZZ$ for $v \in \Sigma_{\infty}$ and
$v \in \Sigma_{\fin}$, respectively.
Then, for $\bfc \in \RR^{S}$ and $z \in \CC$ such that $q(\bfc) > |\Re(z)| + 1$, the integral
$$\hat{\Psi}^{(z)}(\gn|\a; g) = \left(\frac{1}{2 \pi i}\right)^{\#S} \int_{\LL_{S}({\bfc})} \Psi^{(z)}(\gn|{\bf s}, g) \alpha({\bf s}) d\mu_{S}({\bf s})$$
is absolutely convergent and is independent of the choice of $\bfc$, and
the function $z\mapsto \hat{\Psi}^{(z)}(\gn|\a; g)$ is entire.
Furthermore, for $\b \in \Bcal$, $\l \in \CC$ and $g \in G_{\AA}$,
we consider the integral
$$\hat{\Psi}_{\b, \l}(\gn|\a; g) = \frac{1}{2 \pi i} \int_{L_{\s}}\frac{\b(z)}{z+\l} \{ \hat{\Psi}^{(z)}(\gn|\a; g) + \hat{\Psi}^{(-z)}(\gn|\a; g) \}dz$$
for $\s \in \RR$ such that $-\inf(q(\bfs)-1, \Re(\l)) < \s < q(\bfs)-1$.
The integral of the right hand side is absolutely convergent and is independent of the choice of $\s$.
Moreover, for $\a \in \Acal_{S}$, $\b \in \Bcal$ and $\l \in \CC$ with $\Re(\l)>0$,
the Poincar$\rm \acute{e}$ series of $\hat{\Psi}_{\b, \l}(\gn|\a; g)$ is defined to be
$$\hat{\bf \Psi}_{\b, \l}(\gn|\a; g) = \sum_{\gamma \in H_{F} \backslash G_{F}}\hat{\Psi}_{\b, \l}(\gn|\a; \gamma g)$$
for $g \in G_{\AA}$.
We have the following in the same way as \cite[Proposition 9.1]{TsuzukiSpec}:
\begin{enumerate}
\item The series $\hat{{\bf \Psi}}_{\b, \l}(\gn|\a; g)$ is absolutely convergent locally uniformly in $\{\Re(\l)>0\} \times G_{\AA}$.
Moreover, the function $\l \mapsto \hat{{\bf \Psi}}_{\b, \l}(\gn|\a; g)$ on $\Re(\l)>0$ is holomorphic and the function $g \mapsto \hat{{\bf \Psi}}_{\b, \l}(\gn|\a; g)$ on $G_{\AA}$ is continuous, left $G_{F}$-invariant and right $\bfK_{\infty}\bfK_{0}(\gn)$-invariant.

\item For $\Re(\l)>0$, we have $\hat{{\bf \Psi}}_{\b, \l}(\gn|\a; -) \in L^{l}(\gA G_{F} \backslash G_{\AA})$ for any $l>0$ such that $l(1-\Re(\l))<1$.
\end{enumerate}

Let us compute the spectral expansion of $\hat{{\bf \Psi}}_{\b, \l}(\gn|\a; -)$
explicitly.
Recall spectral parameters at $S$ of automorphic forms (cf. \cite[9.1.3]{TsuzukiSpec}).
If an automorphic form $\varphi$ on $G_{\AA}$ satisfies
that 
there exists $\nu_{\varphi, S}  = (\nu_{\varphi, v})_{v \in S} \in \mathfrak{X}_{S}$ such that
$$R(\Omega_{v})\varphi = \frac{\nu_{\varphi, v}^{2}-1}{2} \varphi$$
and
$$R(\TT_{v})\varphi = (q_{v}^{(1-\nu_{\varphi, v})/2} + q_{v}^{(1+\nu_{\varphi, v})/2})\varphi$$
hold for any $v \in \Sigma_{\infty}$ and any $v \in S_{\fin}$, respectively, then we call $\nu_{\varphi, S}$
the spectral parameter at $S$ of $\varphi$.
Set
$$C(\gn, S) = (-1)^{\#S}\vol(H_{\fin} \backslash H_{\fin}\bfK_{0}(\gn)).$$
We remark $\vol(H_{\fin} \backslash H_{\fin}\bfK_{0}(\gn))= D_{F}^{-1/2}[\bfK_{\fin}: \bfK_{0}(\gn)]^{-1}$.
By using Proposition \ref{prop:property of adelic Green} and a similar argument of \cite[Lemma 9.4]{TsuzukiSpec}, we have the following.
\begin{lem}\label{Green = period}
Assume $\Re(\l) > 1$. Then, for any automorphic form $\varphi$ on $G_{\AA}$ with spectral parameter $\nu_{\varphi, S}$, we have
$$\langle \hat{{\bf \Psi}}_{\b, \l}(\gn|\a; -) | \varphi \rangle_{L^{2}} = C(\gn, S)\alpha(\nu_{\varphi, S})P_{\b, \l}^{\bf 1}(\overline{\varphi}),$$
where $\langle \cdot | \cdot \rangle_{L^{2}}$ is the $L^{2}$-inner product on $L^{2}(Z_{\AA}G_{F} \backslash G_{\AA})$.
\end{lem}

For any character $\chi$ of $F^{\times} \backslash \AA^{\times}$ and $\a \in \Acal_{S}$, we define the function $\tilde{\a}_{\chi}$ on $\CC$ by
$$\tilde{\a}_{\chi}(\nu) = \a(({\nu}+ 2 i b(\chi_{v}))_{v \in S})$$
and write $\tilde{\a}(\nu)$ for $\tilde{\a}_{\bf 1}(\nu)$.

Fix an orthonormal basis $\Bcal_{\rm cus}(\gn)$ of $\sum_{\pi \in \Pi_{\rm cus}(\gn)} V_{\pi}^{\bfK_{\infty}\bfK_{0}(\gn)}$
and let
$\Bcal_{\rm res}(\gn)$ be the orthonormal system consisting of all functions $\vol(Z_{\AA}G_{F} \backslash G_{\AA})^{-1/2}\chi \circ \det$ on $G_{\AA}$ for any $\chi \in \Xi_{0}(\go_{F})$ such that $\chi^{2} = {\bf 1}$.
We write $\Lambda(\gn)$ for $\Lambda_{\bf 1}(\gn)$.
By Lemma \ref{Green = period},
the following spectral expansion of $\hat{{\bf \Psi}}_{\b, \l}(\gn|\a; g)$ is given in the same way
as \cite[Lemma 9.6]{TsuzukiSpec}.
\begin{lem}\label{spectral exp of Green}
Assume $\Re(\l)>1$. Then we have the expression
{\allowdisplaybreaks\begin{align*}
\hat{{\bf \Psi}}_{\b, \l}(\gn|\a; g) = & C(\gn, S) \bigg\{\sum_{\varphi \in \Bcal_{\rm cus}(\gn)}\alpha(\nu_{\varphi, S})P_{\b, \l}^{\bf 1}(\overline{\varphi})\varphi(g)
+ \sum_{\varphi \in \Bcal_{\rm res}(\gn)}\alpha(\nu_{\varphi, S})P_{\b, \l}^{\bf 1}(\overline{\varphi})\varphi(g) \\
& + \sum_{\chi \in \Xi(\gn)} \sum_{\rho \in \Lambda_{\chi}(\gn)}
\frac{R_{F}^{-1}}{8 \pi i} \int_{i\RR} \tilde{\alpha}_{\chi}(\nu)P_{\b, \l}^{\bf 1}(\overline{
E_{\chi, \rho}(\nu, -)})E_{\chi, \rho}(\nu, g)d\nu
\bigg\}.
\end{align*}
}The series and integrals in the right-hand side converge absolutely and locally uniformly on $Z_{\AA}G_{F} \backslash G_{\AA}$.
\end{lem}

\begin{lem}\label{lem:Green merom}
For any $g \in G_{\AA}$, the function $\l \mapsto \hat{{\bf \Psi}}_{\b, \l}(\gn|\a; g)$ on $\Re(\l) > 1$ is continued to a meromorphic function on $\Re(\l) > -1/2$. 
\end{lem}
\Proof
We give a proof in the same way as \cite[Lemma 9.8]{TsuzukiSpec}.
Let $\Psi_{\rm cus}(\l) = \Psi_{\rm cus}(\l, \a, g)$,
$\Psi_{\rm res}(\l) = \Psi_{\rm res}(\l, \a, g)$ and
$\Psi_{\rm ct}(\l) = \Psi_{\rm ct}(\l, \a, g)$ be the cuspidal part, the residual part and the Eisenstein part divided by $C(\gn, S)$ in the spectral expansion of $\hat{{\bf \Psi}}_{\b, \l}(\gn|\a; g)$ given in Lemma \ref{spectral exp of Green},
respectively.

First we examine $\Psi_{\rm res}(\l)$. For $\Re(\l) > 0$,
by applying Proposition \ref{prop:e-1 = chi det},
the function $\Psi_{\rm res}(\l)$ is written as
$$\Psi_{\rm res}(\l) = \sum_{\chi \in \Xi_{0}(\go_{F}), \chi^{2} = {\bf 1}}
\a(\nu_{\varphi_{\chi}, S})P_{\b, \l}^{\bf 1}(\overline{\varphi_{\chi}})\varphi_{\chi}(g)
= 2\tilde{\a}(1) \frac{R_{F}}{\vol(Z_{\AA}G_{F} \backslash G_{\AA})} \frac{\b(0)}{\l}$$
and has a meromorphic continuation to $\CC$. From this, ${\rm CT}_{\l =0}\Psi_{\rm res}(\l) =0$ holds.

Next we examine $\Psi_{\rm cus}(\l)$.
By the same computation as the proof of \cite[Lemma 9.8]{TsuzukiSpec},
the series $\Psi_{\rm cus}(\l)$ converges absolutely and the estimate
$$|\Psi_{\rm cus}(\l, \a, g)| \ll y(g)^{-m}, \ \ \ g \in \gS \cap G_{\AA}^{1}$$
holds, where $\gS$ denotes a Siegel set of $G_{\AA}$ such that $G_{\AA}=G_{F}\gS$
and $y$ denotes the height function on $G_{\AA}$.
Moreover, $\Psi_{\rm cus}(\l)$ is analytically continued to an entire function and we have
$${\rm CT}_{\l = 0}\Psi_{\rm cus}(\l) = \sum_{\varphi \in \Bcal_{\rm cus}(\gn)}\alpha(\nu_{\varphi, S})\overline{P_{\rm reg}^{\bf 1}(\varphi)}\varphi(g).$$

Therefore, it is enough to examine $\Psi_{\rm ct}(\l)$.
The argument is more complicated than that of \cite[Lemma 9.8]{TsuzukiSpec}.
Assume $\Re(\l) > 1$ and $\nu \in i \RR$. By the proof of \cite[Theorem 37]{Sugiyama},
the integral $P_{\l, \b}^{\bf 1}(E_{\chi^{-1}, \rho}(-\nu, -))$ can be expressed as
{\allowdisplaybreaks\begin{align*}
P_{\l, \b}^{\bf 1}(E_{\chi^{-1}, \rho}(-\nu, -)) = & P_{\chi^{-1}}({\bf 1}, \l, -\nu) + D_{F}^{-1/2}A_{\chi^{-1}, \rho}(-\nu)\frac{L(-\nu, \chi^{-2})}{L(1 - \nu, \chi^{-2})}P_{\chi}({\bf 1}, \l, \nu) \\
& + Q_{\chi^{-1}, \rho}^{+}({\bf 1}, \l, -\nu) + Q_{\chi^{-1}, \rho}^{-}({\bf 1}, \l, -\nu),
\end{align*}
}where
$$
P_{\chi^{\pm 1}}(\eta, \l, \pm \nu) 
=f_{\chi^{\pm 1}, \rho}^{(\pm \nu)}(e)\d_{\chi, \eta}R_{F}\left\{
\frac{\b((\mp \nu-1)/2)}{\l - (\pm \nu +1)/2} + \frac{\b((\pm \nu+1)/2)}{\l + (\pm \nu +1)/2} \right\}$$
and
$$Q_{\chi^{-1}, \rho}^{\pm}(\eta, \l, -\nu) =
\frac{1}{2 \pi i}\int_{L_{\pm \s}} Z^{*}(\pm z+1/2, \eta, E_{\chi^{-1}, \rho}^{\natural}(-\nu,-)) \frac{\b(z)}{\l + z}dz.$$
We remark $\overline{E_{\chi, \rho}(\nu, -)} = E_{\chi^{-1}, \rho}(-\nu, -)$. Furthermore, by the residue theorem, we have
{\allowdisplaybreaks\begin{align*}
& P_{\l, \b}^{\bf 1}(E_{\chi^{-1}, \rho}(-\nu, -)) \\
= & P_{\chi^{-1}}({\bf 1}, \l, -\nu) + D_{F}^{-1/2}A_{\chi^{-1}, \rho}(-\nu)\frac{\zeta_{F}(-\nu)}{\zeta_{F}(1 - \nu)}P_{\chi}({\bf 1},\l, \nu) + Q_{\chi^{-1}, \rho}^{0}({\bf 1}, \l, -\nu) \\
& -\bigg\{ \frac{\b((-\nu+1)/2)}{\l + (-\nu +1)/2}\Res_{z = (-\nu +1)/2} +
\frac{\b((\nu+1)/2)}{\l + (\nu +1)/2}\Res_{z = (\nu +1)/2} \\
& + \frac{\b((-\nu-1)/2)}{\l + (-\nu -1)/2}\Res_{z = (-\nu - 1)/2}
+\frac{\b((\nu-1)/2)}{\l + (\nu - 1)/2}\Res_{z = (\nu - 1)/2}\bigg\}f_{\chi^{-1}, \rho}^{\bf 1}(-z, -\nu),
\end{align*}
}where we put
$$f_{\chi}^{\eta}(z, \nu) = Z^{*}(z+1/2, \eta, E_{\chi, \rho}^{\natural}(\nu, -))$$
and
$$Q_{\chi, \rho}^{0}(\eta, \l, \nu) =\frac{1}{2 \pi i}\int_{L_{\s}}\{f_{\chi, \rho}^{\eta}(z, \nu) + f_{\chi, \rho}^{\eta}(-z, \nu) \} \frac{\b(z)}{\l+z}dz$$
for $\Re(\l) > -\s$.
Thus we express $\Psi_{\rm ct}(\l)$ as the sum of the following four terms:
{\allowdisplaybreaks\begin{align*}
\Phi_{1}(\l) = & \frac{1}{8 \pi i}\sum_{\rho \in \Lambda(\gn)}f_{{\bf1}, \rho}^{(-\nu)}(e)\int_{i \RR}\tilde{\a}(\nu)
\b((\nu-1)/2) \left\{\frac{1}{\l - (-\nu + 1)/2} + \frac{1}{\l + (-\nu + 1)/2}\right\} E_{{\bf 1}, \rho}(\nu; g)d\nu, \\
\Phi_{2}(\l) = & \frac{1}{8 \pi i}\sum_{\rho \in \Lambda(\gn)}f_{{\bf1}, \rho}^{(\nu)}(e)\int_{i \RR}\tilde{\a}(\nu)
D_{F}^{-1/2}A_{{\bf 1}, \rho}(-\nu)\frac{\zeta_{F}(-\nu)}{\zeta_{F}(1-\nu)}
\b((\nu+1)/2) \\
& \times \left\{\frac{1}{\l - (\nu + 1)/2} + \frac{1}{\l + (\nu + 1)/2}\right\} E_{{\bf 1}, \rho}(\nu; g)d\nu,\\
\Phi_{3}(\l) = & \frac{1}{8 \pi i}\sum_{\chi \in \Xi(\gn)}\sum_{\rho \in \Lambda_{\chi}(\gn)}\int_{i \RR}\tilde{\a}_{\chi}(\nu)Q_{\chi^{-1}, \rho}^{0}({\bf 1}, \l, -\nu)E_{\chi, \rho}(\nu, g) d\nu,\\
\Phi_{4}(\l) = & - \sum_{\chi \in \Xi(\gn)}\sum_{\rho \in \Lambda_{\chi}(\gn)}\frac{R_{F}^{-1}}{8 \pi i}\int_{i \RR} \bigg\{ \frac{\b((\nu+1)/2)}{\l + (\nu +1)/2}\Res_{z = (\nu +1)/2} \\
& + \frac{\b((-\nu+1)/2)}{\l + (-\nu +1)/2}\Res_{z = (-\nu +1)/2}
+ \frac{\b((\nu-1)/2)}{\l + (\nu -1)/2}\Res_{z = (\nu - 1)/2} \\
& +\frac{\b((-\nu-1)/2)}{\l + (-\nu - 1)/2}\Res_{z = (-\nu - 1)/2}\bigg\}\{f_{\chi^{-1}, \rho}^{\bf 1}(-z, -\nu) \} \tilde{\a}_{\chi}(\nu)E_{\chi, \rho}(\nu, g) d\nu.
\end{align*}
}
By the functional equation
$$D_{F}^{-1/2}A_{{\bf 1}, \rho}(-\nu)\frac{\zeta_{F}(-\nu)}{\zeta_{F}(1-\nu)} E_{{\bf 1}, \rho}(\nu, g) = E_{{\bf 1}, \rho}(-\nu, g)$$
of the Eisenstein series, the following equalities hold:
{\allowdisplaybreaks\begin{align*}
\Phi_{2}(\l) = & \frac{1}{8 \pi i}\sum_{\rho \in \Lambda(\gn)}f_{{\bf 1}, \rho}^{(\nu)}(e) \int_{i \RR}\tilde{\a}(\nu)
D_{F}^{-1/2}A_{{\bf 1}, \rho}(-\nu)\frac{\zeta_{F}(-\nu)}{\zeta_{F}(1-\nu)} E_{{\bf 1}, \rho}(\nu, g)
\b((\nu+1)/2) \\
& \times \left\{\frac{1}{\l - (\nu + 1)/2} + \frac{1}{\l + (\nu + 1)/2}\right\}d\nu \\
= & \frac{1}{8 \pi i}\sum_{\rho \in \Lambda(\gn)}f_{{\bf 1}, \rho}^{(\nu)}(e) \int_{i \RR}\tilde{\a}(\nu)
E_{{\bf 1}, \rho}(\nu, g)
\b((-\nu+1)/2) \left\{\frac{1}{\l - (-\nu + 1)/2} + \frac{1}{\l + (-\nu + 1)/2}\right\}d\nu \\
= & \Phi_{1}(\l).
\end{align*}
}Thus we have to consider only $\Phi_{1}(\l)$, $\Phi_{3}(\l)$ and $\Phi_{4}(\l)$.

We take $c > 1$. Then $\Phi_{1}(\l)$ is expressed as
{\allowdisplaybreaks\begin{align*}
\Phi_{1}(\l) = & \frac{1}{8 \pi i}\sum_{\rho \in \Lambda(\gn)}f_{{\bf1}, \rho}^{(-\nu)}(e)\int_{i \RR}\tilde{\a}(\nu)
\b((\nu-1)/2) \left\{\frac{1}{\l - (-\nu + 1)/2} + \frac{1}{\l + (-\nu + 1)/2} \right\}
E_{{\bf 1}, \rho}(\nu; g)d\nu \\
= & \frac{1}{8 \pi i}\sum_{\rho \in \Lambda(\gn)}f_{{\bf1}, \rho}^{(-\nu)}(e)\int_{i \RR}\tilde{\a}(\nu)
\b((\nu-1)/2) \frac{1}{\l + (-\nu + 1)/2} E_{{\bf 1}, \rho}(\nu; g)d\nu \\
& + \frac{1}{8 \pi i}\sum_{\rho \in \Lambda(\gn)}f_{{\bf1}, \rho}^{(-\nu)}(e)
\bigg\{ \int_{L_{c}}\tilde{\a}(\nu)
\b((\nu-1)/2) \frac{1}{\l - (-\nu + 1)/2}
E_{{\bf 1}, \rho}(\nu; g)d\nu \\
& - 2\pi i \tilde{\a}(1)\b(0) \frac{2}{\l} \mathfrak{e}_{{\bf 1}, \rho, -1}(g)\bigg\}.
\end{align*}
}The first term is holomorphic on $\Re(\l) > -1/2$,
the second term is holomorphic on $\Re(\l) > (-c+1)/2$
and the third term is holomorphic on $\CC-\{0\}$.
Hence $\Phi_{1}(\l) = \Phi_{2}(\l)$ has a meromorphic continuation to $\Re(\l) > -1/2$.
Since $\Phi_{3}(\l)$ is described as an absolutely convergent double integral, 
$\Phi_{3}(\l)$ has an analytic continuation to $\CC$.
We note that the integral $Q_{\chi^{-1}, \rho}^{0}({\bf 1}, \l, -\nu)$ is absolutely convergent and is entire as a function in $\l$.
In order to examine $\Phi_{4}(\l)$, we consider the following residues:
{\allowdisplaybreaks\begin{align*}
\Res_{z = (\nu + 1)/2}f_{\chi^{-1}, \rho}^{\bf 1}(-z, -\nu)
= D_{F}^{-1/2+\nu/2}{\rm N}(\gf_{\chi})^{1/2+\nu}B_{\chi^{-1}, \rho}^{\bf1}(-\nu/2, -\nu)
\frac{L(-\nu, \chi^{-1})}{L(1 - \nu, \chi^{-2})}\d_{\chi, {\bf1}}D_{F}^{1/2}R_{F},
\end{align*}
}
{\allowdisplaybreaks\begin{align*}
\Res_{z = (-\nu + 1)/2}f_{\chi^{-1}, \rho}^{\bf 1}(-z, -\nu)
= D_{F}^{-1/2+\nu/2}{\rm N}(\gf_{\chi})^{1/2+\nu}B_{\chi^{-1}, \rho}^{\bf1}
(\nu/2, -\nu)\frac{L(\nu, \chi^{-1})}{L(1 - \nu, \chi^{-2})}\d_{\chi, {\bf1}}D_{F}^{1/2}R_{F},
\end{align*}
}
{\allowdisplaybreaks\begin{align*}
\Res_{z = (\nu - 1)/2}f_{\chi^{-1}, \rho}^{\bf 1}(-z, -\nu)
= D_{F}^{-1/2+\nu/2}{\rm N}(\gf_{\chi})^{1/2+\nu}B_{\chi^{-1}, \rho}^{\bf1}(1-\nu/2, -\nu)
\frac{L(1-\nu, \chi^{-1})}{L(1 - \nu, \chi^{-2})} (-\d_{\chi, {\bf1}}R_{F}),
\end{align*}
}
{\allowdisplaybreaks\begin{align*}
\Res_{z = (-\nu - 1)/2}f_{\chi^{-1}, \rho}^{\bf 1}(-z, -\nu)
= D_{F}^{-1/2+\nu/2}{\rm N}(\gf_{\chi})^{1/2+\nu}B_{\chi^{-1}, \rho}^{\bf1}(1+\nu/2, -\nu)
\frac{L(1+\nu, \chi)}{L(1 - \nu, \chi^{-2})} (-\d_{\chi, {\bf1}}R_{F}).
\end{align*}
}The functions
$\Res_{z = (\pm \nu \mp 1)/2}f_{\chi^{-1}, \rho}^{\bf 1}(-z, -\nu)$
are holomorphic on $i \RR$ as functions in $\nu$ and
vanish unless $\chi = {\bf 1}$.
Therefore, the integral
{\allowdisplaybreaks\begin{align*}
\int_{i \RR} \left\{ \frac{\b((\nu+1)/2)}{\l + (\nu +1)/2}\Res_{z = (\nu +1)/2} +
\frac{\b((-\nu+1)/2)}{\l + (-\nu +1)/2}\Res_{z = (-\nu +1)/2} \right\} f_{\chi^{-1}, \rho}^{\bf 1}(-z, -\nu) \tilde{\a}_{\chi}(\nu)E_{\chi, \rho}(\nu, g) d\nu
\end{align*}
}is holomorphic on $\Re(\l) > -1/2$.

Consider the integral
$$\int_{i \RR}\bigg\{\frac{\b((\nu-1)/2)}{\l + (\nu -1)/2}\Res_{z = (\nu - 1)/2}
+ \frac{\b((-\nu-1)/2)}{\l + (-\nu - 1)/2}\Res_{z = (-\nu - 1)/2}\bigg\}f_{\chi^{-1}, \rho}^{\bf 1}(-z, -\nu)
 \tilde{\a}_{\chi}(\nu)E_{\chi, \rho}(\nu, g) d\nu.$$
Set $F_{\rho}^{+}(\nu) = \Res_{z = (\nu - 1)/2}
f_{{\bf 1}, \rho}^{\bf 1}(-z, -\nu).$
We note that $F_{\rho}^{+}(\nu)$ is entire.
By taking $c > 1$, we obtain
{\allowdisplaybreaks\begin{align*}
& \int_{i \RR}\frac{\b((\nu-1)/2)}{\l + (\nu -1)/2}\Res_{z = (\nu - 1)/2}
f_{{\bf 1}, \rho}^{\bf 1}(-z, -\nu) \tilde{\a}(\nu)E_{{\bf 1}, \rho}(\nu, g) d\nu \\
=& \int_{L_{c}}\frac{\b((\nu - 1)/2)}{\l + (\nu -1)/2}F_{\rho}^{+}(\nu) \tilde{\a}(\nu)E_{{\bf 1}, \rho}(\nu, g) d\nu
 - 2\pi i \frac{\b(0)}{\l}F_{\rho}^{+}(1)\tilde{\a}(1)\mathfrak{e}_{{\bf 1}, \rho, -1}(g).
\end{align*}
}The first term of the right hand side is holomorphic on
$\Re(\l) > (-c + 1)/2$ and the second term is meromorphic on $\CC$.
Set $F_{\rho}^{-}(\nu) = \Res_{z = (-\nu - 1)/2}
f_{{\bf 1}, \rho}^{\bf 1}(-z, -\nu).$
By the relation $B_{\bf 1, \rho}^{\bf 1}(1 - \nu/2, -\nu) = B_{\bf 1, \rho}^{\bf 1}(1 -\nu/2, \nu)A_{\bf 1, \rho}(-\nu)$,
we have
{\allowdisplaybreaks\begin{align*}
F_{\rho}^{-}(-\nu) D_{F}^{-1/2}A_{{\bf 1}, \rho}(-\nu)\frac{\zeta_{F}(- \nu)}{\zeta_{F}(1 - \nu)} 
= &
B_{{\bf 1}, \rho} ^{\bf 1}(1-\nu/2, \nu)D_{F}^{\nu/2}(-R_{F})D_{F}^{-1/2}A_{{\bf 1}, \rho}(-\nu)
=  F_{\rho}^{+}(\nu),
\end{align*}
}
and hence, we obtain
{\allowdisplaybreaks\begin{align*}
& \int_{i \RR}
\frac{\b((-\nu-1)/2)}{\l + (-\nu - 1)/2} F_{\rho}^{-}(\nu) \tilde{\a}(\nu)E_{{\bf 1}, \rho}(\nu, g) d\nu
= \int_{i \RR}
\frac{\b((\nu-1)/2)}{\l + (\nu - 1)/2}
F_{\rho}^{-}(-\nu)
\tilde{\a}(\nu)E_{{\bf 1}, \rho}(-\nu, g) d\nu \\
= & \int_{i \RR}
\frac{\b((\nu-1)/2)}{\l + (\nu - 1)/2}
F_{\rho}^{-}(-\nu)
D_{F}^{-1/2}A_{{\bf 1}, \rho}(-\nu)\frac{\zeta_{F}(- \nu)}{\zeta_{F}(1 - \nu)}
\tilde{\a}(\nu)E_{{\bf 1}, \rho}(\nu, g) d\nu \\
= & \int_{i \RR}
\frac{\b((\nu-1)/2)}{\l + (\nu - 1)/2}
F_{\rho}^{+}(\nu)
\tilde{\a}(\nu)E_{{\bf 1}, \rho}(\nu, g) d\nu \\
= & \int_{L_{c}}
\frac{\b((\nu-1)/2)}{\l + (\nu - 1)/2}
F_{\rho}^{+}(\nu)
\tilde{\a}(\nu)E_{{\bf 1}, \rho}(\nu, g) d\nu
- 2\pi i\frac{\b(0)}{\l}F_{\rho}^{+}(1)
\tilde{\a}(1)\mathfrak{e}_{{\bf 1}, \rho, -1}(g).
\end{align*}
}
Then, in the last line of the equalities above,
the first term is holomorphic on $\Re(\l) > (-c+1)/2$ and
the second term is meromorphic on $\CC$.
Hence we can prove that $\Phi_{4}(\l)$ has a meromorphic continuation to $\Re(\l)>-1/2$.
This gives us a meromorphic continuation of $\Psi_{\rm ct}(\l)$ to $\Re(\l) > -1/2$.
\QED

\begin{lem}
We have
{\allowdisplaybreaks\begin{align*}
{\rm CT}_{\l = 0}\hat{{\bf \Psi}}_{\b, \l}(\gn|\a; g) = & C(\gn, S)
\bigg\{\sum_{\varphi \in \Bcal_{\rm cus}(\gn)}\alpha(\nu_{\varphi, S})\overline{P_{\rm reg}^{\bf 1}(\varphi)}\varphi(g) \\
& + \sum_{\chi \in \Xi(\gn)} \sum_{\rho \in \Lambda_{\chi}(\gn)}
\frac{R_{F}^{-1}}{8 \pi i} \int_{i\RR} \tilde{\alpha}_{\chi}(\nu)P_{\rm reg}^{\bf 1}(
E_{\chi^{-1}, \rho}(-\nu, -))E_{\chi, \rho}(\nu, g)d\nu \\
& + \sum_{\rho \in \Lambda(\gn)} \{f_{{\bf 1}, \rho}^{(0)}(e)+ \d(S(\rho) = \emptyset))\}
\big\{\tilde{\alpha}'(1){\mathfrak e}_{{\bf 1}, \rho, -1}(g)
+ \tilde{\alpha}(1){\mathfrak e}_{{\bf 1}, \rho, 0}(g)\}
\bigg\}\b(0).
\end{align*}
}
\end{lem}
\Proof
In the proof of Lemma \ref{lem:Green merom}, we gave the constant terms of the cuspidal
and residual parts at $\l=0$.
Therefore, it is enough to evaluate the constant term of the Eisenstein part $\Psi_{\rm ct}(\l) = 2\Phi_{1}(\l) + \Phi_{3}(\l) + \Phi_{4}(\l)$.
By the residue theorem, we have
{\allowdisplaybreaks\begin{align*}
{\rm CT}_{\l = 0}\Phi_{1}(\l) 
= & \frac{1}{8 \pi i}\sum_{\rho \in \Lambda(\gn)}f_{{\bf1}, \rho}^{(-\nu)}(e)\int_{i \RR}\tilde{\a}(\nu)
\b((\nu-1)/2) \frac{-1}{ (\nu - 1)/2} E_{{\bf 1}, \rho}(\nu; g)d\nu \\
& + \frac{1}{8 \pi i}\sum_{\rho \in \Lambda(\gn)}f_{{\bf1}, \rho}^{(-\nu)}(e)
 \int_{L_{c}}\tilde{\a}(\nu)
\b((\nu-1)/2) \frac{1}{(\nu - 1)/2}
E_{{\bf 1}, \rho}(\nu; g)d\nu \\
= & \frac{1}{8 \pi i}\sum_{\rho \in \Lambda(\gn)}f_{{\bf1}, \rho}^{(-\nu)}(e)2\pi i \Res_{\nu=1}\{\tilde{\a}(\nu)\b((\nu-1)/2)\frac{1}{(\nu -1)/2}E_{{\bf 1}, \rho}(\nu; g)\} \\
= & \frac{1}{2}\sum_{\rho \in \Lambda(\gn)}f_{{\bf1}, \rho}^{(-\nu)}(e)
\{ \tilde{\a}'(1) \mathfrak{e}_{{\bf 1}, \rho, -1}(g) + \tilde{\a}(1) \mathfrak{e}_{{\bf 1}, \rho, 0}(g)\} \b(0)
\end{align*}
}and the integral $Q_{\chi^{-1}, \rho}^{0}({\bf 1}, 0, -\nu)$ is written as
{\allowdisplaybreaks\begin{align*}
Q_{\chi^{-1}, \rho}^{0}({\bf 1}, 0, -\nu) = & \frac{1}{2 \pi i}\int_{L_{\s}} \{f_{\chi^{-1}, \rho}^{\bf 1}(z, - \nu) + f_{\chi^{-1}, \rho}^{\bf 1}(-z, - \nu) \} \frac{\b(z)}{z}dz \\
= &\ f_{\chi^{-1}, \rho}^{\bf 1}(0, - \nu)\b(0)  + \{ \Res_{z = (1+\nu)/2} + \Res_{z = (1-\nu)/2}
+ \Res_{z = (-1+\nu)/2}\\
& + \Res_{z = (-1-\nu)/2} \} \bigg\{ f_{\chi^{-1}, \rho}^{\bf 1}(z, -\nu) \frac{\b(z)}{z} \bigg\}.
\end{align*}
}Thus the constant term of $\Phi_{3}(\l)$ is evaluated
as
{\allowdisplaybreaks\begin{align*}
& {\rm CT}_{\l = 0}\Phi_{3}(\l) \\
= & \frac{R_{F}^{-1}}{8 \pi i}\sum_{\chi \in \Xi(\gn)}\sum_{\rho \in \Lambda_{\chi}(\gn)}\int_{i \RR}\tilde{\a}_{\chi}(\nu)Q_{\chi^{-1}, \rho}^{0}({\bf 1}, 0, -\nu)E_{\chi, \rho}(\nu, g) d\nu \\
= & \frac{R_{F}^{-1}}{8 \pi i} \sum_{\chi \in \Xi(\gn)}\sum_{\rho \in \Lambda_{\chi}(\gn)}  \bigg\{ \int_{i \RR}\tilde{\a}_{\chi}(\nu)
f_{\chi^{-1}, \rho}^{\bf 1}(0, - \nu) E_{\chi, \rho}(\nu, g) d\nu \b(0)
 + \int_{i \RR}\{ \Res_{z = (1+\nu)/2} \\
&+ \Res_{z = (1-\nu)/2} + \Res_{z = (-1+\nu)/2} + \Res_{z = (-1-\nu)/2} \} \bigg\{ f_{\chi^{-1}, \rho}^{\bf 1}(z, -\nu) \frac{\b(z)}{z} \bigg\} \tilde{\a}_{\chi}(\nu)E_{\chi, \rho}(\nu, g) d\nu \bigg\}.
\end{align*}
}We examine the constant term of $\Phi_{4}(\l)$. By the expression of $\Phi_{4}(\l)$ given in the proof of Lemma \ref{lem:Green merom}, we have
{\allowdisplaybreaks\begin{align*}
& {\rm CT}_{\l = 0}\Phi_{4}(\l) \\
= & - \frac{R_{F}^{-1}}{8 \pi i} \sum_{\chi \in \Xi(\gn)}\sum_{\rho \in \Lambda_{\chi}(\gn)} \int_{i \RR} \bigg \{ \frac{\b((\nu+1)/2)}{(\nu +1)/2}\Res_{z = (\nu +1)/2} +
\frac{\b((-\nu+1)/2)}{(-\nu +1)/2}\Res_{z = (-\nu +1)/2} \bigg\} f_{\chi^{-1}, \rho}^{\bf 1}(-z, -\nu) \\
& \times \tilde{\a}_{\chi}(\nu)E_{\chi, \rho}(\nu, g) d\nu \\
& - 2 \times \frac{R_{F}^{-1}}{8 \pi i} \sum_{\rho \in \Lambda(\gn)} \bigg\{ \int_{L_{c}} \frac{\b((\nu - 1)/2)}{ (\nu -1)/2} F_{\rho}^{+}(\nu) \tilde{\a}(\nu)E_{{\bf 1}, \rho}(\nu, g) d\nu
\bigg\}.
\end{align*}
}Therefore we obtain
{\allowdisplaybreaks\begin{align*}
& {\rm CT}_{\l = 0}\{ \Phi_{3}(\l) + \Phi_{4}(\l) \} \\
= & \frac{R_{F}^{-1}}{8 \pi i} \sum_{\chi \in \Xi(\gn)}\sum_{\rho \in \Lambda_{\chi}(\gn)}  \bigg\{ \int_{i \RR}\tilde{\a}_{\chi}(\nu)
f_{\chi^{-1}, \rho}^{\bf 1}(0, - \nu) E_{\chi, \rho}(\nu, g) d\nu \b(0) \\
& + \int_{i \RR}
\{ \Res_{z = (-1+\nu)/2} + \Res_{z = (-1-\nu)/2} \} \bigg\{ f_{\chi^{-1}, \rho}^{\bf 1}(z, -\nu) \frac{\b(z)}{z} \bigg\} \tilde{\a}_{\chi}(\nu)E_{\chi, \rho}(\nu, g) d\nu \bigg\} \\
& - 2 \times \frac{R_{F}^{-1}}{8 \pi i} \sum_{\rho  \in \Lambda(\gn)} \bigg\{ \int_{L_{c}} \frac{\b((\nu - 1)/2)}{ (\nu -1)/2} F_{\rho}^{+}(\nu) \tilde{\a}(\nu)E_{{\bf 1}, \rho}(\nu, g) d\nu
\bigg\}.
\end{align*}
}By noting the relation
{\allowdisplaybreaks\begin{align*}
& \int_{i \RR}
F_{\rho}^{+}(\nu)\frac{\b((\nu - 1)/2)}{ (\nu -1)/2} \tilde{\a}(\nu)E_{\bf 1, \rho}(\nu, g) d\nu
 = \int_{i \RR}
F_{\rho}^{-}(\nu)\frac{\b((-\nu - 1)/2)}{ (-\nu -1)/2} \tilde{\a}(\nu)E_{\bf 1, \rho}(\nu, g) d\nu,
\end{align*}
}we have
{\allowdisplaybreaks\begin{align*}
& \frac{R_{F}^{-1}}{8 \pi i} \sum_{\rho \in \Lambda(\gn)}
\int_{i \RR}
\{ \Res_{z = (-1+\nu)/2} + \Res_{z = (-1-\nu)/2} \} \bigg\{ f_{{\bf 1}, \rho}^{\bf 1}(z, -\nu) \frac{\b(z)}{z} \bigg\} \tilde{\a}(\nu)E_{{\bf1}, \rho}(\nu, g) d\nu \\
& - 2 \times \frac{R_{F}^{-1}}{8 \pi i} \sum_{\rho \in \Lambda(\gn)} \int_{L_{c}} \frac{\b((\nu - 1)/2)}{ (\nu -1)/2} F_{\rho}^{+}(\nu) \tilde{\a}(\nu)E_{{\bf 1}, \rho}(\nu, g) d\nu \\
 = & \frac{R_{F}^{-1}}{8 \pi i} \sum_{\rho \in \Lambda(\gn)} \int_{i \RR}
\bigg\{F_{\rho}^{+}(\nu)\frac{\b((\nu - 1)/2)}{ (\nu -1)/2} + F_{\rho}^{-}(\nu)\frac{\b((-\nu - 1)/2)}{ (-\nu -1)/2} \bigg\} \tilde{\a}(\nu)E_{\bf 1, \rho}(\nu, g) d\nu \\
& - 2 \times \frac{R_{F}^{-1}}{8 \pi i} \sum_{\rho \in \Lambda(\gn)} \int_{L_{c}} \frac{\b((\nu - 1)/2)}{ (\nu -1)/2} F_{\rho}^{+}(\nu) \tilde{\a}(\nu)E_{{\bf 1}, \rho}(\nu, g) d\nu \\
 = & 2 \times \frac{R_{F}^{-1}}{8 \pi i} \sum_{\rho \in \Lambda(\gn)} \int_{i \RR}
F_{\rho}^{+}(\nu)\frac{\b((\nu - 1)/2)}{ (\nu -1)/2} \tilde{\a}(\nu)E_{\bf 1, \rho}(\nu, g) d\nu \\
& - 2 \times \frac{R_{F}^{-1}}{8 \pi i} \sum_{\rho \in \Lambda(\gn)} \int_{L_{c}} \frac{\b((\nu - 1)/2)}{ (\nu -1)/2} F_{\rho}^{+}(\nu) \tilde{\a}(\nu)E_{{\bf 1}, \rho}(\nu, g) d\nu
\\
= & \frac{R_{F}^{-1}}{4 \pi i} \sum_{\rho \in \Lambda(\gn)} (- 2 \pi i)\Res_{\nu = 1}
\bigg\{\frac{\b((\nu - 1)/2)}{ (\nu -1)/2} F_{\rho}^{+}(\nu) \tilde{\a}(\nu)E_{{\bf 1}, \rho}(\nu, g)\bigg\}.
\end{align*}
}Here the residue is expressed as
{\allowdisplaybreaks\begin{align*}
& \Res_{\nu = 1}
\left\{ \frac{\b((\nu - 1)/2)}{ (\nu -1)/2} F_{\rho}^{+}(\nu) \tilde{\a}(\nu)E_{{\bf 1}, \rho}(\nu, g) \right\} \\
= &\Res_{\nu = 1} \left\{ \frac{\b((\nu - 1)/2)}{ (\nu -1)/2}
\tilde{\a}(\nu)E_{{\bf 1}, \rho}(\nu, g)
D_{F}^{-1/2+\nu/2}B_{\chi^{-1}, \rho}^{\bf1}(1-\nu/2, -\nu) (-R_{F}) \right\} \\
= &
\{ 2\tilde{\a}'(1) \mathfrak{e}_{{\bf 1}, \rho, -1}(g) + 2\tilde{\a}(1) \mathfrak{e}_{{\bf 1}, \rho, 0}(g)\}
 \d(S(\rho)=\emptyset)
(-R_{F}) \b(0).
\end{align*}
}We note $D_{F}^{(\nu-1)/2}B_{\eta, \rho}^{\eta}(1 - \nu/2, -\nu) = \tilde{\eta}(\gD_{F/\QQ})\d(S(\rho)=\emptyset)
$ for any $\eta \in \Xi_{0}(\go_{F})$ satisfying $\eta^{2}={\bf 1}$.
Therefore we obtain
{\allowdisplaybreaks\begin{align*}
& \frac{R_{F}^{-1}}{8 \pi i} \sum_{\rho \in \Lambda(\gn)}
\int_{i \RR}
\{ \Res_{z = (-1+\nu)/2} + \Res_{z = (-1-\nu)/2} \} \bigg\{ f_{{\bf 1}, \rho}^{\bf 1}(z, -\nu) \frac{\b(z)}{z} \bigg\} \tilde{\a}(\nu)E_{\bf 1, \rho}(\nu, g) d\nu \\
= & \frac{R_{F}^{-1}}{4 \pi i} \sum_{\rho \in \Lambda(\gn)}  2 \pi i
\{ 2\tilde{\a}'(1) \mathfrak{e}_{{\bf 1}, \rho, -1}(g) + 2\tilde{\a}(1) \mathfrak{e}_{{\bf 1}, \rho, 0}(g)\}
 \d(S(\rho)=\emptyset) R_{F} \b(0) \\
= & \sum_{\rho \in \Lambda(\gn)} \d(S(\rho)=\emptyset)
\{ \tilde{\a}'(1) \mathfrak{e}_{{\bf 1}, \rho, -1}(g) + \tilde{\a}(1) \mathfrak{e}_{{\bf 1}, \rho, 0}(g)\} \b(0),
\end{align*}
}and hence
{\allowdisplaybreaks\begin{align*}
{\rm CT}_{\l = 0}\Psi_{\rm ct}(\l)
= & \frac{R_{F}^{-1}}{8 \pi i} \sum_{\chi \in \Xi(\gn)}\sum_{\rho \in \Lambda_{\chi}(\gn)}
\int_{i \RR}\tilde{\a}_{\chi}(\nu)
f_{\chi^{-1}, \rho}^{\bf 1}(0, - \nu) E_{\chi, \rho}(\nu, g) d\nu \b(0) \\
& + \sum_{\rho \in \Lambda(\gn)}
\{ f_{{\bf1}, \rho}^{(0)}(e)+\d(S(\rho)=\emptyset) \}
\{ \tilde{\a}'(1) \mathfrak{e}_{{\bf 1}, \rho, -1}(g) + \tilde{\a}(1) \mathfrak{e}_{{\bf 1}, \rho, 0}(g)\} \b(0).
\end{align*}
}This gives the expression of ${\rm CT}_{\l = 0}\hat{\bf\Psi}_{\b, \l}(\gn|\a;g)$.
\QED

Then, we define the regularized smoothed kernel $\hat{{\bf \Psi}}_{\rm reg}(\gn|\a; g)$
by the relation
$${\rm CT}_{\l = 0}\hat{\bf\Psi}_{\b, \l}(\gn|\a;g) = \hat{{\bf \Psi}}_{\rm reg}(\gn|\a; g)\b(0), \hspace{4mm} \b \in \Bcal.$$

Let $\gS$ be a Siegel set of $G_{\AA}$ such that $G_{\AA}= G_{F}\gS$.
We have the following estimate of $\hat{{\bf \Psi}}_{\rm reg}(\gn|\a; g)$.
\begin{lem}\label{lem: estimation of spectral decomposition}
There exists $N \in \NN$ such that
for any $m \in \NN$,
the following estimates hold for any $g \in G_{\AA}^{1}\cap \gS$ uniformly.
\begin{enumerate}\label{esti of Green}
\item $\sum_{\varphi \in \Bcal_{\rm cus}(\gn)}|\alpha(\nu_{\varphi, S})\overline{P_{\rm reg}^{\bf 1}(\varphi)}\varphi(g)| \ll y(g)^{-m}$,

\item $$\sum_{\chi \in \Xi(\gn)} \sum_{\rho \in \Lambda_{\chi}(\gn)}
\frac{R_{F}^{-1}}{8 \pi i} \int_{i\RR} |\tilde{\alpha}_{\chi}(\nu)P_{\rm reg}^{\bf 1}(
E_{\chi^{-1}, \rho}(-\nu, -))E_{\chi, \rho}(\nu, g)||d\nu| \ll y(g)^{N},$$

\item $|{\mathfrak e}_{{\bf 1}, \rho, 0}(g)| + |{\mathfrak e}_{{\bf 1}, \rho, -1}(g)| \ll y(g)$,

\item $|\hat{{\bf \Psi}}_{\rm reg}(\gn|\a; g)| \ll y(g)^{N}$.
\end{enumerate}
\end{lem}
\Proof
There exists a positive constant $C$ such that $L_{\fin}(s, \chi)$ does not vanish for any
non-quadratic character $\chi $ of $F^{\times} \backslash \AA^{\times}$
if $\Re(s) \ge 1-C/\log \{ \gq(\chi) (3+|\Im(s)|)\}$ (cf. \cite[Theorem 5.10]{Iwaniec-Kowalski}).
Hence, by virtue of the proof of \cite[Theorem 3.11]{Titchmarsh}, the estimate
$$\frac{1}{|L_{\fin}(1, \chi)|}\ll \log \mathfrak{q}(\chi)$$
holds uniformly for non-quadratic characters $\chi$ of $F^{\times} \backslash \AA^{\times}$. 
Next we give a generalized Siegel's theorem for quadratic characters of $F^{\times} \backslash \AA^{\times}$. By \cite[Theorem 2.3.1]{Molteni},
for any $\e >0$, the estimate
$$|L_{\fin}(1, \chi)| \gg \mathfrak{q}(\chi)^{-\e}$$
holds uniformly for quadratic characters $\chi$ of $F^{\times} \backslash \AA^{\times}$. Indeed, \cite[Theorem 2.3.1]{Molteni} works for general $L$-functions over $F$
in the sense of \cite{Perelli}.

As a consequence, we have the estimate
$$\frac{1}{|L_{\fin}(1+\nu, \chi^{2})|} \ll \mathfrak{q}(\chi^{2}|\cdot|_{\AA}^{\nu})^{\e}, \ \nu \in i \RR$$
with the implied constant independent of $\chi \in \Xi(\gn)$ and $\gn$.
Combining this with the argument of the proof of \cite[Lemma 9.9]{TsuzukiSpec}, we have the assertions.
\QED

\section{Periods of regularized automorphic smoothed kernels: spectral side}
\label{Periods of regularized automorphic smoothed kernels: spectral side}
By $(4)$ in Lemma \ref{esti of Green},
the integral $P_{\b, \l}^{\eta}(\hat{{\bf \Psi}}_{\rm reg}(\gn|\a; g))$ converges absolutely for $\Re(\l) > N$ and is holomorphic on $\Re(\l) > N$.
We have the following expression in the same way as \cite[Lemma 10.1]{TsuzukiSpec}.
\begin{lem}\label{spectral exp of reg Green}
For $\Re(\l) > N$, we have the expression
{\allowdisplaybreaks\begin{align*}
P_{\b, \l}^{\eta}(\hat{{\bf \Psi}}_{\rm reg}(\gn|\a; -))
= C(\gn, S)\{\PP_{\rm cus}^{\eta}(\b, \l, \a) + \PP_{\rm eis}^{\eta}(\b, \l, \a) + \PP_{\rm res}^{\eta}(\b, \l, \a)\},
\end{align*}
}where
$$\PP_{\rm cus}^{\eta}(\b, \l, \a) = \sum_{\varphi \in \Bcal_{\rm cus}(\gn)}\alpha(\nu_{\varphi, S})\overline{P_{\rm reg}^{\bf 1}(\varphi)}P_{\b, \l}^{\eta}(\varphi),$$
$$\PP_{\rm eis}^{\eta}(\b, \l, \a) = \sum_{\chi \in \Xi(\gn)} \sum_{\rho \in \Lambda_{\chi}(\gn)}
\frac{R_{F}^{-1}}{8 \pi i} \int_{i\RR} \tilde{\alpha}_{\chi}(\nu)P_{\rm reg}^{\bf 1}(
E_{\chi^{-1}, \rho}(-\nu, -))P_{\b, \l}^{\eta}(E_{\chi, \rho}(\nu, -))d\nu$$
and
$$\PP_{\rm res}^{\eta}(\b, \l, \a) = \sum_{\rho \in \Lambda(\gn)} \{f_{{\bf 1}, \rho}^{(0)}(e) + \d(S(\rho)=\emptyset)\}(\tilde{\alpha}'(1)P_{\b, \l}^{\eta}({\mathfrak e}_{{\bf 1}, \rho, -1}) + \tilde{\alpha}(1)P_{\b, \l}^{\eta}({\mathfrak e}_{{\bf 1}, \rho, 0})).$$
Here the series converge absolutely and locally uniformly on $\Re(\l) > N$.
\end{lem}
The value $\PP_{\rm res}^{\eta}(\b, \l, \a)$ is described by Propositions \ref{prop:e-1 = chi det} and \ref{prop:P(e_{chi, rho, 0})}. Then we have the following.
\begin{lem}\label{lem: merom of P res}
The function $\l \mapsto \PP_{\rm res}^{\eta}(\b, \l, \a)$ on $\Re(\l) > N$ is analytically continued to a meromorphic function on $\CC$. Its constant term at $\l = 0$ is given by
{\allowdisplaybreaks\begin{align*}
{\rm CT}_{\l=0}\PP_{\rm res}^{\eta}(\b, \l, \a)
= & \sum_{\rho \in \Lambda(\gn)}\{f_{{\bf 1}, \rho}^{(0)}(e) + \d(S(\rho)=\emptyset)\} \tilde{\alpha}(1)
\frac{\Gcal(\eta)D_{F}^{-1/2}}{\zeta_{F}(2)} \\
& \times \bigg\{
- \frac{1}{2}\d_{\eta, {\bf 1}}\tilde{B}_{{\bf 1}, \rho}^{\bf 1}(0)R_{F}^{2}\b^{\prime\prime}(0) 
+ a_{{\bf 1}, \rho}^{\eta}(0)\b(0)\bigg\}.
\end{align*}
}Here
$\tilde{B}^{\eta}_{\chi, \rho}(z) = \e(-z, \chi^{-1}\eta)B_{\chi, \rho}^{\eta}(-z+1/2, 1)$ and
{\allowdisplaybreaks\begin{align*}
a_{{\bf 1}, \rho}^{\eta}(0) = &
 - \frac{1}{2}(\tilde{B}_{{\bf 1}, \rho}^{\bf 1})^{\prime\prime}(0)\d_{\eta, {\bf 1}}R_{F}^{2}
- 2\tilde{B}_{{\bf 1}, \rho}^{\bf 1}(0)R_{F}C_{1}({\bf1})\d_{\eta, {\bf 1}}
+ \tilde{B}_{{\bf 1}, \rho}^{\eta}(0)C_{0}(\eta)^{2}.
\end{align*}
}
\end{lem}

\begin{lem}\label{lem: merom of P eisen}
The function $\l \mapsto \PP_{\rm eis}^{\eta}(\b, \l, \a)$ on $\Re(\l)>N$ is analytically continued to a meromorphic function on $\Re(\l) > -1/2$.
\end{lem}
\Proof
By Proposition \ref{prop:Z* of Eisen}, we have
{\allowdisplaybreaks\begin{align*}
Z^{*}(s, \eta, E_{\chi, \rho}^{\natural}(\nu,-)) = &
\Gcal(\eta)D_{F}^{-\nu/2}
{\rm N}(\gf_{\chi})^{1/2-\nu}
B_{\chi, \rho}^{\eta}(s, \nu)\frac{L(s + \nu/2, \chi \eta)L(s - \nu/2, \chi^{-1} \eta)}{L(1+\nu, \chi^{2})}.
\end{align*}
}Set
$$\mathfrak{L}_{\chi, \rho}^{\eta}(\nu) = D_{F}^{\nu/2}
{\rm N}(\gf_{\chi})^{1/2+\nu}B_{\chi^{-1}, \rho}^{\eta}(1/2, -\nu)\frac{L((1 + \nu)/2, \chi \eta)L((1 - \nu)/2, \chi^{-1} \eta)}{L(1-\nu, \chi^{-2})}
$$
and recall the expression
{\allowdisplaybreaks\begin{align*}
P_{\b,\l}^{\eta}(E_{\chi, \rho}(\nu, -)) = & P_{\chi}(\eta, \l, \nu) + D_{F}^{-1/2}A_{\chi, \rho}(\nu)\frac{L(\nu, \chi^{2})}{L(1+ \nu, \chi^{2})}P_{\chi^{-1}}(\eta, \l, -\nu)
 + Q_{\chi, \rho}^{+}(\eta, \l, \nu) + Q_{\chi, \rho}^{-}(\eta, \l, \nu).
\end{align*}
}We remark
{\allowdisplaybreaks\begin{align*}
\PP_{\rm eis}^{\eta}(\b, \l, \a) = & \sum_{\chi \in \Xi(\gn)} \sum_{\rho \in \Lambda_{\chi}(\gn)}
\frac{R_{F}^{-1}}{8 \pi i} \int_{i\RR} \tilde{\alpha}_{\chi}(\nu)
\Gcal({\bf 1}) \mathfrak{L}_{\chi, \rho}^{\bf 1}(\nu)
\bigg\{P_{\chi}(\eta, \l, \nu) \\
& + D_{F}^{-1/2}A_{\chi, \rho}(\nu)\frac{L(\nu, \chi^{2})}{L(1 + \nu, \chi^{2})}
P_{\chi^{-1}}(\eta, \l, -\nu)
+ Q_{\chi, \rho}^{0}(\eta, \l, \nu) \\
& - \sum_{a = (\pm \nu \pm 1)/2}\frac{\b(a)}{\l + a} \Res_{z = a}\{ f_{\chi, \rho}^{\eta}(-z, \nu) \} \bigg\} d\nu.
\end{align*}
}In order to examine $\PP_{\rm eis}^{\eta}(\b, \l, \a)$, we decompose this into the following four terms:
{\allowdisplaybreaks\begin{align*}
\Phi_{1}^{+}(\l) = &
\sum_{\chi \in \Xi(\gn)} \sum_{\rho \in \Lambda_{\chi}(\gn)}
\frac{R_{F}^{-1}}{8 \pi i} \int_{i\RR} \tilde{\alpha}_{\chi}(\nu)
D_{F}^{-1/2} \mathfrak{L}_{\chi, \rho}^{\bf 1}(\nu)
f_{\chi, \rho}^{(0)}(e) \\
& \times \d_{\chi, \eta}R_{F}\bigg\{
\frac{1}{\l - (\nu + 1)/2} +
\frac{1}{\l + (\nu + 1)/2}
\bigg\} \b((\nu + 1)/2) d \nu,
\end{align*}
}
{\allowdisplaybreaks\begin{align*}
\Phi_{1}^{-}(\l) = &
\sum_{\chi \in \Xi(\gn)} \sum_{\rho \in \Lambda_{\chi}(\gn)}
\frac{R_{F}^{-1}}{8 \pi i} \int_{i\RR} \tilde{\alpha}_{\chi}(\nu)
D_{F}^{-1/2} \mathfrak{L}_{\chi, \rho}^{\bf 1}(\nu)
D_{F}^{-1/2}A_{\chi, \rho}(\nu)\frac{L(\nu, \chi^{2})}{L(1 + \nu, \chi^{2})}
f_{\chi^{-1}, \rho}^{(0)}(e) \\
& \times \d_{\chi, \eta}R_{F} \bigg\{
\frac{1}{\l - (-\nu + 1)/2} +
\frac{1}{\l + (-\nu + 1)/2}
\bigg\} \b((-\nu + 1)/2) d \nu,
\end{align*}
}
{\allowdisplaybreaks\begin{align*}\Phi_{2}(\l) =
\sum_{\chi \in \Xi(\gn)} \sum_{\rho \in \Lambda_{\chi}(\gn)}
\frac{R_{F}^{-1}}{8 \pi i} \int_{i\RR} \tilde{\alpha}_{\chi}(\nu)
D_{F}^{-1/2} \mathfrak{L}_{\chi, \rho}^{\bf 1}(\nu)
Q_{\chi, \rho}^{0}(\eta, \l, \nu)d \nu,
\end{align*}
}
{\allowdisplaybreaks\begin{align*}
\Phi_{3}(\l) = &
- \sum_{\chi \in \Xi(\gn)} \sum_{\rho \in \Lambda_{\chi}(\gn)}
\frac{R_{F}^{-1}}{8 \pi i} \int_{i\RR} \tilde{\alpha}_{\chi}(\nu)
D_{F}^{-1/2} \mathfrak{L}_{\chi, \rho}^{\bf 1}(\nu)
\sum_{a = (\pm \nu \pm 1)/2}\frac{\b(a)}{\l + a} \Res_{z = a}\{ f_{\chi, \rho}^{\eta}(-z, \nu)\}
d \nu.
\end{align*}
}When $\chi = \eta$, by using the functional equations
$$
\mathfrak{L}_{\eta, \rho}^{\bf 1}(\nu)
D_{F}^{-1/2}A_{\eta, \rho}(\nu)\frac{\zeta_{F}(\nu)}{\zeta_{F}(1 + \nu)}
=
\mathfrak{L}_{\eta, \rho}^{\bf 1}(-\nu)
$$
and $B_{\eta, \rho}^{\bf 1}(1/2, \nu) = B_{\eta, \rho}^{\bf 1}(1/2, -\nu)A_{\eta, \rho}(\nu)$,
we obtain
$\Phi_{1}^{+}(\l) = \Phi_{1}^{-}(\l).$
The term $\Phi_{1}^{+}(\l)$ is expressed as
{\allowdisplaybreaks\begin{align*}
\Phi_{1}^{+}(\l)= &
\sum_{\chi \in \Xi(\gn)} \sum_{\rho \in \Lambda_{\chi}(\gn)}
\bigg\{ \frac{R_{F}^{-1}}{8 \pi i} \int_{i\RR} \tilde{\alpha}_{\chi}(\nu)
D_{F}^{-1/2} \mathfrak{L}_{\chi, \rho}^{\bf 1}(\nu)
f_{\chi, \rho}^{(0)}(e)\d_{\chi, \eta}R_{F}
\frac{1}{\l + (\nu + 1)/2}
\b((\nu + 1)/2) d \nu \\
& +
\frac{R_{F}^{-1}}{8 \pi i} \int_{i\RR} \tilde{\alpha}_{\chi}(\nu)
D_{F}^{-1/2} \mathfrak{L}_{\chi, \rho}^{\bf 1}(\nu)
f_{\chi, \rho}^{(0)}(e)\d_{\chi, \eta}R_{F}
\frac{1}{\l - (\nu + 1)/2}
\b((\nu + 1)/2) d \nu \bigg\}.
\end{align*}
}Then the first term in the summation is holomorphic on $\Re(\l) > -1/2$.
For any fixed $\s > 1$, the second term in the summation is transformed into
{\allowdisplaybreaks\begin{align*}
& \frac{R_{F}^{-1}}{8 \pi i}D_{F}^{-1/2} \bigg\{ \int_{L_{-\s}} \tilde{\alpha}_{\chi}(\nu)
\mathfrak{L}_{\chi, \rho}^{\bf 1}(\nu)
f_{\chi, \rho}^{(0)}(e)\d_{\chi, \eta}R_{F}
\frac{1}{\l - (\nu + 1)/2}
\b((\nu + 1)/2) d \nu \\
& + \d_{\chi, \eta}\d(\gf_{\eta} = \go_{F}) 2\pi i \Res_{\nu = -1}\left(\frac{\b((\nu+1)/2)}{\l -(\nu+1)/2}\tilde{\alpha}_{\eta}(\nu)\mathfrak{L}_{\eta, \rho}^{\bf 1}(\nu)\right) f_{\chi, \rho}^{(0)}(e)R_{F}
\bigg\}.
\end{align*}
}The first term in the expression above is meromorphic on $\Re(\l) > (-\s +1)/2$. In order to prove the meromorphicity of the second term in the expression above,
we put
$$D_{F}^{\nu/2}\frac{L((1+\nu)/2, \eta)L((1-\nu)/2, \eta)}{\zeta_{F}(1-\nu)}
= \frac{D_{-2}^{\eta}}{(\nu+1)^{2}} + \frac{D_{-1}^{\eta}}{\nu+1} + D_{0}^{\eta} + \Ocal((\nu+1)), \ (\nu \rightarrow -1),$$
$$B_{\eta, \rho}^{\bf 1}(1/2, -\nu) = p_{0}^{\eta}(\rho) + p_{1}^{\eta}(\rho)(\nu + 1) +
p_{2}^{\eta}(\rho)(\nu+1)^{2} + \Ocal((\nu+1)^{3}), \ (\nu \rightarrow -1)$$
and
$$\frac{\b((\nu+1)/2)}{\l -(\nu+1)/2}\tilde{\alpha}_{\eta}(\nu) = q_{0}^{\eta}(\l) + q_{1}^{\eta}(\l)(\nu + 1) + \Ocal((\nu+1)^{2}), \ (\nu \rightarrow -1).$$
Then these give the following expressions:
$$\Res_{\nu = -1} \bigg\{ \frac{\b((\nu+1)/2)}{\l -(\nu+1)/2}\tilde{\alpha}_{\eta}(\nu)\mathfrak{L}_{\eta, \rho}^{\bf 1}(\nu) \bigg\} = p_{0}^{\eta}(\rho)q_{1}^{\eta}(\l)D_{-2}^{\eta}
+ p_{0}^{\eta}(\rho)q_{0}^{\eta}(\l)D_{-1}^{\eta}
+ p_{1}^{\eta}(\rho)q_{0}^{\eta}(\l)D_{-2}^{\eta},
$$
$$q_{0}^{\eta}(\l) = \frac{\tilde{\a}_{\eta}(1)\b(0)}{\l}, \ \ q_{1}^{\eta}(\l) = \left(\frac{\tilde{\a}_{\eta}'(1)}{\l} + \frac{\tilde{\a}_{\eta}(1)}{2\l^{2}}\right) \b(0).$$
Therefore $\Phi_{1}^{+}(\l) = \Phi_{1}^{-}(\l)$ has a meromorphic continuation to $\Re(\l)> -1/2$.
Since $\Phi_{2}(\l)$ is described as an absolutely convergent double integral, $\Phi_{2}(\l)$ is entire.

We examine $\Phi_{3}(\l)$. This is written as
{\allowdisplaybreaks\begin{align*}
\Phi_{3}(\l) = & -\sum_{\chi \in \Xi(\gn)} \sum_{\rho \in \Lambda_{\chi}(\gn)}\frac{R_{F}^{-1}}{8\pi i} \bigg\{ \int_{i \RR}\tilde{\a}_{\chi}(\nu)
D_{F}^{-1/2}\mathfrak{L}_{\chi, \rho}^{\bf 1}(\nu)
\sum_{a = (\pm \nu + 1)/2}\frac{\b(a)}{\l+a}\Res_{z=a}f_{\chi, \rho}^{\eta}(-z, \nu) d\nu \\
& + \int_{i \RR}\tilde{\a}_{\chi}(\nu)
D_{F}^{-1/2}\mathfrak{L}_{\chi, \rho}^{\bf 1}(\nu)
\sum_{a = (\pm \nu - 1)/2}\frac{\b(a)}{\l+a}\Res_{z=a}f_{\chi, \rho}^{\eta}(-z, \nu) d\nu \bigg\}.
\end{align*}
}In the bracket of the right hand side, the first term is holomorphic on $\Re(\l) > -1/2$
and the part of $a = (-\nu-1)/2$ in the second term is transposed into
{\allowdisplaybreaks\begin{align*}
& \int_{i \RR}\tilde{\a}_{\chi}(\nu)
D_{F}^{-1/2}\mathfrak{L}_{\chi, \rho}^{\bf 1}(\nu)
\frac{\b((-\nu-1)/2)}{\l+(-\nu-1)/2}\Res_{z=(-\nu-1)/2}f_{\chi, \rho}^{\eta}(-z, \nu) d\nu \\
= &
\int_{L_{-\s}}\tilde{\a}_{\chi}(\nu)
D_{F}^{-1/2}\mathfrak{L}_{\chi, \rho}^{\bf 1}(\nu)
\frac{\b((-\nu-1)/2)}{\l + (-\nu-1)/2}\Res_{z=(-\nu-1)/2}f_{\chi, \rho}^{\eta}(-z, \nu) d\nu \\
& +2\pi i \ \delta_{\chi, \eta} \Res_{\nu=-1}(\mathfrak{L}_{\eta, \rho}^{\bf 1}(\nu))\tilde{\a}_{\eta}(1)
D_{F}^{-1/2}\frac{\b(0)}{\l}\Gcal(\eta)D_{F}^{1/2}B_{\eta, \rho}^{\eta}(1/2, -1)(-R_{F})
\end{align*}
}for any fixed $\s > 1$.
We note
$\Res_{\nu=-1}\mathfrak{L}_{\eta, \rho}^{\bf 1}(\nu)
=p_{0}^{\eta}(\rho)D_{-1}^{\eta} + p_{1}^{\eta}(\rho)D_{-2}^{\eta}.
$
Thus the part of $a=(-\nu-1)/2$ is meromorphic on $\Re(\l) > -1/2$.
Noting that the part of
$a=(\nu-1)/2$ equals that of $a=(-\nu-1)/2$, 
the function $\Phi_{3}(\l)$ has a meromorphic continuation to $\Re(\l) > -1/2$. This completes the proof.
\QED
\begin{lem}\label{CT of P eis + P res}
We have
{\allowdisplaybreaks\begin{align*}
& {\rm CT}_{\l=0}\PP_{\rm eis}^{\eta}(\b, \l, \a) +  {\rm CT}_{\l=0}\PP_{\rm res}^{\eta}(\b, \l, \a) \\
=
&\bigg\{ \Gcal(\eta) D_{F}^{-1/2}R_{F}^{-1}\sum_{\chi \in \Xi(\gn)}\sum_{\rho \in \Lambda_{\chi}(\gn)}
\frac{1}{8 \pi i}\int_{i \RR} \tilde{\a}_{\chi}(\nu)\mathfrak{L}_{\chi, \rho}^{\bf 1}(\nu)\mathfrak{L}_{\chi^{-1}, \rho}^{\eta}(-\nu)d \nu \\
& + \d(\gf_{\eta} = \go_{F}) \{ Y_{2}^{\eta}(\gn) \tilde{\a}_{\eta}''(1) + Y_{1}^{\eta}(\gn) \tilde{\a}_{\eta}'(1) + Y_{0}^{\eta}(\gn) \tilde{\a}_{\eta}(1)\}
+ Y_{-1}^{\eta}(\gn) \tilde{\a}(1)
\bigg\} \b(0),
\end{align*}
}where we put
$$Y_{2}^{\eta}(\gn) = \sum_{\rho \in \Lambda(\gn)}
D_{F}^{-1/2}\{f_{\eta, \rho}^{(0)}(e) + \d(S(\rho)=\emptyset)\}\frac{1}{2}p_{0}^{\eta}(\rho)D_{-2}^{\eta},
$$
$$Y_{1}^{\eta}(\gn) = \sum_{\rho \in \Lambda(\gn)}
D_{F}^{-1/2}\{f_{\eta, \rho}^{(0)}(e) + \d(S(\rho)=\emptyset)\}
\{D_{-1}^{\eta}p_{0}^{\eta}(\rho) + D_{-2}^{\eta}p_{1}^{\eta}(\rho)\},
$$
$$Y_{0}^{\eta}(\gn) = \sum_{\rho \in \Lambda(\gn)}D_{F}^{-1/2}
\{f_{\eta, \rho}^{(0)}(e) + \d(S(\rho)=\emptyset)\}
\{D_{-2}^{\eta}p_{2}^{\eta}(\rho) + D_{-1}^{\eta}p_{1}^{\eta}(\rho) + D_{0}^{\eta}p_{0}^{\eta}(\rho)\}
$$
and
$$Y_{-1}^{\eta}(\gn) = \sum_{\rho \in \Lambda(\gn)} \frac{\Gcal(\eta)D_{F}^{-1/2}}{\zeta_{F}(2)} \{f_{{\bf 1}, \rho}^{(0)}(e) + \d(S(\rho)=\emptyset)\}
a_{{\bf 1}, \rho}^{\eta}(0).
$$
\end{lem}
\Proof
Let $\Phi_{1}^{+}$, $\Phi_{2}$ and $\Phi_{3}$ be the functions defined in the proof of
Lemma \ref{lem: merom of P eisen}.
Then, we obtain
${\rm CT}_{\l=0}\PP_{\rm eis}^{\eta}(\b, \l, \a) = {\rm CT}_{\l=0}(2\Phi_{1}^{+}(\l) + \Phi_{2}(\l) + \Phi_{3}(\l))$.
A direct computation gives us
{\allowdisplaybreaks\begin{align*}
{\rm CT}_{\l=0}\Phi_{1}^{+}(\l) = &\delta(\gf_{\eta}= \go_{F}) \sum_{\rho \in \Lambda(\gn)}\frac{1}{8 \pi i}\int_{i \RR}
\tilde{\a}_{\eta}(\nu)D_{F}^{-1/2} \mathfrak{L}_{\eta, \rho}^{\bf 1}(\nu)f_{\eta, \rho}^{(0)}(e)\frac{\b(0)}{(\nu+1)/2} d\nu \\
& + \delta(\gf_{\eta}=\go_{F})\sum_{\rho \in \Lambda(\gn)}\frac{1}{8 \pi i}\int_{L_{-\s}}
\tilde{\a}_{\eta}(\nu)D_{F}^{-1/2} \mathfrak{L}_{\eta, \rho}^{\bf 1}(\nu)f_{\eta, \rho}^{(0)}(e)\frac{\b(0)}{-(\nu+1)/2}d\nu
\end{align*}
}and
{\allowdisplaybreaks\begin{align*}
{\rm CT}_{\l=0}\Phi_{2}(\l) = 
\sum_{\chi \in \Xi(\gn)} \sum_{\rho \in \Lambda_{\chi}(\gn)}
\frac{R_{F}^{-1}}{8 \pi i} \int_{i\RR} \tilde{\alpha}_{\chi}(\nu)
D_{F}^{-1/2} \mathfrak{L}_{\chi, \rho}^{\bf 1}(\nu)
Q_{\chi, \rho}^{0}(\eta, 0, \nu)d \nu,
\end{align*}
}where
$$Q_{\chi, \rho}^{0}(\eta, 0, \nu) = f_{\chi, \rho}^{\eta}(0, \nu)\b(0) + \sum_{a=(\pm \nu \pm1)/2}\Res_{z=a}\bigg\{f_{\chi, \rho}^{\eta}(-z, \nu)\frac{\b(z)}{z}\bigg\}.$$
The constant term of $\Phi_{3}(\l)$ at $\l=0$ is evaluated as
{\allowdisplaybreaks\begin{align*}
& {\rm CT}_{\l=0}\Phi_{3}(\l) \\
= & -\sum_{\chi \in \Xi(\gn)} \sum_{\rho \in \Lambda_{\chi}(\gn)}\frac{R_{F}^{-1}}{8\pi i} \bigg\{ \int_{i \RR}\tilde{\a}_{\chi}(\nu)
D_{F}^{-1/2}\mathfrak{L}_{\chi, \rho}^{\bf 1}(\nu)
\sum_{a = (\pm \nu + 1)/2}\frac{\b(a)}{a}\Res_{z=a}f_{\chi, \rho}^{\eta}(-z, \nu) d\nu\bigg\} \\
& -2 \d(\gf_{\eta}=\go_{F})\sum_{\rho \in \Lambda(\gn)} \frac{R_{F}^{-1}}{8\pi i}\bigg\{\int_{L_{-\s}}\tilde{\a}_{\eta}(\nu)
D_{F}^{-1/2}\mathfrak{L}_{\eta, \rho}^{\bf 1}(\nu)
\frac{\b((-\nu-1)/2)}{ (-\nu-1)/2}\Res_{z=(-\nu-1)/2}f_{\eta, \rho}^{\eta}(-z, \nu) d\nu \bigg\}.
\end{align*}
}Hence, we obtain
{\allowdisplaybreaks\begin{align*}
& {\rm CT}_{\l=0}\PP_{\rm eis}^{\eta}(\b, \l, \a) \\
= & 
\ 2\d(\gf_{\eta}=\go_{F})\sum_{\rho \in \Lambda(\gn)}\frac{1}{8 \pi i}
D_{F}^{-1/2}
f_{\eta, \rho}^{(0)}(e)
\left(\int_{i \RR}-\int_{L_{-\s}}\right)
\tilde{\a}_{\eta}(\nu)\mathfrak{L}_{\eta, \rho}^{\bf 1}(\nu)\frac{\b((\nu+1)/2)}{(\nu+1)/2} d\nu \\
& + \sum_{\chi \in \Xi(\gn)} \sum_{\rho \in \Lambda_{\chi}(\gn)}
\frac{R_{F}^{-1}}{8 \pi i} \int_{i\RR} \tilde{\alpha}_{\chi}(\nu)
D_{F}^{-1/2} \mathfrak{L}_{\chi, \rho}^{\bf 1}(\nu)
f_{\chi, \rho}^{\eta}(0, \nu)\b(0)d \nu \\
&
+ 2\d(\gf_{\eta}=\go_{F})\sum_{\rho \in \Lambda(\gn)}\frac{R_{F}^{-1}}{8\pi i} \bigg\{ \left(\int_{i \RR} - \int_{L_{-\s}}\right)
\tilde{\a}_{\eta}(\nu)
D_{F}^{-1/2}\mathfrak{L}_{\eta, \rho}^{\bf 1}(\nu)
\frac{\b((-\nu-1)/2)}{(-\nu -1)/2} \\
& \times \Res_{z=(-\nu-1)/2}f_{\eta, \rho}^{\eta}(-z, \nu) d\nu \bigg\} \\
= &\ \d(\gf_{\eta}=\go_{F})\sum_{\rho \in \Lambda(\gn)}\frac{1}{2}D_{F}^{-1/2}
f_{\eta, \rho}^{(0)}(e)
\Res_{\nu = -1}\bigg\{
\tilde{\a}_{\eta}(\nu)\mathfrak{L}_{\eta, \rho}^{\bf 1}(\nu)\frac{\b((\nu+1)/2)}{(\nu+1)/2}\bigg\} \\
& + \sum_{\chi \in \Xi(\gn)} \sum_{\rho \in \Lambda_{\chi}(\gn)}
\frac{R_{F}^{-1}}{8 \pi i} \int_{i\RR} \tilde{\alpha}_{\chi}(\nu)
D_{F}^{-1/2} \mathfrak{L}_{\chi, \rho}^{\bf 1}(\nu)
\Gcal(\eta) \mathfrak{L}_{\chi^{-1}, \rho}^{\eta}(-\nu) \b(0)d \nu \\
&
+ \d(\gf_{\eta}=\go_{F})\sum_{\rho \in \Lambda(\gn)}\frac{R_{F}^{-1}}{2} \Res_{\nu = -1} \bigg\{
\tilde{\a}_{\eta}(\nu)
D_{F}^{-1/2}\mathfrak{L}_{\eta, \rho}^{\bf 1}(\nu)
\frac{\b((-\nu-1)/2)}{(-\nu -1)/2}\Res_{z=(-\nu-1)/2}f_{\eta, \rho}^{\eta}(-z, \nu)
\bigg\}.
\end{align*}
}We remark
$$
\Res_{z=(-\nu-1)/2}f_{\eta, \rho}^{\eta}(-z, \nu) = \Gcal(\eta)(-R_{F}) D_{F}^{-\nu/2}B_{\eta, \rho}^{\eta}(\nu/2+1, \nu)= -R_{F}\d(S(\rho)=\emptyset)$$
and compute the residues as follows:
{\allowdisplaybreaks\begin{align*}
&\Res_{\nu = -1} \bigg\{
\tilde{\a}_{\eta}(\nu)
\mathfrak{L}_{\eta, \rho}^{\bf 1}(\nu)
\frac{\b((\nu+1)/2)}{(\nu +1)/2}\bigg\}= -\Res_{\nu = -1} \bigg\{
\tilde{\a}_{\eta}(\nu)
\mathfrak{L}_{\eta, \rho}^{\bf 1}(\nu)
\frac{\b((-\nu-1)/2)}{(-\nu -1)/2}\bigg\}\\
= & \ \tilde{\a}_{\eta}''(1)p_{0}^{\eta}(\rho)D_{-2}^{\eta}\b(0)+
 2 \tilde{\a}_{\eta}'(1)\{p_{0}^{\eta}(\rho)D_{-1}^{\eta}+ p_{1}^{\eta}(\rho)D_{-2}^{\eta}\}\b(0) \\
& +\tilde{\a}_{\eta}(1) \bigg\{ D_{-2}^{\eta}\left(2p_{2}^{\eta}(\rho)\b(0) +\frac{1}{4}p_{0}^{\eta}(\rho)\b''(0)\right)
+2 D_{-1}^{\eta}p_{1}^{\eta}(\rho)\b(0) +2D_{0}^{\eta}p_{0}^{\eta}(\rho) \b(0) \bigg\}.
\end{align*}
}One can check that
the sum of all terms containing $\b''(0)$ in ${\rm CT}_{\l=0}\PP_{\rm eis}^{\eta}(\b, \l, \a) +{\rm CT}_{\l=0}\PP_{\rm res}^{\eta}(\b, \l, \a)$
vanishes with the aid of Lemma \ref{lem: merom of P res}.
As a consequence, we obtain the assertion.
\QED

\begin{lem}\label{lem:estimations of Y and Z}
For any $\e>0$, we have the following estimates
$$|Y_{j}^{\eta}(\gn)| \ll {\rm N}(\gn)^{\e}, \hspace{3mm} j \in \{-1, 0, 1, 2\},$$
where the implied constant is independent of $\gn$.
\end{lem}
\Proof
The proof is given by describing $Y_{j}^{\eta}(\gn)$ for $j \in \{-1, 0, 1, 2\}$ explicitly.
Since $\eta$ is unramified, we have
$$f_{\eta, \rho}^{(0)}(e)=\prod_{v \in S_{1}(\rho)}\eta_{v}(\varpi_{v})q_{v}^{1/2} \prod_{k=2}^{n}\prod_{v \in S_{k}(\rho)} (1-q_{v}^{-1})\eta_{v}(\varpi_{v})^{k} \left(\frac{q_{v}+1}{q_{v}-1}\right)^{1/2}q_{v}^{k/2}
$$
and
{\allowdisplaybreaks\begin{align*}
p_{0}^{\eta}(\rho) =& \tilde{\eta}(\gD_{F/\QQ})\prod_{v \in S_{1}(\rho)}(1-\eta_{v}(\varpi_{v}))\frac{q_{v}}{q_{v}-1}q_{v}^{-1/2} \\
& \times \prod_{k=2}^{n}\prod_{v \in S_{k}(\rho)}
\bigg\{ \frac{(\eta_{v}(\varpi_{v})-1)(\eta_{v}(\varpi_{v}) q_{v} -1)}{q_{v}-q_{v}^{-1}} \left(\frac{q_{v}+1}{q_{v}-1}\right)^{1/2}q_{v}^{-k/2} \bigg\}.
\end{align*}
}Moreover, we obtain expressions
$$p_{1}^{\eta}(\rho) =\tilde{\eta}(\gD_{F/\QQ}) \sum_{w \in S(\rho)}
\big\{ \prod_{v \in S(\rho)-\{w\}}Y_{v}^{\eta}(-1) \big\} (Y_{w}^{\eta})'(-1)
$$
and
{\allowdisplaybreaks\begin{align*}
p_{2}^{\eta}(\rho) = & \frac{\tilde{\eta}(\gD_{F/\QQ})}{2}\sum_{w \in S(\rho)}
\bigg[ \sum_{x \in S(\rho)-\{w\}} \bigg\{ \prod_{v \in S(\rho)-\{w, x\}}Y_{v}^{\eta}(-1)\bigg\}
(Y_{w}^{\eta})'(-1) (Y_{x}^{\eta})'(-1) \\
& + \bigg\{ \prod_{v \in S(\rho)-\{w\}}Y_{v}^{\eta}(-1) \bigg\} (Y_{w}^{\eta})''(-1) \bigg],
\end{align*}
}where we set
$$C_{v} = \d(v \in S_{1}(\rho)) + \d \left(v \in \coprod_{k=2}^{n}S_{k}(\rho)\right) \left(\frac{q_{v}+1}{q_{v}-1}\right)^{1/2}
$$
and
$$Y_{v}^{\eta}(\nu) = C_{v} \{q_{v}+1+\eta_{v}(\varpi_{v})(q_{v}^{(1+\nu)/2}+q_{v}^{(1-\nu)/2})\}
\frac{q_{v}^{k\nu/2}}{q_{v}-q_{v}^{\nu}}.
$$
Further we have
$$(Y_{v}^{\eta})'(-1)=C_{v}(\log q_{v}^{k})q_{v}^{-k/2}
\frac{-\eta_{v}(\varpi_{v})q_{v}(q_{v}-1)^{2}+ k (1+\eta_{v}(\varpi_{v}))q_{v}(q_{v}^{2}-1) + 2(1+\eta_{v}(\varpi_{v}))q_{v}} {2k(q_{v}^{2}-1)(q_{v}-1)}$$
and
{\allowdisplaybreaks\begin{align*}
&(Y_{v}^{\eta})''(-1)\\
 = & C_{v}\bigg[ \eta_{v}(\varpi_{v})(\log q_{v}^{k})^{2}q_{v}^{-k/2}
\frac{(1+q_{v})(q_{v}-q_{v}^{-1}) + (1-q_{v})\{k(q_{v}-q_{v}^{-1})+2q_{v}^{-1}\}}
{4k^{2}(q_{v}-q_{v}^{-1})^{2}} \\
&+
(\log q_{v}^{k})^{2}q_{v}^{-k/2}\frac{\eta_{v}(\varpi_{v})\{k(q_{v}^{3}-q_{v})+2q_{v}\}} {4k^{2}(1+q_{v})(1-q_{v}^{2})}\\
&+ (\log q_{v}^{k})^{2}q_{v}^{-k/2}\frac{(1+\eta_{v}(\varpi_{v}))q_{v}}{k^{2}(q_{v}^{2}-1)^{3}(q_{v}-1)}
\bigg\{\left(\frac{k^{2}}{4}(q_{v}^{2}-1)+1\right)(q_{v}^{2} -1)^{2} + (k(q_{v}^{2}-1)+2)(q_{v}^{2}-1)
\bigg\}\bigg].
\end{align*}
}Thus, by noting $\# S(\rho) \ll \log(1+{\rm N}(\gn))$, we obtain the estimates of $Y_{j}^{\eta}(\gn)$ for $j \in \{0, 1, 2\}$.

Next let us examine $Y_{-1}^{\eta}(\gn)$.
We have the following expressions:
$$\tilde{B}_{{\bf 1}, \rho}^{\eta}(0) = \e(0, \eta)B_{{\bf 1}, \rho}^{\eta}(1/2, 1),
$$
{\allowdisplaybreaks\begin{align*}
B_{{\bf1}, \rho}^{\eta}(1/2, 1) = &\prod_{v \in S_{1}(\rho)}\frac{(\eta_{v}(\varpi_{v})-1)q_{v}^{-1/2}}{(1-q_{v}^{-1})} \\
& \times \prod_{k=2}^{n}\prod_{v \in S_{k}(\rho)}\bigg\{ \frac{\eta_{v}(\varpi_{v})^{k}
(\eta_{v}(\varpi_{v})-1) (\eta_{v}(\varpi_{v})-q_{v}^{-1})}{1 -q_{v}^{-2}} \left(\frac{q_{v}+1}{q_{v}-1}\right)^{1/2}q_{v}^{-k/2}\bigg\},
\end{align*}
}
$$(\tilde{B}_{{\bf 1}, \rho}^{\bf 1})''(0) = \e''(0, {\bf 1})B_{\rho}(0)+2 \e'(0, {\bf 1})B'_{\rho}(0)+
\e(0, {\bf 1})B''_{\rho}(0).$$
Here we set
$B_{\rho}(z) =B_{{\bf 1}, \rho}^{\bf 1}(-z+1/2, 1) = D_{F}^{-z}\prod_{v \in S(\rho)}B_{v}(z)$
and
{\allowdisplaybreaks\begin{align*}
B_{v}(z)= & \d(v \in S_{1}(\rho))(q_{v}^{z}-1)\frac{q_{v}^{-1/2}}{1-q_{v}^{-1}} \\
& + \sum_{k=2}^{n}\d(v \in S_{k}(\rho))
(q_{v}^{kz}-q_{v}^{(k-1)z-1}-q_{v}^{(k-1)z}+q_{v}^{(k-2)z-1}) \left(\frac{q_{v}+1}{q_{v}-1}\right)^{1/2}\frac{q_{v}^{-k/2}}{1-q_{v}^{-2}}.
\end{align*}
}A direct computation gives us
$$B'_{\rho}(0)=(\log D_{F}^{-1}) \prod_{v \in S(\rho)}B_{v}(0)
+ \sum_{w \in S(\rho)} \big\{ \prod_{v \in S(\rho)-\{w\}}B_{v}(0) \big\}B'_{v}(0),
$$
{\allowdisplaybreaks\begin{align*}
B''_{\rho}(0) = & (\log D_{F}^{-1})^{2}\prod_{v \in S(\rho)}B_{v}(0)
+2(\log D_{F}^{-1}) \sum_{w \in S(\rho)}\big\{ \prod_{v \in S(\rho)-\{w\}}B_{v}(0) \big\} B'_{v}(0)\\
& + \sum_{w \in S(\rho)} \left\{ \sum_{x \in S(\rho) - \{ w\}} \big\{ \prod_{v \in S(\rho)-\{w, x\}}B_{v}(0) \big\} B'_{w}(0)B'_{x}(0)+
\big\{ \prod_{v \in S(\rho)-\{w\}}B_{v}(0) \big\} B_{w}''(0)\right\},
\end{align*}
}
$$B'_{v}(0) = \d(v \in S_{1}(\rho))(\log q_{v})\frac{q_{v}^{-1/2}}{1-q_{v}^{-1}}
+ \sum_{k=2}^{n}\d(v \in S_{k}(\rho))(\log q_{v})\left(\frac{q_{v}+1}{q_{v}-1}\right)^{1/2}\frac{q_{v}^{-k/2}}{1 + q_{v}^{-1}}$$
and
{\allowdisplaybreaks\begin{align*}
B''_{v}(0) = & \d(v \in S_{1}(\rho))(\log q_{v})^{2}\frac{q_{v}^{-1/2}}{1-q_{v}^{-1}} \\
& + \sum_{k=2}^{n}\d(v \in S_{k}(\rho))
(\log q_{v}^{k})^{2}\frac{2k-1-(2k-3)q_{v}^{-1}}{k^{2}}\left(\frac{q_{v}+1}{q_{v}-1}\right)^{1/2}\frac{q_{v}^{-k/2}}{1 - q_{v}^{-2}}.
\end{align*}
}This completes the proof of the estimate of $Y_{-1}^{\eta}(\gn)$.
\QED
By the aid of Lemmas \ref{spectral exp of reg Green} and
\ref{CT of P eis + P res}, we obtain the expression of the spectral side
of $P_{\rm reg}^{\eta}(\hat{{\bf \Psi}}_{\rm reg}(\gn|\a; -))$.
\begin{thm}
\label{spectral side}
The value $P_{\rm reg}^{\eta}(\hat{{\bf \Psi}}_{\rm reg}(\gn|\a; -))$ can be defined and we have
$$
P_{\rm reg}^{\eta}(\hat{{\bf \Psi}}_{\rm reg}(\gn|\a; -)) = C(\gn, S)\{ \II_{\rm cus}^{\eta}(\gn|\a) + \II_{\rm eis}^{\eta}(\gn|\a) + \DD^{\eta}(\gn|\a) \}.
$$
Here we put
$$\II_{\rm cus}^{\eta}(\gn|\a) = \sum_{\varphi \in \Bcal_{\rm cus}(\gn)}\alpha(\nu_{\varphi, S})\overline{P_{\rm reg}^{\bf 1}(\varphi)}P_{\rm reg}^{\eta}(\varphi),
$$
$$\II_{\rm eis}^{\eta}(\gn|\a) = \sum_{\chi \in \Xi(\gn)} \sum_{\rho \in \Lambda_{\chi}(\gn)}
\frac{R_{F}^{-1}}{8 \pi i} \int_{i\RR} \tilde{\alpha}_{\chi}(\nu)P_{\rm reg}^{\bf 1}(
E_{\chi^{-1}, \rho}(-\nu, -))P_{\rm reg}^{\eta}(E_{\chi, \rho}(\nu, -))d\nu
$$
and
$$\DD^{\eta}(\gn|\a) = \d(\gf_{\eta} = \go_{F}) \{ Y_{2}^{\eta}(\gn) \tilde{\a}_{\eta}''(1) + Y_{1}^{\eta}(\gn) \tilde{\a}_{\eta}'(1) + Y_{0}^{\eta}(\gn) \tilde{\a}_{\eta}(1)\}
+ Y_{-1}^{\eta}(\gn) \tilde{\a}(1).
$$
\end{thm}

\section{Periods of regularized automorphic smoothed kernels: geometric side}
\label{Periods of regularized automorphic smoothed kernels: geometric side}
In this section, we describe the geometric expression of $\hat{{\bf \Psi}}_{\rm reg}(\gn|\a; -)$ and its regularized $\eta$-period $P_{\rm reg}^{\eta}(\hat{{\bf \Psi}}_{\rm reg}(\gn|\a; -))$.
For $\d \in G_{F}$, we put
${\rm St(\d)} = H_{F} \cap \d^{-1}H_{F}\d$.
By 
\cite[Lemma 11.1]{TsuzukiSpec},
the following elements of $G_{F}$ form a complete system of representatives of the double coset space $H_{F} \backslash G_{F}/H_{F}$:
\begin{center}
$e = \left(\begin{matrix}1 & 0 \\ 0 & 1\end{matrix}\right)$,
$w_{0} = \left(\begin{matrix} 0 & -1 \\ 1 & 0\end{matrix}\right)$,

$u = \left(\begin{matrix}1 & 1 \\ 0 & 1\end{matrix}\right)$,
$\overline{u} = \left(\begin{matrix}1 & 0 \\ 1 & 1\end{matrix}\right)$,
$uw_{0} = \left(\begin{matrix}1 & -1 \\ 1 & 0\end{matrix}\right)$,
$\overline{u}w_{0}= \left(\begin{matrix}0 & -1 \\ 1 & -1\end{matrix}\right)$,

$\d_{b} = \left(\begin{matrix} 1+b^{-1} & 1 \\ 1 & 1\end{matrix}\right), \hspace{2mm}{b\in F^{\times} - \{-1\}}$.
\end{center}
Moreover, we have ${\rm St}(e) ={\rm St}(w_{0}) = H_{F}$ and ${\rm St}(\d) = Z_{F}$ for any $\d \in \{u, \overline{u},
uw_{0}, \overline{u}w_{0}\} \cup \{\d_{b} | b \in F^{\times} - \{-1\} \}$.
We note
$$H_{F} \backslash G_{F} = \coprod_{\d \in H_{F} \backslash G_{F}/H_{F}} H_{F}\backslash (H_{F}\d H_{F}) \cong \coprod_{\d \in H_{F} \backslash G_{F}/H_{F}} {\rm St}(\d) \backslash H_{F}.$$
Thus we obtain the following expression for $\Re(\l) > 0$:
$$
\hat{{\bf \Psi}}_{\b, \l} \left(\gn \bigg| \a; \left(\begin{matrix} t & 0 \\ 0 & 1\end{matrix}\right)
\left(\begin{matrix}1 & x_{\eta} \\ 0 & 1\end{matrix}\right)\right)
= \sum_{\d \in H_{F} \backslash G_{F}/ H_{F}}
\sum_{\gamma \in {\rm St}(\d) \backslash H_{F}} \hat{\Psi}_{\b, \l} \left(\gn \bigg| \a; \d \gamma \left(\begin{matrix} t & 0 \\ 0 & 1\end{matrix}\right)
\left(\begin{matrix}1 & x_{\eta} \\ 0 & 1\end{matrix}\right)\right).
$$
Set 
$$J_{\d}(\b, \l, \a; t) =
\sum_{\gamma \in {\rm St}(\d) \backslash H_{F}} \hat{\Psi}_{\b, \l} \left(\gn \bigg| \a; 
\d \gamma
\left(\begin{matrix} t & 0 \\ 0 & 1\end{matrix}\right)
\left(\begin{matrix}1 & x_{\eta} \\ 0 & 1\end{matrix}\right)\right)
$$
for any $\d \in H_{F} \backslash G_{F}/ H_{F}$.
We examine $J_{\d}(\b, \l, \a; t)$.
With a minor modification, we obtain Lemmas \ref{lem: J-id} and \ref{lem: J-u} by the same computation as \cite[Lemma 11.2]{TsuzukiSpec} and \cite[Lemma 11.3]{TsuzukiSpec}, respectively.
\begin{lem}\label{lem: J-id}
Both functions $\l \mapsto J_{e}(\b, \l, \a; t)$ and
$\l \mapsto J_{w_{0}}(\b, \l, \a; t)$ are analytically continued to entire functions.
The values of these functions at $\l = 0$
are equal to
$J_{\rm id}(\a, t)\b(0)$ and
$\d(\gn = \go_{F})J_{\rm id}(\a, t)\b(0)$, respectively,
where
$$
J_{\rm id}(\a, t)= \d(\gf_{\eta} = \go_{F})\left( \frac{1}{2 \pi i} \right)^{\#S}
\int_{\LL_{S}(\bf c)} \Upsilon_{S}^{\bf 1}(\bfs) \a(\bfs) d\mu_{S}(\bfs)
$$
with
$$
\Upsilon_{S}^{\bf 1}(\bfs) = \left\{ \prod_{v\in \Sigma_{\infty}} \frac{-1}{8}\frac{\Gamma((s_{v} + 1)/4)^{2}}{\Gamma((s_{v} + 3)/4)^{2}} \right\} \left\{ \prod_{v \in S_{\fin}} (1-q_{v}^{-(s_{v}+1)/2})^{-1}(1-q_{v}^{(s_{v} + 1)/2})^{-1} \right\}.
$$
\end{lem}
We put
$$J_{{\rm u}}(\b, \l, \a; t) = J_{u}(\b, \l, \a, t) + J_{\overline{u}w_{0}}(\b, \l, \a, t)$$
and
$$J_{\bar{\rm u}}(\b, \l, \a; t) = J_{uw_{0}}(\b, \l, \a, t) + J_{\bar{u}}(\b, \l, \a, t).$$
\begin{lem}\label{lem: J-u}
For $ *\in \{ {\rm u}, \bar{\rm u}\}$, the function $\l \mapsto J_{*}(\b, \l, \a, t)$
is analytically continued to an entire function and the value at $\l = 0$
is equal to $J_{*}(\a, t)\b(0)$, where
{\allowdisplaybreaks\begin{align*}
J_{{\rm u}}(\a; t)
= & \left(\frac{1}{2\pi i}\right)^{\#S} \sum_{a \in F^{\times}}
\int_{\LL_{S}(\bf c)} \bigg\{ \hat{\Psi}^{(0)} \left(\gn \bigg| \bfs;
\left(\begin{matrix} 1 & at^{-1} \\ 0 & 1\end{matrix}\right)
\left(\begin{matrix}1 & x_{\eta} \\ 0 & 1\end{matrix}\right)\right) \\
& + \d(\gn = \go_{F})
\hat{\Psi}^{(0)} \left(\gn \bigg| \bfs;
\left(\begin{matrix} 1 & 0 \\ at^{-1} & 1\end{matrix}\right)
\left(\begin{matrix}1 & 0 \\ -x_{\eta} & 1\end{matrix}\right)w_{0}\right)
\bigg\}\a(\bfs) d\mu_{S}(\bfs)
\end{align*}
}and
{\allowdisplaybreaks\begin{align*}
J_{\bar{\rm u}}(\a; t)
= & \left(\frac{1}{2\pi i}\right)^{\#S} \sum_{a \in F^{\times}}
\int_{\LL_{S}(\bf c)} \bigg\{ \hat{\Psi}^{(0)} \left(\gn \bigg| \bfs;
\left(\begin{matrix} 1 & 0 \\ at & 1\end{matrix}\right)
\left(\begin{matrix}1 & x_{\eta} \\ 0 & 1\end{matrix}\right)\right) \\
& + \d(\gn = \go_{F})
\hat{\Psi}^{(0)} \left(\gn \bigg| \bfs;
\left(\begin{matrix} 1 & at \\ 0 & 1\end{matrix}\right)
\left(\begin{matrix}1 & 0 \\ -x_{\eta} & 1\end{matrix}\right)w_{0}\right)
\bigg\}\a(\bfs) d\mu_{S}(\bfs).
\end{align*}
}These series-integrals are absolutely convergent.
\end{lem}
We put
{\allowdisplaybreaks\begin{align*}
J_{\rm hyp}(\b, \l, \a; t) = \sum_{b \in F^{\times}-\{-1\}}J_{\d_{b}}(\b, \l, \a; t)
= \sum_{b \in F^{\times}-\{-1\}} \sum_{a \in F^{\times}}\hat{\Psi}_{\b, \l}
\left(\gn \bigg| \a; \delta_{b}
\left(\begin{matrix} at & 0 \\ 0 & 1\end{matrix}\right)
\left(\begin{matrix} 1 & x_{\eta} \\ 0 & 1\end{matrix}\right)\right).
\end{align*}
}We obtain Lemma \ref{lem: J-hyp} by the same proof as \cite[Lemma 11.21]{TsuzukiSpec}.
In \cite[Lemma 11.21]{TsuzukiSpec}, it is assumed that $\gn$ is square free. However, the argument in \cite[Lemma 11.21]{TsuzukiSpec} works with a minor modification.

\begin{lem}\label{lem: J-hyp}
The function $J_{\rm hyp}(\b, \l, \a; t)$ on $\Re(\l)>1$ is analytically continued to an entire function and the value at $\l = 0$ is $J_{\rm hyp}(\a, t)\b(0)$, where
$$
J_{\rm hyp}(\a, t) = \sum_{b \in F^{\times}-\{-1\}} \sum_{a \in F^{\times}}
{\Psi}^{(0)}
\left(\gn \bigg| \a; \d_{b}
\left(\begin{matrix} at & 0 \\ 0 & 1\end{matrix}\right)
\left(\begin{matrix} 1 & x_{\eta} \\ 0 & 1\end{matrix}\right)\right).
$$
The series converges absolutely and locally uniformly in $t \in \AA^{\times}$.
\end{lem}

Therefore we obtain the geometric expression of $\hat{{\bf \Psi}}_{\rm reg}\left(\gn \bigg| \a; \left(\begin{matrix} t & 0 \\ 0 & 1\end{matrix}\right)
\left(\begin{matrix} 1 & x_{\eta} \\ 0 & 1\end{matrix}\right)\right)$
by Lemmas \ref{lem: J-id}, \ref{lem: J-u} and \ref{lem: J-hyp}.
\begin{prop}\label{geom exp of reg Green}
Let $\gn$ be an ideal of $\go_{F}$ and
$S$ a finite subset of $\Sigma_{F}$ satisfying $\Sigma_{\infty} \subset S$ and $S \cap S(\gn) = \emptyset$. Let $\eta$ be a character satisfying ($\star$) in \S 2.1.
Then, for any $\a \in \Acal_{S}$, we have
{\allowdisplaybreaks\begin{align*}
& \hat{{\bf \Psi}}_{\rm reg}\left(\gn \bigg| \a; \left(\begin{matrix} t & 0 \\ 0 & 1\end{matrix}\right)
\left(\begin{matrix} 1 & x_{\eta} \\ 0 & 1\end{matrix}\right)\right) \\
= & (1 + \d(\gn = \go_{F}))J_{\rm id}(\a; t) + J_{\rm u}(\a; t) + J_{\bar{\rm u}}(\a; t) + J_{\rm hyp}(\a; t), \ \ t \in \AA^{\times}.
\end{align*}
}
\end{prop}

Next let us compute $P_{\b, \l}^{\eta}(\hat{{\bf \Psi}}_{\rm reg}(\gn|\a; -))$ explicitly.
Define
$$\JJ_{*}^{\eta}(\b, \l; \a) = \int_{F^{\times}\backslash \AA^{\times}}J_{*}(\a; t)
\{\hat{\b}_{\l}(|t|_{\AA}) + \hat{\b}_{\l}(|t|_{\AA}^{-1})\}\eta(t)\eta_{\fin}(x_{\eta,\fin})d^{\times}t$$
and
$$\Upsilon_{S}^{\eta}(\bfs) = \left\{ \prod_{v\in \Sigma_{\infty}} \frac{-1}{8}\frac{\Gamma((s_{v} + 1)/4)^{2}}{\Gamma((s_{v} + 3)/4)^{2}} \right\} \left\{\prod_{v \in S_{\fin}} (1-q_{v}^{(s_{v}+1)/2})^{-1}(1-\eta_{v}(\varpi_{v})q_{v}^{-(s_{v} + 1)/2})^{-1} \right\}.
$$
For any ideal $\ga$ of $\go_{F}$, we set
{\allowdisplaybreaks\begin{align*}
\gC_{S, \ga}^{\eta}(\bfs) = & C_{0}(\eta) + R(\eta) \bigg\{\log(D_{F}{\rm N}(\ga)) + \frac{d_{F}}{2}(C_{\rm Euler} + 2 \log2 -\log \pi) \\
& + \sum_{v \in S_{\fin}} \frac{\log q_{v}}{1 - q_{v}^{(s_{v} + 1)/2}}
+ \frac{1}{2}\sum_{v \in \Sigma_{\infty}}\left(\psi \left(\frac{s_{v} + 1}{4}\right) + \psi \left(\frac{s_{v} + 3}{4}\right)\right)\bigg\},
\end{align*}
}where $\psi(z) = \Gamma'(s) / \Gamma(s)$ is the digamma function and $C_{\rm Euler}$ is the Euler constant.
We note that if $\eta \neq {\bf1}$, then $\gC_{S, \ga}^{\eta}(\bfs)$ is independent of the choice of $\ga$, and $\gC_{S, \ga}^{\eta}(\bfs) = C_{0}(\eta) = L(1, \eta)$.
Put
$$\mathfrak{K}_{\eta}(\gn | \bfs) =  \sum_{b \in F^{\times}-\{-1\}} \int_{\AA^{\times}}
{\Psi}^{(0)}
\left(\gn \bigg| \bfs; \d_{b}
\left(\begin{matrix} t & 0 \\ 0 & 1\end{matrix}\right)
\left(\begin{matrix} 1 & x_{\eta} \\ 0 & 1\end{matrix}\right)\right)
\eta(t)\eta_{\fin}(x_{\eta, \fin})d^{\times}t.
$$
The defining series-integral converges absolutely if
we take $c\in \RR$ such that $\Re(\bfs) = \underline{c} =(c)_{v \in S}$ and $(c+1)/4 > 1$.
By the expression of $\hat{{\bf \Psi}}_{\rm reg}(\gn|\a; -)$ in Proposition \ref{geom exp of reg Green} and
a similar computation as in the proof of \cite[Theorem 12.1]{TsuzukiSpec}, we can express the geometric side of $P_{\rm reg}^{\eta}(\hat{{\bf \Psi}}_{\rm reg}(\gn|\a; -))$ as follows.
\begin{thm}
\label{geometric side}
For any $* \in \{{\rm id}, {\rm u}, \bar{\rm u}, {\rm hyp}\}$, the integral
$\JJ_{*}^{\eta}(\b, \l; \a)$ converges absolutely and locally uniformly in $\{ \l \in \CC \ | \Re(\l) > 1\}$. The function $\l \mapsto \JJ_{*}^{\eta}(\b, \l; \a)$ is analytically continued to a meromorphic function on $\{\l \in \CC \ | \Re(\l)>-1 \}$.
Moreover, the constant term ${\rm CT}_{\l = 0}\JJ_{*}^{\eta}(\b, \l; \a)$ is equal to
$\JJ_{*}^{\eta}(\gn | \a)\b(0)$, where
$$\JJ_{\rm id}^{\eta}(\gn| \a) = 0,$$

$$\JJ_{\rm u}^{\eta}(\gn| \a) = (1+ \d (\gn = \go_{F}))D_{F}^{1/2}\Gcal(\eta)
\int_{\LL_{S}(\bfc)}\Upsilon_{S}^{\eta}(\bfs)\gC_{S, \go_{F}}^{\eta}(\bfs) \a(\bfs)d\mu_{S}(\bfs),$$

$$\JJ_{\bar{\rm u}}^{\eta}(\gn| \a) = (1+ \d (\gn = \go_{F}))D_{F}^{1/2}\Gcal(\eta)
\int_{\LL_{S}(\bfc)}\Upsilon_{S}^{\eta}(\bfs)\gC_{S, \gn}^{\eta}(\bfs) \a(\bfs)d\mu_{S}(\bfs)$$
and
$$\JJ_{\rm hyp}^{\eta}(\gn| \a) = \left(\frac{1}{2\pi i} \right)^{\#S}\int_{\LL_{S}(\underline{c})}\mathfrak{K}_{\eta}(\gn | \bfs)\a(\bfs)d\mu_{S}(\bfs).$$
In particular, we have
$$P_{\rm reg}^{\eta}(\hat{{\bf \Psi}}_{\rm reg}(\gn|\a; -)) =
\JJ_{\rm u}^{\eta}(\gn| \a) + \JJ_{\bar{\rm u}}^{\eta}(\gn| \a) + \JJ_{\rm hyp}^{\eta}(\gn| \a).$$
\end{thm}

\section{Proofs of main theorems}
\label{Proofs of main theorems}
Fix a character $\eta$ of $F^{\times} \backslash \AA^{\times}$
so that $\eta^{2}={\bf1}$ and $\eta_{v}(-1)=1$ for any $v \in \Sigma_{\infty}$.
Let $S$ be a finite subset of $\Sigma_{F}$ such that $S \supset \Sigma_{\infty}$ and $S_{\fin} \cap S(\gf_{\eta}) = \emptyset$.
Let $J_{S, \eta}'$ be the set of all ideals $\gn$ of $\go_{F}$ such that
$S(\gn) \cap (S \cup S(\gf_{\eta})) = \emptyset$ and $\tilde{\eta}(\gn)=1$.
By Theorems \ref{spectral side} and
 \ref{geometric side}, we obtain the relative trace formula
$$C(\gn, S)\{ \II_{\rm cus}^{\eta}(\gn|\a) + \II_{\rm eis}^{\eta}(\gn|\a) + \DD^{\eta}(\gn|\a) \} = \JJ_{\rm u}^{\eta}(\gn| \a) + \JJ_{\bar{\rm u}}^{\eta}(\gn| \a) + \JJ_{\rm hyp}^{\eta}(\gn|\a)$$
for any $\a \in \Acal_{S}$ and $\gn \in J_{S, \eta}'$.
The following estimate of $\JJ_{\rm hyp}^{\eta}(\gn|\a)$ is given by the same argument in the proof of \cite[Lemma 12.9]{TsuzukiSpec} with a minor modification.
\begin{lem}\label{lem:estimation of hyp}
For any $\a \in \Acal_{S}$ and $q>0$, we have
$\JJ_{\rm hyp}^{\eta}(\gn|\a) \ll {\rm N}(\gn)^{-q}$
with the implied constant independent of $\gn \in J_{S, \eta}'$.
\end{lem}
\begin{lem}\label{lem:estimation of B}
For any $\e > 0$, we have
$$|B_{\chi, \rho}^{\eta}(1/2, \nu)| \ll {\rm N}(\gf_{\chi})^{-1/2-\e} {\rm N}(\gn)^{\e}, \hspace{5mm}  \nu \in i \RR, \ \rho \in \Lambda_{\chi}(\gn), \ \chi \in \Xi(\gn)$$
with the implied constant independent of $\gn \in J_{S, \eta}'$.
\end{lem}
\Proof
Assume $\nu \in i \RR$. Then, the following estimate holds for any $\e >0$:
{\allowdisplaybreaks\begin{align*}
&|B_{\chi, \rho}^{\eta}(1/2, \nu)|\\
= & \ \prod_{k = 0}^{n}\prod_{v \in S_{k}(\rho)} |Q_{k, \chi_{v}}^{(\nu)}(\eta_{v}, 1)||L(1+\nu, \chi_{v}^{2})| \prod_{v \in U_{1}(\rho)}(1+q_{v}^{-1}) \\
& \times \prod_{k = 2}^{n}\prod_{v \in U_{k}(\rho)}\left(\frac{q_{v}+1}{q_{v}-1}\right)^{1/2} \prod_{k = 0}^{n}\prod_{v \in R_{k}(\rho)}q_{v}^{d_{v}/2}(1-q_{v}^{-1})^{1/2}|\overline{\Gcal(\chi_{v})}| \\
\ll & \prod_{v \in U_{1}(\rho)}(1+q_{v}^{-1})\left(1 + \frac{2}{q_{v}^{1/2} + q_{v}^{-1/2}}\right)\frac{1}{1 - q_{v}^{-1}}
\prod_{k = 2}^{n}\prod_{v \in U_{k}(\rho)}\left(\frac{q_{v}+1}{q_{v}-1}\right)^{1/2}
q_{v}^{-1}(q_{v}^{1/2} + 1)^{2}\frac{1}{1-q_{v}^{-1}} \\
& \times
\prod_{k = 0}^{n}\prod_{v \in R_{k}(\rho)}q_{v}^{d_{v}/2}(1-q_{v}^{-1})^{1/2}
\frac{q_{v}^{-f(\chi_{v})/2}q_{v}^{-d_{v}/2}}{1-q_{v}^{-1}} \frac{1}{1-q_{v}^{-1}} \\
\ll & \ {\rm N}(\gf_{\chi})^{-1/2 + \e} {\rm N}(\gn \gf_{\chi}^{-2})^{\e}.
\end{align*}
}This completes the proof.
\QED
Note that $[\bfK_{\fin }:\bfK_{0}(\gn)] = {\rm N}(\gn)\prod_{v \in S(\gn)}(1 + q_{v}^{-1})$ holds by an easy computation.
\begin{lem}\label{lem:estimation of eis}
For any $\a \in \Acal_{S}$, there exists $\d >0$ such that $|C(\gn, S)\II_{\rm eis}^{\eta}(\gn|\a)| \ll {\rm N}(\gn)^{-\d}$
with the implied constant independent of $\gn \in J_{S, \eta}'$.
\end{lem}
\Proof
We recall that for any $\e>0$, the estimate $|L_{\fin}(1+\nu, \chi^{2})|^{-1} \ll \mathfrak{q}(\chi^{2}|\cdot|_{\AA}^{\nu})^{\e}, \ \nu \in i \RR$ holds
with the implied constant independent of $\chi \in \Xi(\gn)$ and $\gn$.
This was given in the proof of Lemma \ref{lem: estimation of spectral decomposition}.
Let $\theta$ be a real number such that
$|L_{\fin}(1/2+i t, \chi)| \ll \gq(\chi |\cdot|_{\AA}^{i t} )^{1/4+\theta}, \ t \in \RR$
uniformly for any $\chi \in \Xi(\gn)$ and $\gn$.
We can take such $\theta$ so that $-1/4<\theta<0$ by \cite{Michel-Venkatesh}.
Thus, by the aid of Lemma \ref{lem:estimation of B} and Stirling's formula,
the explicit description of $P_{\rm reg}^{\eta}(E_{\chi, \rho}(\nu, -))$ in Proposition \ref{prop:Eisen P=L} gives us
the estimate
{\allowdisplaybreaks\begin{align*}
|P_{\rm reg}^{\eta}(E_{\chi, \rho}(\nu, -))| 
\ll & \ {\rm N}(\gf_{\chi})^{1/2} {\rm N}(\gf_{\chi})^{-1/2-\e} {\rm N}(\gn)^{\e}
 ({\rm N}(\gf_{\chi})^{1/4+\theta})^{2} {\rm N}(\gf_{\chi})^{\e}\prod_{v\in \Sigma_{\infty}} (1+|\nu+2ib(\chi_{v})|)^{1/2+2\theta +\e}\\
= & \ {\rm N}(\gf_{\chi})^{1/2+2\theta}{\rm N}(\gn)^{\e}\prod_{v\in \Sigma_{\infty}}(1+|\nu + 2ib(\chi_{v})|)^{1/2+2\theta +\e} \\
\ll & \ {\rm N}(\gn)^{1/4+\theta + \e}\prod_{v\in \Sigma_{\infty}}(1+|\nu+ 2ib(\chi_{v})|)^{1/2+2\theta +\e}
\end{align*}
}for any $\e>0$, where the implied constant is independent of $\nu \in i\RR, \chi \in \Xi(\gn)$ and $\gn \in J_{S, \eta}'$.
With the aid of Lemma \ref{esti of X(n)}, we have
{\allowdisplaybreaks\begin{align*}
|C(\gn, S)\II_{\rm eis}^{\eta}(\gn | \a)|
\ll & \ [\bfK_{\fin}: \bfK_{0}(\gn)]^{-1} \sum_{\chi \in \Xi(\gn)}\sum_{\rho \in \Lambda_{\chi}(\gn)}
\int_{i \RR}
|P_{\rm reg}^{\bf 1}(E_{\chi^{-1}, \rho}(-\nu, -))||P_{\rm reg}^{\eta}(E_{\chi, \rho}(\nu, -))|
|\tilde{\a}_{\chi}(\nu )||d \nu| \\
\ll & \ {\rm N}(\gn)^{-1} \hspace{-1mm}
\sum_{\chi \in \Xi(\gn)}
\bigg(\sum_{\ga | \gn}1\bigg)
\int_{y \in \RR}
{\rm N}(\gn)^{1/2+2\theta + 2\e}
\{ \prod_{v\in \Sigma_{\infty}} (1+|y+ 2b(\chi_{v})|)^{1+4\theta +2\e} \}
|\tilde{\a}_{\chi}(i y)|d y \\
\ll & \ {\rm N}(\gn )^{-1/2+2\theta+3\e} \hspace{-3mm}
\sum_{\chi \in \Xi_{\rm ker}(\mathfrak{\gn})}\sum_{b \in L_{0}}
\int_{y \in \RR}\{\prod_{v\in \Sigma_{\infty}}(1+|y+ 2b_{v}|)^{1+4\theta +2\e} \}
|\a((iy + 2ib_{v})_{v \in \Sigma_{\infty}})|d y \\
\ll&\ {\rm N}(\gn )^{2\theta +4\e}
\int_{y \in \RR^{d_{F}}}(1+||y||^{2})^{1+4\theta +2\e} 
|\a(iy)|d y.
\end{align*}
}Note $\sum_{\ga | \gn}1 \ll {\rm N}(\gn)^{\e}$.
Since we can take $\e>0$ so that $2\theta +4\e<0$,
we obtain the assertion.
\QED

\begin{lem}\label{lem:estimation of D}
For any $\e>0$ and $\a \in \Acal_{S}$, we have
$|C(\gn, S)\DD^{\eta}(\gn|\a)| \ll {\rm N}(\gn)^{-1+\e}$
with the implied constant independent of $\gn \in J_{S, \eta}'$.
\end{lem}
\Proof
This follows immediately from Lemma \ref{lem:estimations of Y and Z}.
\QED

For $\gn \in J_{S, \eta}'$, we set
$\langle \l_{S}^{\eta}(\gn), f \rangle = 2D_{F}^{1/2}\Gcal(\eta)^{-1}[\bfK_{\fin}:\bfK_{0}(\gn)]^{-1}
\sum_{\pi \in \Pi_{\rm cus}(\gn)}\PP^{\eta}(\pi, \bfK_{0}(\gn))f(\nu_{\pi, S})$
for any $f \in C_{c}(\mathfrak{X}_{S}^{0+})$.
The measure is extended to a measure on the Schwartz space ${\Scal}(\mathfrak{X}_{S}^{0+})$ (cf. \cite[Lemma 13.16]{TsuzukiSpec}).
Combining Lemmas \ref{lem:positivity of P}, \ref{Adjoint L}, \ref{lem:estimation of hyp},
\ref{lem:estimation of eis} and \ref{lem:estimation of D} with the argument in \cite[Lemma 13.18]{TsuzukiSpec},
we obtain the following theorem.
\begin{thm}\label{submain thm}
For a fixed $\a \in \Acal_{S}$,
there exists $\d>0$ such that
 for any infinite subset $\Lambda \subset J_{S, \eta}'$,
we have
$$\langle \l_{S}^{\eta}(\gn), \a \rangle =
\sum_{ \pi \in \Pi_{\rm cus}(\gn)} \frac{[\bfK_{\fin}: \bfK_{0}(\gf_{\pi})]}{{\rm N}(\gf_{\pi})[\bfK_{\fin}: \bfK_{0}(\gn)]}w_{\gn}^{\eta}(\pi)
\frac{L(1/2, \pi)L(1/2, \pi \otimes \eta)}{ L^{S_{\pi}}(1, \pi; {\rm Ad})}\a(\nu_{\pi, S})
=
\langle \l_{S}^{\eta}, \a \rangle + \Ocal({\rm N}(\gn)^{-\d})$$
as ${\rm N}(\gn) \rightarrow \infty$ in $\gn \in \Lambda$.
\end{thm}

We show the proof of Theorem \ref{thm:asymptotic of central values}.
For $\gn \in J_{S, \eta}$, let
$\gn = \prod_{k=1}^{s}\gp_{k}^{a_{k}}\prod_{k=s+1}^{s+l}\gp_{k}$ with $a_{k} \ge 2$
be a prime ideal decomposition of $\gn$.
For $\pi \in \Pi_{\rm cus}(\gn)$ with $\gf_{\pi} = \prod_{k=1}^{s}\gp_{k}^{b_{k}}\prod_{k=s+1}^{s+l}\gp_{k}^{\e_{k}}$,
Lemma \ref{lem:positivity of P} gives us
$$w_{\gn}^{\eta}(\pi) = \d(\ (\e_{s+k})_{k} \in \{ 1 \}^{l},\ (a_{k}-b_{k})_{k}\in (2\NN_{0})^{s} \ )
\prod_{k=1}^{s}\left(\frac{ {\rm N}(\gp_{k})+1 }{ {\rm N}(\gp_{k})-1 }\right)^{\d(b_{k}=0)}.$$
Hence, by setting
$L(\pi)=\frac{L(1/2, \pi)L(1/2, \pi \otimes \eta)}{L^{S_{\pi}}(1, \pi, {\rm Ad})}\a(\nu_{\pi, S})$
for a fixed $\a \in \Acal_{S}$, we obtain
\begin{align*}
& \frac{1}{{\rm N}(\gn)}\sum_{\pi \in \Pi_{\rm cus}^{*}(\gn)}
\frac{L(1/2, \pi)L(1/2, \pi \otimes \eta)}{L^{S_{\pi}}(1, \pi, {\rm Ad})}\a(\nu_{\pi, S}) \\
= & \left(\sum_{\pi \in \Pi_{\rm cus}(\gn)}
+\sum_{j=1}^{s+l}(-1)^{j}\sum_{1 \le i_{1} < \cdots < i_{j} \le s+l}
\sum_{\pi \in \Pi_{\rm cus}(\gn \prod_{k=1}^{j} \gp_{i_{k}}^{-1})} \right)
\frac{[\bfK_{\fin}: \bfK_{0}(\gf_{\pi})]}{{\rm N}(\gf_{\pi})[\bfK_{\fin}: \bfK_{0}(\gn)]}
w_{\gn }^{\eta}(\pi)L(\pi) \\
= & \left(\sum_{\pi \in \Pi_{\rm cus}(\gn)}
+\sum_{j=1}^{s}(-1)^{j}\sum_{1 \le i_{1} < \cdots < i_{j} \le s}
\sum_{\pi \in \Pi_{\rm cus}(\gn \prod_{k=1}^{j} \gp_{i_{k}}^{-2})} \right)
\frac{[\bfK_{\fin}: \bfK_{0}(\gf_{\pi})]}{{\rm N}(\gf_{\pi})[\bfK_{\fin}: \bfK_{0}(\gn)]}
w_{\gn }^{\eta}(\pi)L(\pi) \\
= & \sum_{\pi \in \Pi_{\rm cus}(\gn)}
\frac{[\bfK_{\fin}: \bfK_{0}(\gf_{\pi})]}{{\rm N}(\gf_{\pi})[\bfK_{\fin}: \bfK_{0}(\gn)]}w_{\gn }^{\eta}(\pi)L(\pi) 
+\sum_{j=1}^{s}\sum_{1 \le i_{1} < \cdots < i_{j} \le s}
(-1)^{j}
\frac{[\bfK_{\fin}: \bfK_{0}(\gn \prod_{k=1}^{j} \gp_{i_{k}}^{-2})]}{[\bfK_{\fin}: \bfK_{0}(\gn)]} \\
& \times  \sum_{\pi \in \Pi_{\rm cus}(\gn \prod_{k=1}^{j} \gp_{i_{k}}^{-2})}
\frac{[\bfK_{\fin}: \bfK_{0}(\gf_{\pi})]}{{\rm N}(\gf_{\pi})[\bfK_{\fin}: \bfK_{0}(\gn\prod_{k=1}^{j} \gp_{i_{k}}^{-2})]}
\bigg\{
\prod_{v \in S_{2}(\gn) \cap S(\prod_{k=1}^{j}\gp_{i_{k}})}
\frac{q_{v}+1}{q_{v}-1}
\bigg\}
w_{\gn \prod_{k=1}^{j} \gp_{i_{k}}^{-2} }^{\eta}(\pi)L(\pi) \\
= & \langle \l_{S}^{\eta}, \a \rangle
+\Ocal({\rm N}(\gn)^{-\d})
+ \sum_{j=1}^{s}(-1)^{j}\sum_{1 \le i_{1} < \cdots < i_{j} \le s}
\frac{ {\rm N}(\gn\prod_{k=1}^{j}\gp_{i_{k}}^{-2}) \prod_{v \in S(\gn \prod_{k=1}^{j}\gp_{i_{k}}^{-2})}(1+q_{v}^{-1}) }
{{\rm N}(\gn) \prod_{v \in S(\gn)}(1+q_{v}^{-1}) } \\
& \times \bigg\{ \prod_{v \in S_{2}(\gn) \cap S(\prod_{k=1}^{j}\gp_{i_{k}})}
\frac{q_{v}+1}{q_{v}-1} \bigg\}
\bigg\{ \langle \l_{S}^{\eta}, \a \rangle
+\Ocal({\rm N}(\gn\prod_{k=1}^{j}\gp_{i_{k}}^{-2})^{-\d})\bigg\} \\
= & \left( 1 + \sum_{j=1}^{s}(-1)^{j}\sum_{1 \le i_{1} < \cdots < i_{j} \le s}
\frac{\prod_{v \in S_{2}(\gn) \cap S(\prod_{k=1}^{j}\gp_{i_{k}})}(1-q_{v}^{-1})^{-1}}
{\prod_{k=1}^{j}{\rm N}(\gp_{i_{k}})^{2} } \right)
\langle \l_{S}^{\eta}, \a \rangle \\
& + \Ocal \left(
{\rm N}(\gn)^{-\delta} \left(1 + \sum_{j=1}^{s}\sum_{1 \le i_{1} < \cdots < i_{j} \le s}
\frac{\prod_{v \in S_{2}(\gn) \cap S(\prod_{k=1}^{j}\gp_{i_{k}})}(1-q_{v}^{-1})^{-1}}
{\prod_{k=1}^{j}{\rm N}(\gp_{i_{k}})^{2-2\delta} } \right)\right) \\
= & \left( \prod_{v \in S_{2}(\gn)}\{1-(1-q_{v}^{-1})^{-1}q_{v}^{-2}\}\prod_{v \in S(\gn) - (S_{1}(\gn)\cup S_{2}(\gn))}(1-q_{v}^{-2}) \right) \ \langle \l_{S}^{\eta}, \a \rangle
+\Ocal({\rm N}(\gn)^{-\d}).
\end{align*}
Here we note Theorem \ref{submain thm} and
an explicit formula of $w_{\gn}^{\eta}(\pi)$.
This completes the proof of Theorem \ref{thm:asymptotic of central values}
by extending the assertion for $\a \in \Acal_{S}$ to that for $f \in \Scal(\frak{X}_{S}^{0+})$ with the aid of
the proof of \cite[Theorem 13.17]{TsuzukiSpec}.

We show the proof of Theorem \ref{thm:nonvanishing}
which is same as in \cite[Corollary 1.2]{TsuzukiSpec}.
Let $\JJ$ be the set of all $(\nu_{v})_{v \in S} \in \mathfrak{X}_{S}^{0}$ such that
$(1-\nu_{v}^{2})/4 \in J_{v}$ for any $v \in \Sigma_{\infty}$
and
$q_{v}^{-\nu_{v}/2}+q_{v}^{\nu_{v}/2} \in J_{v}$ for any $v \in S_{\fin}$.
Put
$$I(\gn; \JJ) = \frac{1}{{\rm N}(\gn)}
\sum_{\begin{subarray}{c}
\pi \in \Pi_{\rm cus}^{*}(\gn) \\
\nu_{\pi, S} \in \JJ
\end{subarray}}
\frac{L(1/2, \pi)L(1/2, \pi \otimes \eta)}{L^{S_{\pi}}(1, \pi, {\rm Ad})}.$$
By Theorem \ref{thm:asymptotic of central values} and $\vol(\JJ, \l_{S}^{\eta}) > 0$,
for any $M>0$, there exists $\gn \in \Lambda$ with ${\rm N}(\gn) > M$ such that
$$|I(\gn; \JJ) - C(\gn) \vol(\JJ, \l_{S}^{\eta}) | <
2^{-1} \{ \prod_{v \in \Sigma_{\fin}}\{1-(q_{v}^{2}-q_{v})^{-1}\} \} \ \zeta_{F}(2)^{-1} 
\vol(\JJ, \l_{S}^{\eta}).$$
Therefore,
$I(\gn; \JJ) > 2^{-1} \{ \prod_{v \in \Sigma_{\fin}}\{1-(q_{v}^{2}-q_{v})^{-1}\} \} \ \zeta_{F}(2)^{-1}\vol(\JJ, \l_{S}^{\eta})>0$
holds by virtue of
$0< \{ \prod_{v \in \Sigma_{\fin}}\{1-(q_{v}^{2}-q_{v})^{-1}\} \} \ \zeta_{F}(2)^{-1}
< C(\gn).$
This completes the proof of Theorem \ref{thm:nonvanishing}.

We remark that Theorems \ref{thm:asymptotic of trivial} and \ref{thm:subconvexity} are proved in the same way as
\cite[Theorem 1.3, Corollary 1.4]{TsuzukiSpec} since we can generalize \cite[Theorem 14.1]{TsuzukiSpec} to the case of arbitrary level with a minor modification
by using the relative trace formula explained in \S \ref{Proofs of main theorems}.

\section*{Acknowledgements}
The author would like to thank Professor Takao Watanabe for useful comments.
He would also like to thank Professor Masao Tsuzuki for a lot of fruitful discussions and constructive comments,
and for informing him of the thesis \cite{Molteni}.
The author was supported by Grant-in-Aid for JSPS Fellows (25$\cdot$668).


\end{document}